\documentclass[10pt]{article}

\usepackage{fullpage}
\usepackage{amsthm,amsfonts,amssymb,mathrsfs,amsmath,latexsym,mathtools}
\usepackage{amsmath, amssymb, amsthm}
\usepackage[noadjust]{cite}
\usepackage{epsfig}
\usepackage{float} 
\usepackage{graphicx}
\usepackage{comment}
\usepackage{enumerate}
\usepackage[shortlabels]{enumitem}
\usepackage{enumitem}
\usepackage{caption}
\usepackage{subcaption} 
\usepackage{overpic}
\usepackage[super]{nth}
\usepackage{bbm}
\usepackage{xcolor}
\usepackage{hyperref}

\usepackage{tikz}

\numberwithin{equation}{section}

\newtheorem{thm}{Theorem}[section]
\newtheorem{prop}[thm]{Proposition}
\newtheorem{lem}[thm]{Lemma}
\newtheorem{cor}[thm]{Corollary}

\newtheorem{conjecture}[thm]{Conjecture}
\theoremstyle{definition}
\newtheorem{definition}[thm]{Definition}

\theoremstyle{remark}
\newtheorem{remark}[thm]{Remark}

\newcommand{\beq}{\begin{equation}}
\newcommand{\eeq}{\end{equation}}
\newcommand{\beqq}{\begin{equation*}}
\newcommand{\eeqq}{\end{equation*}}

\DeclareMathOperator{\re}{Re}

\DeclareMathOperator*{\res}{Res}

\newcommand{\al}{\alpha}

\newcommand{\De}{\Delta}
\newcommand{\om}{\omega}

\newcommand{\tn}[1]{\textnormal{#1}}
\def\fddto{\xrightarrow{\textit{f.d.d.}}}

\newcommand{\hide}[1]{}
\newcommand{\mr}{\mathrm}
\newcommand{\mc}{\mathcal}
\newcommand{\mb}{\mathbf}
\newcommand{\ms}{\mathsf}

\newcommand{\bs}{\boldsymbol} 
\renewcommand{\d}{\text{d}}

\def\abs#1{\left|#1\right|}

\newcommand{\N}{{\mathbb N}}
\newcommand{\Z}{{\mathbb Z}}
\newcommand{\B}{\mathbb{B}^{\mathrm{br}}}
\newcommand{\R}{{\mathbb R}}
\newcommand{\PP}{{\mathbb P}}
\newcommand{\C}{{\mathbb C}}

\renewcommand{\i}{\text{i}}

\newcommand{\LPP}{\mathcal{L}}
\newcommand{\lv}{\ell}

\newcommand{\prob}{\mathbb{P}}

\newcommand{\eqind}{\stackrel{d}{=}}
\newcommand{\eqinp}{\stackrel{p}{=}}
\newcommand{\eqinfdd}{\stackrel{f.d.d.}{=\!=\!=}}

\newcommand{\wh}[1]{\widehat{#1}}

\DeclareMathOperator{\TW}{TW}
\def \ii {\mathrm{i}}
\newcommand{\dd}{\mathrm{d}}

\def\law{{\rm Law}}
\def\fddto{\xrightarrow{\textit{f.d.d.}}}

\newcommand{\mcG}{\mr G} 

\newcommand{\veca}{\mathbf{a}}

\newcommand{\xd}{x}
\newcommand{\td}{t}
\newcommand{\hd}{h}
\newcommand{\LL}{L}
\newcommand{\Ql}{\ms P(\mathbf{\td}, \mb \xd, \mb \hd)}  
\DeclareMathOperator{\dist}{dist}

\newcommand{\At}{\mr A}
\newcommand{\Bt}{\mr B}
\newcommand{\lt}{-}		
\newcommand{\rt}{+}		
\newcommand{\lrt}{\pm}		
\newcommand{\fLi}{\ff_{L,i}}		
		
\newcommand{\ab}{\mathsf{\sigma}}	
\newcommand{\Exp}{\mathbb{E}}
\newcommand{\cd}{\mathsf{c}}	
\newcommand{\vecb}{\mb b}
\newcommand{\vecu}{\mb u}
\newcommand{\vecv}{\mb v}

\newcommand{\vecn}{\mb{n}}
\newcommand{\bn}{\vecn}
\newcommand{\bsigma}{{\bs \sigma}}
\newcommand{\btau}{{\bs \tau}}

\newcommand{\listset}{\mc S}
\newcommand{\listn}{\listset_{\bn}}
\newcommand{\lista}{\mc{S}}

\newcommand{\mcH}{\mr H}
\newcommand{\ing}{\mr I}
\DeclareMathOperator{\type}{\mathbf{type}} 

\newcommand{\GG}{\mr G}

\newcommand{\FF}{{\mathsf{F}}}
\newcommand{\bM}{\mathbf{M}}
\newcommand{\bN}{\mathbf{N}}
\newcommand{\bT}{\mathbf{T}}
\newcommand{\bz}{\mathbf{z}}
\newcommand{\bsxi}{\bs{\xi}}
\newcommand{\bseta}{\bs{\eta}}

\newcommand{\bone}{\mathbf{1}} 

\newcommand{\bu}{{\vecu}}
\newcommand{\bv}{{\vecv}}
\newcommand{\K}{\mathrm{K}}
\newcommand{\slope}{\mr m}
\newcommand{\ST}{\mathrm{S}}
\newcommand{\bfr}{\mathbf{r}}
\newcommand{\bfs}{\mathbf{s}}
\newcommand{\bw}{\mathbf{w}}

\newcommand{\bxo}{\bs{\xi}^1}
\newcommand{\beo}{\bs{\eta}^1}
\newcommand{\bxt}{\bs{\xi}^2}
\newcommand{\bet}{\bs{\eta}^2}
\newcommand{\bxr}{\bs{\xi}^3}
\newcommand{\ber}{\bs{\eta}^3}
\newcommand{\bxot}{\bs{\xi}^{12}}
\newcommand{\beot}{\bs{\eta}^{12}}
\newcommand{\bxtr}{\bs{\xi}^{23}}
\newcommand{\betr}{\bs{\eta}^{23}}
\newcommand{\bxa}{\bs{\xi}^{123}}
\newcommand{\bea}{\bs{\eta}^{123}}

\newcommand{\sd}{\mathsf{d}}
\newcommand{\dc}{\mathsf{A}}

\newcommand{\ff}{{\mathsf{f}}}

\newcommand{\pt}{\mathsf{q}}

\newcommand{\w}{\pt}

\newcommand{\rg}{\mathsf{R}} 
\newcommand{\rgo}{\rg_1}
\newcommand{\rgt}{\rg_2}
\newcommand{\rgth}{\rg_3}
\newcommand{\rgf}{\rg_4}
\newcommand{\rgfi}{\rg_5}
\newcommand{\rgs}{\rg_6}
\newcommand{\rgse}{\rg_7}

\newcommand{\equind}{\stackrel{d}{=}}
\newcommand{\Brb}{\mathbb{B}^{\mathrm{br}}}

\newcommand{\con}{\Sigma}
\newcommand{\HH}{\mr H}
\newcommand{\EE}{\mr E}

\newcommand{\cp}{z_c}
\newcommand{\J}{\mathsf{J}}
\newcommand{\cK}{\mathsf{Z}}

\newcommand{\QQ}{\mathsf{Q}} 
\newcommand{\DD}{\mathsf{D}}

\newcommand{\sa}{a} 

\newcommand{\mv}{h} 
\newcommand{\DDz}{\widetilde{D}} 
\newcommand{\Qt}{\mr Q}

\newcommand{\pp}{\mathsf{z}_c}
\newcommand{\pq}{\mathsf{w}_c}

\title{Conditional exponential directed last passage percolation under a one-point upper large deviation event}

\author{Jinho Baik\footnote{Department of Mathematics, University of Michigan,
Ann Arbor, MI, 48109, USA, \texttt{baik@umich.edu}} \and Dylan Cordaro\footnote{Department of Mathematics, University of Michigan,
Ann Arbor, MI, 48109, USA, \texttt{dcordaro@umich.edu}}
\and Tejaswi Tripathi\footnote{Department of Mathematics, University of Michigan,
Ann Arbor, MI, 48109, USA, \texttt{tejaswit@umich.edu}}}

\date{\today}

\begin{document}

\maketitle

\begin{abstract} 
Under typical scaling, the last passage time field of the directed last passage percolation model with exponential site distributions converges to the KPZ fixed point.
In this paper, we consider an atypical scenario in which the last passage time to a specific site is unusually large, and we explore how the last passage time field changes under this one-point upper large deviation event. We prove a conditional law of large numbers and compute the limiting fluctuations in certain regimes.
Our proofs rely on an analysis of explicit multi-point distributions. 
\end{abstract}



\section{Introduction and main results} \label{sec:intro}

\subsection{Introduction}

Let $\N$ denote the set of natural numbers and set $\N_0=\N\cup \{0\}$. 
For two points $\mb p= ( p_1,  p_2)$  and $\mb q= ( q_1,  q_2)$ in $\N^2$ satisfying $p_1\le q_1$ and $p_2\le  q_2$, 
an up/right path from $\mb p$ to $\mb q$  is a sequence 
$\pi= \{ \mb v_i \}_{i=1}^r$ with $r=q_1+q_2-p_1-p_2+1$, where 
$\mb v_1=\mb p$, $\mb v_r= \mb q$, and $\mb v_{i+1}-\mb v_{i}\in \{ (1,0), (0,1)\}$ for every $i$. 

\begin{definition}[Exponential LPP]\label{def:explpp}
Let $\{ \om_{\mb v} : \mb v \in \N^2\}$ be a collection of i.i.d.\@ exponential random variables of mean $1$. 
The last passage time from $\mb p$ to $\mb q$ is 
\begin{equation*}
    \LPP_{\mb p}(\mb q) = \max_{\pi: \mb p \to \mb q} E(\pi), 
    \qquad E(\pi)= \sum_{\mb v\in \pi} \om_{\mb v}, 
\end{equation*}
where the maximum is taken over all up/right paths from $\mb p$ to $\mb q$. 
When $\mb p=(1,1)$, we write $\LPP_{(1,1)}(\mb q)= \LPP(\mb q)$. 
We call the 2-dimensional random field $\LPP= \{\LPP(\mb q) : \mb q \in \N^2\}$ the exponential last passage percolation, or simply exponential LPP.  
For $(\alpha, \beta)\in \R_+^2$, we define
\beq \label{eq:LPPforreal}
    \LPP(\alpha, \beta)  = \LPP(\lceil\alpha\rceil, \lceil\beta\rceil)
\eeq
where $\lceil\alpha \rceil$ denotes the smallest integer that is greater than or equal to $\alpha$. 
\end{definition}

The exponential LPP is equivalent to several fundamental models in probability and statistical physics, including the corner growth model with wedge initial condition, the continuous-time totally asymmetric simple exclusion process (TASEP) with step initial condition, and the tandem queues model.
It is also an archetypal example of the KPZ universality class. Many results have been established for exponential LPP:

\begin{itemize}
\item The law of large numbers was obtained by \cite{Ros81}: For every $(x,y)\in \R_+^2$, 
\begin{equation} \label{eq:LLN}
    \lim_{N\to \infty} \frac{\mathcal{L}(xN, yN)}{N} = 
    \bar{\LPP}(x,y):=(\sqrt{x} + \sqrt{y})^2
\end{equation}
almost surely and also in probability. 
We also set $\bar{\LPP}_{(x',y')}(x,y)=(\sqrt{x-x'}+\sqrt{y-y'})^2$ so that $\bar{\LPP}(x,y)= \bar{\LPP}_{(0,0)}(x,y)$.

\item Johansson proved the convergence of the one-point distributions  in \cite{Joh00}: 
\beq \label{eq:onepoint}
    \lim_{N\to \infty} \frac{\LPP(xN, yN)-  \bar{\LPP}(x,y) N}{ (xy)^{-1/6}(\sqrt{x}+\sqrt{y})^{4/3}  N^{1/3}} \eqind \TW_2 
\eeq
for every $(x,y)\in \R_+^2$, where $\TW_2$ is the GUE Tracy-Widom distribution. 
Here, $\eqind$ denotes convergence in distribution. 

\item The two-dimensional random field also converges  \cite{MQR21}. In particular, for given $(x,y)\in \R^2_+$, 
\beq \label{eq:fdc}
    \lim_{N \to \infty} \frac{\LPP \left(t x N + s \frac{2 x^{2/3} y^{1/6}}{(\sqrt{x}+\sqrt{y})^{1/3}}  N^{2/3}, {t} y N - s \frac{2 x^{1/6} y^{2/3}}{(\sqrt{x}+\sqrt{y})^{1/3}}  N^{2/3}\right) - \bar{\LPP}(tx,ty)N}
    {(xy)^{-1/6}(\sqrt{x}+\sqrt{y})^{4/3}N^{1/3}} 
    \eqinfdd \mathsf{H}_{\tn{step}}(s,t )
\eeq
in the sense of finite-dimensional distributions for $(s,t) \in \R\times \R_+$,
where $\mathsf{H}_{\tn{step}}$ denotes the KPZ fixed point with narrow wedge initial condition. 

\item The upper large deviation result was obtained by  \cite{Joh00}: 
\beq	 \label{eq:upperld} 
    \lim_{N\to \infty} \frac1{N} \log \prob( \LPP(aN, bN) > \lv N) = - \J(\lv) \qquad \text{for $\lv> \bar{\LPP}(a,b)$}
\eeq
where\footnote{The paper \cite{Joh00} gives a variational formula for the rate function $\J(\lv)$. The variational formula can be solved to give the explicit formula. For example, see \cite[(46)]{LeDoussalMajumdarSchehr16} for the case when $a=b$.}
\begin{equation} \label{eq:large_deviation_fn}
    \J(\lv)= \sqrt{D} + a \log \bigg( \frac{\lv+a-b-\sqrt{D}}{\lv+a-b+\sqrt{D}}\bigg) + b \log \bigg( \frac{\lv-a+b-\sqrt{D}}{\lv-a+b+\sqrt{D}}\bigg) 
\end{equation}
with 
\begin{equation} \label{eq:Ddefn}
	D = \lv^2 - 2(a+b)\lv + (a-b)^2.
\end{equation}
The same paper also obtained the corresponding lower large deviation result. 
A hydrodynamic upper large deviation result was established in \cite{QT25} (in the TASEP setting), extending works of \cite{Jensen, Varadhan04}. 
For a comprehensive list of works on large deviations in KPZ universality class models, see \cite{DDV24}.
\end{itemize}

The goal of this paper is to investigate the behavior of exponential LPP conditioned on the event that the last passage time at a specific site is larger than expected.
Let $a,b>0$, and suppose that $\mc L(aN, bN) = \lv N$ for some $\lv > \bar{\LPP}(a,b)$. 
Given this conditioning, what does the field  $\{ \mc L(xaN, ybN) \}_{(x,y)\in \R_+^2}$ look like when $N$ is large? 
Which sites are affected by the conditioning at $(aN, bN)$? 
How does this conditioning influence the law of large numbers and the fluctuation behavior of the last passage time field?

This question has been considered recently for several models.
The conditional law of large numbers was obtained for the KPZ equation in \cite{LinTsai25} and for the directed landscape in \cite{DasTsai24}.
These works also considered conditioning at multiple points.
Conditional fluctuations were obtained for the KPZ fixed point in \cite{LW24, NZ24, LiuZhang25} and for the periodic KPZ fixed point in \cite{BL24}.
In this paper, we examine a different model—the exponential LPP—and prove conditional law of large numbers and conditional fluctuation results.
Regarding the fluctuation results, we extend the works of \cite{LW24, NZ24, LiuZhang25, BL24} to a regime that was not considered before.
Furthermore, we propose several conjectures concerning both the conditional law of large numbers and the limiting fluctuations in the full two-dimensional regime.


\medskip 

For comparison, consider the one-dimensional random field $\mathcal{S}= \{\mathcal{S}_n: n \in \N\}$, 
where $\mathcal{S}_n= X_1+\cdots+ X_n$ is the sum of i.i.d.\@ exponential random variables with mean $\mu=1$. 
For $\lv> 1$, it is straightforward to show that  
\beq 
    \law \left( \left\{ \frac{\mathcal{S}_{[tN]}-t\lv N}{\frac1{\sqrt{\Lambda''(\lv)}} N^{1/2}} \right\}_{t \in (0,1)} \bigg| \mc S_N =\ell N  \right) \fddto  \law \left( \{ \Brb(t)\}_{t \in (0,1)} \right)
\eeq
as $N \to \infty$, where $\Brb$ denotes a standard Brownian bridge, and $\Lambda(\alpha)$ is the rate function for the large deviations of the sum of independent exponential random variables. 
Explicitly,  
$\Lambda(\alpha)= \sup_\lambda (\alpha \lambda + \log (1-\lambda)) = \alpha -1- \log \alpha$ for $\alpha>0$, and thus $\frac1{\sqrt{\Lambda''(\lv)}} = \lv$. 
See, for example, \cite{Bor95} for the case where $X_i$ are general random variables. 
Note that $\mathcal{S}$ can be viewed as the  restriction of the exponential LPP on the first row, since $\mathcal{S}_n \equind \LPP(n, 1)$.

\subsection{Conditional law of large numbers}

Throughout this paper, the conditional probability $\prob(E \, | \, \mc L(a,b)=c)$ is understood as 
\begin{equation} \label{eq:conddefLPP}
    \prob(E \, | \,\mc L(a,b)=c) = \lim_{\epsilon\downarrow 0} \frac{\prob( E\cap \{ \mc L(a,b)\in (c-\epsilon, c+\epsilon)\})}{\prob(\mc L(a,b)\in (c-\epsilon, c+\epsilon) )}
    = \frac{\frac{\partial}{\partial c} \prob( E\cap \{ \mc L(a,b)\le c\})}{\frac{\partial}{\partial c}  \prob(\mc L(a,b)\le c )} . 
\end{equation}
The first result of this paper is a conditional law of large numbers.  
Compare the result with \eqref{eq:LLN}. 

\begin{thm}[Conditional Law of Large Numbers]\label{thm:LLN} 
Fix $\sa,b>0$ and $\lv> \bar{\LPP}(a,b)$. 
Let $D=\lv^2 - 2(a+b)\lv + (a-b)^2$ 
as in \eqref{eq:Ddefn} and define the function 
\beq \label{eq:mvxyshaded}
    \mv(x,y)=  \frac12 \left[ (\lv +a-b)x + (\lv -a+b)y -  |x-y| \sqrt{D} \right].
\eeq
Then, for every $\epsilon>0$, 
\beq \label{eq:CLLN}
    \lim_{N\to \infty} \prob\left[ \left| \frac{\LPP(xaN,ybN)}{N}- \mv(x,y)\right|>\epsilon \,\, \bigg| \, \LPP(aN,bN)=\lv N \right] =0 
\eeq
for $(x,y)\in (0,1)^2$ satisfying 
\beq \label{eq:critical_slope} 
    \frac{1}{\slope} < \frac{y}{x} < \slope  \qquad \text{where} \quad \slope := \frac{\lv-a-b+\sqrt{D}}{\lv-a-b-\sqrt{D}}.
\eeq
\end{thm}

The function $h(x,y)$ is a piecewise linear function; see Figure \ref{fig:hlevel} for its level curves. 

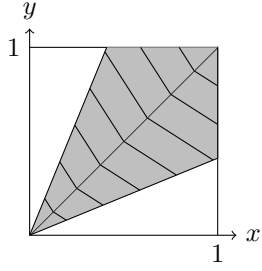
\begin{figure}[h]\centering
\begin{tikzpicture}[scale=2.5]
\draw[thin] (1,1) to (0,1) node[left] {$1$};
\draw[thin] (1,1) to (1,0) node[below] {$1$};
\fill[lightgray] (0, 0) to (1, 0.41) to (1,1) to (0,0); 
\draw[thin] (0, 0) to (1, 0.41) to (1,1) to (0,0); 
\draw[thin] (0.125, 0.125) to (0.19523, 0.41*0.19523); 
\draw[thin] (0.25, 0.25) to (0.39047, 0.41*0.39047); 
\draw[thin] (0.375, 0.375) to (0.58571, 0.41*0.58571);
\draw[thin] (0.5, 0.5) to (0.78095, 0.41*0.78095);
\draw[thin] (0.625, 0.625) to (0.97619, 0.41*0.97619);
\draw[thin] (0.75, 0.75) to (1, 0.59);
\draw[thin] (0.875, 0.875) to (1, 0.795);
\draw [->, thin][black] (0, 0) to (1.1,0) node[right] {$x$};
\draw [->, thin][black] (0,0) to (0, 1.1) node[above] {$y$};
\draw[thin][black] (0, 0) to (1,1); 
\fill[lightgray] (0, 0) to (0.41, 1) to (1,1) to (0,0); 
\draw[thin] (0, 0) to (0.41, 1); 
\draw[thin] (0.125, 0.125) to (0.41*0.19523, 0.19523); 
\draw[thin] (0.25, 0.25) to (0.41*0.39047, 0.39047); 
\draw[thin] (0.375, 0.375) to (0.41*0.58571, 0.58571);
\draw[thin] (0.5, 0.5) to (0.41*0.78095, 0.78095);
\draw[thin] (0.625, 0.625) to (0.41*0.97619, 0.97619);
\draw[thin] (0.75, 0.75) to (0.59, 1);
\draw[thin] (0.875, 0.875) to (0.795, 1);
\end{tikzpicture}
\caption{Level curves of $h(x,y)$}
\label{fig:hlevel}
\end{figure}


We propose the following conjecture for all points $(x,y)\in \R_+^2$.

\begin{conjecture}
\label{conj:LLN}
Define the regions
\beqq
    \Omega_1= \{ (x,y)\in (1, \infty)^2 : \frac{1}{\slope} < \frac{y-1}{x-1}< \slope\} 
    \quad\text{and}\quad
    \Omega_2= \{ (x,y)\in \R_+^2 : \frac{1}{\slope} < \frac{y}{x}< \slope\} .
\eeqq
Under the same assumptions as in Theorem \ref{thm:LLN}, we conjecture that \eqref{eq:CLLN} holds with 
\beq \label{eq:LLNconjv}
    	\mv(x,y)= \begin{cases}
    \lv + \bar{\LPP}_{(a,b)}(xa, yb)   \quad &\text{for }(x,y)\in \Omega_1\\
    \frac12 \left[ (\lv+a-b)x +(\lv-a+b)y - |x-y|\sqrt{D} \right] \quad &\text{for } (x,y)\in \Omega_2\setminus \Omega_1\\
   \bar{\LPP}(xa, yb) \quad &\text{for } (x,y)\in \R_+^2\setminus \Omega_2.
    \end{cases}
\eeq
\end{conjecture}

We provide a heuristic argument for the above conjecture in Subsection \ref{sec:llnconj}. See Figure \ref{fig:Omega12} for the picture of the regions $\Omega_1$ and $\Omega_2$. 
The conjecture suggests that conditioning on $\LPP(aN, bN)=\lv N$ does not affect the hydrodynamic limit of the last passage time for points in the region $\R_+^2\setminus \Omega_2$, and only has a trivial effect on points in $\Omega_1$. 
In contrast, the hydrodynamic limit of the conditional last passage time to points in $\Omega_2\setminus \Omega_1$ is conjectured to be a piecewise linear function. Theorem \ref{thm:LLN} establishes this part of the conjecture for $(x,y)\in (0,1)^2$.

\begin{figure}[h]\centering
\begin{tikzpicture}[scale=1.5]
\fill[lightgray] (0, 0) to (2, 0.82) to (2, 1.41) to (2,2) to (1.41, 2) to (0.82, 2) to (0,0); 
\fill[gray]  (1, 1) to (2, 1.41) to (2,2) to (1.41, 2) to (1,1); 
\draw[thin] (1.41, 2)  to (1, 1) to (2, 1.41); 
\draw[thin] (0.82, 2) to (0, 0) to (2, 0.82); 
\draw [->, thin][black] (0, 0) to (2.1,0)  node[right] {$x$};
\draw [->, thin][black] (0,0) to (0, 2.1) node[left] {$y$};
\node at (1.6, 1.6)  {$\Omega_1$};
\node at (0.8, 0.6)  {$\Omega_2$}; 
    \filldraw (1,1) circle (1pt);
       \node[below] at (1,0) {$1$};    
    \node[left] at (0,1) {$1$};    
\end{tikzpicture}
\caption{The dark gray region is $\Omega_1$. The union of the light gray region and the dark gray region is $\Omega_2$}
\label{fig:Omega12}
\end{figure}
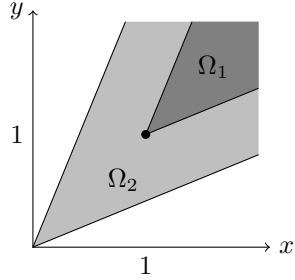

The level curves of the conjectured $h(x,y)$ are shown on the right panel of Figure \ref{fig:levelcurves}. 
The function $h(x,y)$ is continuous everywhere; however, it is $C^1$ on $\R^2_+ \setminus \{(t,t): 0 \le t \le 1\}$. 
At the boundary of $\Omega_2$, the level curves of $h(x,y)$ are tangential to the level curves of the unconditional limit $\bar{\LPP}(xa,yb)$. 
Similarly, at the boundary of $\Omega_1$, the level curves are tangential to the curves $\lv+\bar{\LPP}_{(a,b)}(xa,yb)$. 
The left panel displays the level curves of the unconditional limit, $\bar{\LPP}(xa,yb)$.

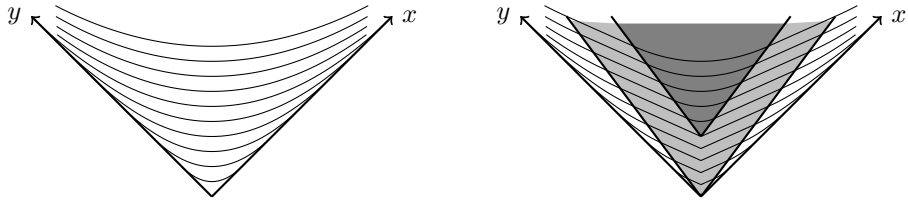
\begin{figure}[h] \centering
\begin{tikzpicture}[scale=0.8]
\draw[thick, ->] (0, 0) -- (3, 3) node[right] {$x$};
\draw[thick, ->] (0,0) -- (-3, 3) node[left] {$y$};

\foreach \c in {1, 2, 3, 4, 5, 6} {
\draw[domain=-\c/2 : \c/2, samples=100] plot (\x, {\x*\x/\c+\c/4});
} 
\foreach \c in {7, 8, 9, 10} {
\draw[domain=- 2.5-0.01*\c  : 2.5+0.01*\c, samples=100] plot (\x, {\x*\x/\c+\c/4});
} 
\end{tikzpicture}
\qquad 
\begin{tikzpicture}[scale=0.8]
\fill[gray] (0, 1) to (1.4, 1.34188034188*1.4+1) to (-1.4, 1.34188034188*1.4+1) to (0,1); 
\fill[lightgray] (0, 0) to (2.18, 1.34188034188*2.18) to (1.4, 1.34188034188*1.4+1) to (0,1) to (-1.4, 1.34188034188*1.4+1) to (-2.18, 1.34188034188*2.18) to (0,0); 

\draw[thick, ->] (0, 0) -- (3, 3) node[right] {$x$};
\draw[thick, ->] (0,0) -- (-3, 3) node[left] {$y$};

\draw[thick, domain=-2.23566878981:2.23566878981, samples=100] plot (\x, {1.34188034188*abs(\x)});  
\draw[thick, domain=-1.49044585987:1.49044585987, samples=100] plot (\x, {1.34188034188*abs(\x)+1});  


\foreach \c in {1, 2, 3, 4, 5} {
\draw[domain=-0.2235469249*\c: 0.22354692494*\c, samples=3] plot (\x, {0.4472135955*abs(\x)+\c/5}); 
} 

\foreach \c in {6, 7, 8, 9, 10} {
\draw[domain=-0.2235469249*\c: -0.22354692494*\c+1.1177346247, samples=100] plot (\x, {0.4472135955*abs(\x)+\c/5});
\draw[domain=0.22354692494*\c-1.1177346247: 0.22354692494*\c, samples=100] plot (\x, {0.4472135955*abs(\x)+\c/5});
\draw[domain=-0.22354692494*\c+1.1177346247:0.22354692494*\c-1.1177346247, samples=100] plot (\x, {\x*\x/(\c-5)+(\c-5)/4+1});
} 

\foreach \c in {1, 2, 3, 4, 5, 6} {
\draw[domain=-\c/2 : -0.22354692494*\c, samples=100] plot (\x, {\x*\x/\c+\c/4}); 
\draw[domain=0.22354692494*\c : \c/2, samples=100] plot (\x, {\x*\x/\c+\c/4});
} 
\foreach \c in {7, 8, 9, 10} {
\draw[domain=- 2.5-0.01*\c : -0.22354692494*\c, samples=100] plot (\x, {\x*\x/\c+\c/4});
\draw[domain=0.22354692494*\c : 2.5+0.01*\c, samples=100] plot (\x, {\x*\x/\c+\c/4});
} 
\end{tikzpicture}
\caption{Left: Level curves of $\bar{\LPP}(x,y)$ for the law of large numbers  \eqref{eq:LLN}, rotated by 45 degrees. Right: Level curves of the function $h(x,y)$ for the conjectured conditional law of large numbers \eqref{eq:LLNconjv}, rotated by 45 degrees, when $a=b=1$ and $\lv=5$. The gray area is $\Omega_1$; the light gray area denotes  $\Omega_2\setminus \Omega_1$.}
\label{fig:levelcurves}
\end{figure}

As mentioned before, the papers \cite{LinTsai25} and \cite{DasTsai24} considered the conditional law of large numbers for the KPZ equation and the directed landscape.
These papers state results for times before the conditioning time, which corresponds to the regime $x+y\le 2$ in the exponential LPP.\footnote{Li-Cheng Tsai informed us that the result of \cite{DasTsai24} can be extended to all times.}

\subsection{Conditional fluctuations}  \label{sec:fluctuations1}

We now consider fluctuations and present two results. The first pertains to points along the diagonal line, as shown in the left panel of Figure \ref{fig:diag}. 
For these points, convergence holds in the sense of finite-dimensional distributions. This result may be compared with \eqref{eq:fdc}.

\begin{thm}[Diagonal multi-point fluctuations] \label{thm:diagfluc} 
Fix $a,b > 0$ and $\lv> \bar{\LPP}(a,b)$. Let $D=\lv^2-2(a+b)\lv+(a-b)^2$. 
Define the positive real numbers   
\beq \label{eq:abdf}
    \ab= \frac{\sqrt{(a+b)\lv- (a-b)^2} D^{1/4}}{2\sqrt{ab}}
    \quad \text{and} \quad 
    \cd_\pm= \bigg(1 \pm \frac{(a-b)\sqrt{D}}{(a+b)\lv -(a-b)^2} \bigg)^{1/2}. 
\eeq
Then, 
\beq \begin{split} \label{eq:diagfluc}
    &\law \left( \left\{ \frac{ \LPP( t aN +  s \frac{a(\ell-a+b)\ab}{\ell\sqrt{D}}   N^{1/2}, t bN -  s \frac{b(\ell+a-b)\ab}{\ell\sqrt{D}}  N^{1/2})- t \lv N}{\ab N^{1/2}} \right\}_{(s,t)\in \R\times (0,1)}
    \, \bigg| \,  \LPP(aN, bN)= \lv N \right) \\
    &\fddto  \law \left( \{ \B_1(t) - |\B_2(t) - s| \}_{(s,t)\in \R\times (0,1)} \right)
\end{split} \eeq
as $N\to \infty$, where 
\beq \label{eq:defcd}
    \B_1(t)= \frac{\cd_+ \B_+(t)+ \cd_- \B_-(t)}{\sqrt{2}}  \quad\text{and}\quad 
    \B_2(t)= \frac{\cd_+ \B_+(t)- \cd_- \B_-(t)}{\sqrt{2}} 
\eeq
for two independent standard Brownian bridges $\B_+$ and $\B_-$. 
\end{thm}

We note that $\B_1$ and $\B_2$ are standard Brownian bridges with covariance 
\beq
    \Exp [ \B_1(t)\B_2(t)] = \frac{(a-b)\sqrt{D}}{(a+b)\lv -(a-b)^2} t(1-t), \qquad t\in (0,1). 
\eeq 
They are independent only when $a=b$. 

\begin{figure}[h]\centering
\begin{tikzpicture}[scale=2.5]
\draw[thin] (1, 0) to (1,1) to (0,1); 
\fill[lightgray] (0, 0) to (0.1, 0) to (1,0.9) to (1,1) to (0.9,1) to (0,0.1) to (0,0); 
\draw[thin] (0, 0) to (1, 1); 
\draw [->, thin][black] (0, 0) to (1.1,0); 
\draw [->, thin][black] (0,0) to (0, 1.1); 
\end{tikzpicture}
\qquad \qquad 
\begin{tikzpicture}[scale=2.5]
\draw[thin] (1, 0) to (1,1) to (0,1); 
\fill[lightgray] (0, 0) to (1, 0.41) to (1,1) to (0,0); 
\draw[thin] (0, 0) to (1, 0.41) to (1,1) to (0,0); 
\fill[lightgray] (0, 0) to (0.41, 1) to (1,1) to (0,0); 
\draw[thin] (0, 0) to (0.41, 1)  to (1,1) to (0,0); 
\draw [->, thin][black] (0, 0) to (1.1,0); 
\draw [->, thin][black] (0,0) to (0, 1.1);
\end{tikzpicture}
\caption{The left picture is related to Theorem \ref{thm:diagfluc}. The right picture is related to Theorem \ref{thm:offdiagflucsameside}.}
\label{fig:diag}
\end{figure}
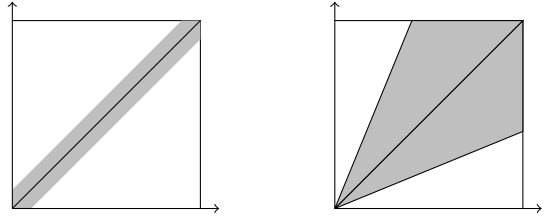

The second result concerns fluctuations at points off the diagonal line, as shown in the right panel of Figure~\ref{fig:diag}. However, for these points, we were only able to prove convergence for two-point distributions.

\begin{thm} [Off-diagonal two-point fluctuations] \label{thm:offdiagflucsameside} 
Fix $a,b > 0$ and $\lv> \bar{\LPP}(a,b)$. 
Let the function $h(x,y)$ and the positive numbers $\slope$, $\ab$, and $\cd_{\pm}$ be as defined in Theorems \ref{thm:LLN} and \ref{thm:diagfluc}. 
Then, for two distinct points $(x_1, y_1), (x_2, y_2)\in (0,1)^2$ and two real numbers $\mr r_1, \mr r_2\in \R$, 
\beqq
\begin{split}
    &\lim_{N \to \infty} \prob \left[\frac{\mc L(x_iaN, y_ibN) - \mv(x_i,y_i)N}{\sqrt{2} \ab   N^{1/2}} > \mr r_i, \ i=1, 2 \bigg| \LPP(aN, bN)= \lv N \right] \\ 
    &\qquad  \quad 
    = \begin{dcases}  \prob \left[ \cd_{+} \B\left(\frac{\slope y_i- x_i}{\slope -1} \right) > \mr r_i,  i=1,2 \right] 
    \qquad &\text{if $\frac{1}{\slope} < \frac{y_1}{x_1}, \frac{y_2}{x_2} < 1$,} \\
    \prob \left[ \cd_{-} \B\left(\frac{\slope x_i- y_i}{\slope -1} \right) > \mr r_i,  i=1,2 \right]
\qquad &\text{if $1 < \frac{y_1}{x_1}, \frac{y_2}{x_2} < \slope$,}
    \end{dcases}
\end{split} \eeqq
where $\B$ is a standard Brownian bridge, and $\slope$ is the constant defined in \eqref{eq:critical_slope}. 
\end{thm}

We expect that the results will hold for multi-point distributions as well. However, since the analysis becomes quite involved, in this paper we focus only on two-point distribution results, leaving the general case for future work.

\begin{figure}[h]\centering
\begin{tikzpicture}[scale=2.5]
\draw[thin] (1,1) to (0,1) node[left] {$1$};
\draw[thin] (1,1) to (1,0) node[below] {$1$};
\fill[lightgray] (0, 0) to (1, 0.41) to (1,1) to (0,0); 
\draw[thin] (0, 0) to (1, 0.41) to (1,1) to (0,0); 
\draw[thin] (0.17, 0.17) to (1, 0.51); 
\draw[thin] (0.338, 0.338) to (1, 0.61); 
\draw[thin] (0.508, 0.508) to (1, 0.71); 
\draw[thin] (0.678, 0.678) to (1, 0.81); 
\draw[thin] (0.847, 0.847) to (1, 0.91); 
\draw [->, thin][black] (0, 0) to (1.1,0) node[right] {$x$};
\draw [->, thin][black] (0,0) to (0, 1.1) node[above] {$y$};
\draw[thin][black] (0, 0) to (1,1);  
\end{tikzpicture}
\qquad\qquad\begin{tikzpicture}[scale=2.5]
\draw[thin] (1,1) to (0,1) node[left] {$1$};
\draw[thin] (1,1) to (1,0) node[below] {$1$}; 
\draw [->, thin][black] (0, 0) to (1.1,0) node[right] {$x$};
\draw [->, thin][black] (0,0) to (0, 1.1) node[above] {$y$};
\draw[thin][black] (0, 0) to (1,1); 
\fill[lightgray] (0, 0) to (0.41, 1) to (1,1) to (0,0); 
\draw[thin] (0, 0) to (0.41, 1); 
\draw[thin] (0.17, 0.17) to (0.510, 1); 
\draw[thin] (0.338, 0.338) to (0.609, 1); 
\draw[thin] (0.508, 0.508) to (0.71, 1); 
\draw[thin] (0.678, 0.678) to (0.81, 1); 
\draw[thin] (0.847, 0.847) to (0.91, 1); 
\end{tikzpicture}
\caption{Left: Level curves of $(x,y)\mapsto \slope y- x$ for $y<x$. Right: Level curves of $(x,y)\mapsto \slope x- y$ for $y>x$}
\label{fig:mlevel}
\end{figure}
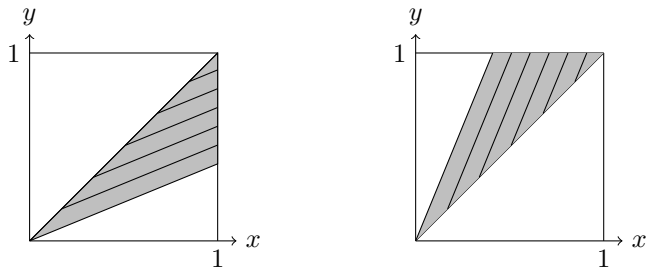

The above result shows that if two points in the gray region in the right panel of Figure \ref{fig:diag} lie on a level curve of the mapping $(x,y) \mapsto \slope y - x$ and both are below the diagonal line, then the corresponding limiting two-point distributions are identical. A similar statement holds for the mapping $(x,y) \mapsto \slope x - y$ for points above the diagonal line; see Figure \ref{fig:mlevel}.
We further conjecture that the distributions for points below the diagonal line and those for points above the diagonal line become independent. In other words, we conjecture that the Brownian bridges governing the fluctuations in the below-diagonal and above-diagonal regimes are independent; see \eqref{eq:flconb} below.

\medskip
For general points, we propose the following conjecture. A heuristic argument supporting this conjecture is provided in Subsection \ref{sec:fluctuationsconjecture}. 
Although items (a) and (c) below are stated only for one-point distributions, the extension of the conjecture to convergence to the KPZ fixed point—analogous to \eqref{eq:fdc}—is straightforward, and thus we omit it here.

\begin{conjecture} 
\label{conj:offdiagfl}
Let $\Omega_1$ and $\Omega_2$ be the regions defined in Conjecture \ref{conj:LLN}, and let $\mv(x,y)$ be the function defined in \eqref{eq:LLNconjv}. 
Under the same assumptions and notation as in Theorem \ref{thm:diagfluc} and \ref{thm:offdiagflucsameside}, and conditional on the event $\LPP(aN,bN)=\lv N$, we conjecture that the following results hold. 
\begin{enumerate}[(a)]
\item For each $(x,y)\in \Omega_1$, 
\beq \label{eq:flcona}
    \lim_{N\to \infty} \frac{\LPP(xaN, ybN)- h(x,y) N}{ (ab(x-1)(y-1))^{-1/6}(\sqrt{a(x-1)}+\sqrt{b(y-1)})^{4/3}  N^{1/3}} \eqind \TW_2.
\eeq
\item We expect that 
\beq \label{eq:flconb}
    \frac{\LPP(xaN, ybN)-  h(x,y) N}{ \sqrt{2} \ab   N^{1/2}} \fddto 
    \begin{dcases} \cd_{+} \B_{+} \left( \frac{\slope y- x}{\slope -1} \right) \quad &\text{for $(x,y)\in \Omega_2\setminus \overline{\Omega}_1$ satisfying $y<x$,} \\
    \cd_{-} \B_{-} \left( \frac{\slope x- y}{\slope -1} \right) \quad &\text{for $(x,y)\in \Omega_2\setminus \overline{\Omega}_1$ satisfying  $y>x$}
    \end{dcases}
\eeq
as $N\to \infty$, where $\B_{+}$ and $\B_-$ are independent standard Brownian bridges, and furthermore, they are the same ones appearing in \eqref{eq:defcd}. 

\item For each $(x,y)\in \R_+^2\setminus \overline{\Omega}_2$, 
\beq \label{eq:flconc}
    \lim_{N\to \infty} \frac{\LPP(xaN, ybN)- h(x,y) N}{ (abxy)^{-1/6}(\sqrt{ax}+\sqrt{by})^{4/3}  N^{1/3}} \eqind \TW_2. 
\eeq
\end{enumerate}
\end{conjecture}


\subsection{Comparison with the conditional KPZ fixed point} \label{sec:compareKPZ}

The KPZ fixed point under a one-point upper large deviation event was recently studied in \cite{LW24, NZ24, LiuZhang25}. 
Let $\mathsf{H}_{\tn{step}}(s,t)$ for $(s,t)\in \R\times \R_+$, denote the KPZ fixed point with the narrow wedge initial condition. 
In \cite[Remark 1.5]{LW24}, Liu and Wang proved that\footnote{We have rewritten the result of \cite{LW24} using the identity $\min\{u,v\}= \frac{u+v-|u-v|}{2}$, and adjusted the parameters to eliminate the factor  $\sqrt{2}$.} for every $(X,T)\in \R\times \R_+$, 
\beq \label{eq:KPZfpfl} \begin{split}
	&\law \left( \left\{ \frac{ \mathsf{H}_{\tn{step}}( tX+ s\frac{T^{3/4}}{2L^{1/4}}, tT) - tL }{T^{1/4} L^{1/4} } \right\}_{(s,t)\in \R\times (0,1)}
    \, \bigg| \,  \mathsf{H}_{\tn{step}}(X, T )= L \right) \\
 	&\qquad 
	\fddto  \law \left( \{ \B_1(t) - |\B_2(t) - s| \}_{(s,t)\in \R\times (0,1)} \right)
\end{split} \eeq
as $L\to \infty$, where the Brownian bridges $\B_1$ and $\B_2$ are independent. 
Theorem~\ref{thm:diagfluc} is similar to this result, but the Brownian bridges that appear in that theorem are generally not independent. It is intriguing why this dependence arises in the exponential LPP. While this dependence follows from explicit computation, we do not have a simple conceptual explanation for this phenomenon. 
A fluctuation result for points away from the straight line from $(0,0)$ to $(X,T)$, analogous to Theorem \ref{thm:offdiagflucsameside} of this paper, has not yet been obtained for the KPZ fixed point.

The fluctuations of the conditional KPZ fixed point were also studied near the conditioning time in \cite{LiuZhang25}, and after the conditioning time in \cite{NZ24}. 
In our upcoming paper, we study the fluctuations both near and beyond the conditioning time in Exponential LPP, and obtain, for example, a local limiting field around the conditioning point $(aN,bN)$.
A version of Theorem \ref{thm:diagfluc} was similarly established for the periodic KPZ fixed point in \cite{BL24}; in that context, the limiting distribution involves a Brownian bridge and a Brownian bridge on a circle, which are again independent.

The fluctuation scaling $L^{1/4}$ in \eqref{eq:KPZfpfl} is consistent with a formal KPZ limit of the result \eqref{eq:diagfluc}, as we now show. 
For simplicity, consider the case $a=b=1$. The typical KPZ behavior is that 
\beq
    \LPP \ (tN + 2^{2/3}sN^{2/3}, tN - 2^{2/3}s   N^{2/3} ) \approx 4tN + 2^{4/3} N^{1/3} \mathsf{H}_{\tn{step}}(s,t ) . 
\eeq 
If we consider the case when $\ell= 4 + 2^{4/3} LN^{-2/3}$ with $L\to \infty$ and $L=o(N^{2/3})$, and formally apply the above approximation, then the conditioning event $\LPP(N,N)=\ell N$ translates to $\mathsf{H}_{\tn{step}}(0,1 )\approx L$. 
On the other hand, a formal application of the above approximation to the quantity in \eqref{eq:diagfluc} becomes 
\beq \begin{split} 
    &\frac{ \LPP( t N +  s \frac{\sqrt{\ell}}{\sqrt{2}D^{1/4}}   N^{1/2}, t N -  s \frac{\sqrt{\ell}}{\sqrt{2}D^{1/4}}   N^{1/2})- t \lv N}{2^{-1/2}\sqrt{\ell}D^{1/4} N^{1/2}} 
    \approx \frac{ \mathsf{H}_{\tn{step}}( s \frac{\sqrt{\ell}}{2^{7/6}D^{1/4} N^{1/6}} , t) - \frac{t(\ell-4) N^{2/3}}{2^{4/3}} }{ 2^{-11/6} \sqrt{\ell}D^{1/4} N^{1/6} }, 
\end{split} \eeq
which is $\frac{ \mathsf{H}_{\tn{step}}( \frac{s}{2L^{1/4}}, t) - tL }{ L^{1/4} }$, 
since $D = \ell^2-4\ell \approx 2^{10/3}LN^{-2/3}$. This is exactly the term in \eqref{eq:KPZfpfl} when $(X,T)=(0,1)$.

\subsection{Method of proof and outline of the paper}

Theorem \ref{thm:LLN} follows from Theorems \ref{thm:diagfluc} and \ref{thm:offdiagflucsameside}; thus, we prove only these latter two theorems.
Our approach is based on the analysis of explicit multi-point distribution formulas for the exponential LPP. The multi-point distributions in so-called space-like directions were computed in the mid-2000s in \cite{Joh03, BFS08}. The distributions for general points, including those in time-like directions, were obtained more recently by Liu \cite{Liu22a}.

The proof of Theorem \ref{thm:diagfluc} is similar to that of \cite{LW24} for the KPZ fixed point, and we have adapted it for the exponential LPP.
However, the proof of Theorem \ref{thm:offdiagflucsameside} requires substantially more effort and constitutes the most technical part of this paper.

The explicit multi-point distribution formula from \cite{Liu22a} involves an integral of a Fredholm determinant. In random matrix theory and KPZ models, upper large deviation and upper tail limits are often readily obtained from Fredholm determinants, as the operator becomes small and the determinant can be approximated by its trace using the method of steepest descent. In our case, however, the operator acts on nested contours. 
For Theorem \ref{thm:offdiagflucsameside}, the critical points relevant to the steepest descent method are ordered such that the contours cannot be deformed appropriately without crossing the poles of the kernel. As a consequence, we must keep track of all residue contributions, which quickly becomes challenging.

Due to these complexities, we restrict our analysis to two-point distribution results and leave multi-point distribution considerations for future work. Even for two-point distributions, the locations of the critical points depend on the relative positions of the points, requiring the consideration of seven distinct regimes. In contrast, the proof of Theorem \ref{thm:diagfluc} is simpler, since the critical points are fixed and the poles of the kernel do not need to be considered.

Although we do not use the Fredholm determinant formula directly in our analysis, instead relying on its series expansion, we still encounter the same underlying challenges.

\medskip

This paper is organized as follows.
In Section \ref{sec:conjectures}, we provide heuristic reasoning behind Conjectures \ref{conj:LLN} and \ref{conj:offdiagfl}, and state an additional conjecture regarding conditional geodesics.
In Section \ref{sec:finitemulti}, we present explicit formulas for the multi-time conditional distributions.
Section \ref{sec:threelemmas} contains two miscellaneous lemmas used throughout the paper.
In Section \ref{sec:asfunction}, we examine functions that play a central role in our analysis and derive their limits and bounds.
The proof of Theorem \ref{thm:diagfluc} is given in Section \ref{sec:proofofdiag}.
Finally, Theorem \ref{thm:offdiagflucsameside}, which constitutes the most technical part of the paper, is proved in Section \ref{sec:proofoffdiagflucsameside}. 


\subsection*{Acknowledgments}

The work of Baik was supported in part by NSF grant DMS-2246790 and by the Simons Fellows in Mathematics program. 
We would like to thank Milind Hegde, Zhipeng Liu and Bálint Virág for useful discussions, and the anonymous referees for several helpful suggestions.

\section{Heuristics} \label{sec:conjectures}

We give a heuristic argument for Conjectures \ref{conj:LLN} and \ref{conj:offdiagfl}. 
We also discuss a conjecture on the conditional geodesics. 

\subsection{Heuristic argument for Conjecture \ref{conj:LLN}} \label{sec:llnconj}

Using the dynamic programming recursion, or equivalently the Bellman equation, we have
\begin{equation} \label{eq:recursion}
\LPP(xaN,ybN)=\max_{0 \leq t \leq \min\{x,y,1\}}\left\{\LPP(taN,tbN)+\LPP_{taN,tbN}(xaN,ybN)\right\}.
\end{equation}
The condition $t \leq \min\{x,y,1\}$ ensures that the point $(xaN,ybN)$ lies in the up-right direction from the point $(taN,tbN)$.

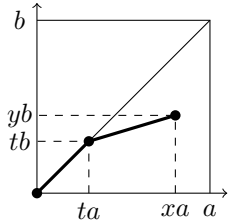
\begin{figure}[h] \centering
\begin{tikzpicture}[scale=2.3]
\draw[thin] (1, 0) to (1,1) to (0,1); 
\draw[thin] (0, 0) to (1,1); 
\draw[very thick] (0, 0) to (0.3,0.3) to (0.8,0.45); 
\fill (0,0) circle(0.03);
\fill (0.3,0.3) circle(0.03);
\fill (0.8,0.45) circle(0.03);
\draw[dashed, thin] (0.3, 0) to (0.3, 0.3) to (0,0.3); 
\draw[dashed, thin] (0.8, 0) to (0.8, 0.45) to (0,0.45); 
\node at (0.3, -0.1) {$ta$};
\node at (0.8, -0.1) {$xa$};
\node at (1, -0.1) {$a$};
\node at (-0.1, 0.3) {$tb$};
\node at (-0.1, 0.45) {$yb$};
\node at (-0.1, 1) {$b$};
\draw [->, thin][black] (0, 0) to (1.1,0);
\draw [->, thin][black] (0,0) to (0, 1.1);
\end{tikzpicture}
\caption{Conjectural maximizing path}
\label{fig:pathmax}
\end{figure}

Now suppose that $\LPP(aN,bN)=\ell N$ for some $\ell>\bar{\LPP}(a,b)$. 
For $\LPP(aN,bN)$ to be large, it suffices that there exists a single path $\pi$ ending at $(aN,bN)$ with a large value of $E(\pi)=\sum_{\bv\in \pi}\omega_{\bv}$.
It is reasonable to expect that such a path remains close to the straight line segment from $(0,0)$ to $(aN,bN)$. 
Moreover, if we assume that the weights $\omega_{\bv}$ along this path are all of roughly the same order (i.e., the large value of $E(\pi)$ is not due to a small number of exceptional sites $\bv$), then we may expect that
\begin{equation}\label{eq:high-spine}
\LPP(taN,tbN) \approx t\ell N
\end{equation}
for every $t\in [0,1]$.

For general points $(x,y)\in \R_+^2$, we conjecture that the geodesic to $(xaN, ybN)$ follows the high-time spine from $(1,1)$ to $(taN, tbN)$ for some $t\in [0,1]$, and then grow typically from $(taN, tbN)$ to $(xaN, ybN)$ (see Figure \ref{fig:pathmax}). Using \eqref{eq:recursion} and \eqref{eq:high-spine}, we therefore conjecture that, conditional on the event $\LPP(aN,bN)=\lv N$, 
\beq \label{eq:conj1}
    \lim_{N \to \infty} \frac{\LPP(xaN,ybN)}{N} 
    \eqinp \max \left\{ H(t)  : 0 \leq t \leq \min\{x,y,1\}\right\}
    \quad \text{where} \quad H(t) =t \ell + \bar{\LPP}_{ta, tb} (xa, yb). 
\eeq
We find that the maximizer is 
\beq \label{eq:tcc}
    t_c= \begin{cases} 1 \qquad & \text{for $(x,y)\in \Omega_1$,} \\
	\frac{\slope y-x}{\slope-1} \qquad & \text{for $(x,y)\in \Omega_2\setminus \Omega_1$ satisfying $y<x$,} \\
	\frac{\slope x-y}{\slope-1} \qquad & \text{for $(x,y)\in \Omega_2\setminus \Omega_1$ satisfying $y>x$,} \\
	0 \qquad & \text{for $(x,y)\in \R_+^2\setminus \Omega_2$,} \\
	\end{cases} 
\eeq
and the maximum value is 
\beq \label{eq:Htctemp}
	H(t_c)= t_c \lv+  \bar{\LPP}_{t_ca, t_cb} (xa, yb) = h(x,y)
\eeq
as in \eqref{eq:LLNconjv}.

\begin{remark}
Recall the function $h(x,y)$ in \eqref{eq:LLNconjv}. In physical coordinates $X=ax$ and $Y=by$, the function $U(X,Y) := h(X/a,Y/b)$ satisfies
\begin{equation}
    (U_X-1)(U_Y-1)-1 = 0, \quad (X,Y) \in \R^2_+ \setminus \{(at,bt): 0 \le t \le 1\},
\end{equation}
which is the Hamilton-Jacobi equation for exponential LPP.
The characteristic curves for this Hamilton--Jacobi equation, expressed in normalized coordinates $(x,y)$, satisfy $\frac{\d y}{\d x} = \frac{h_x-a}{h_y-b}$.
On a smooth linear piece of $\mv$ below the diagonal, the characteristic curves are given by the level curves of $(x,y) \mapsto \slope y - x$, while above the diagonal they are given by the level curves of $(x,y) \mapsto \slope x - y$.
These are precisely the lines along which the optimizer $t_c$ in \eqref{eq:tcc} remains unchanged. The fluctuations also remain the same along these characteristic lines; see Theorem~\ref{thm:offdiagflucsameside} and Section~\ref{sec:fluctuationsconjecture} below.
\end{remark}

\medskip
For the directed landscape, the results of \cite{DDV24}, especially Proposition 2.1, show that the heuristic argument above essentially holds in that model.\footnote{Private communication with Sayan Das.}

\subsection{Heuristic argument for Conjecture \ref{conj:offdiagfl}}  \label{sec:fluctuationsconjecture}

\begin{figure}[h]\centering
\begin{tikzpicture}[scale=2.2]
\draw[thin] (0.7, 0) to (0.7,0.7) to (0,0.7); 
\draw[thin] (0, 0) to (0.7, 0.7); 
\draw[very thick] (0, 0) to (0.37,0.3) to (0.8,0.45); 
\fill (0,0) circle(0.03);
\fill (0.37,0.3) circle(0.03);
\fill (0.8,0.45) circle(0.03);
\draw[dashed, thin] (0.8, 0) to (0.8, 0.45) to (0,0.45); 
\node at (0.9, -0.1) {$x$};
\node at (0.7, -0.1) {$1$};
\node at (-0.1, 0.45) {$y$};
\node at (-0.1, 0.7) {$1$};
\draw [->, thin][black] (0, 0) to (1.1,0);
\draw [->, thin][black] (0,0) to (0, 1.1);
\end{tikzpicture}
\qquad\quad 
\begin{tikzpicture}[scale=2.2]
\draw[thin] (0.7, 0) to (0.7,0.7) to (0,0.7); 
\draw[thin] (0, 0) to (0.7, 0.7); 
\draw[very thick] (0, 0) to (0.8,0.3); 
\fill (0,0) circle(0.03);
\fill (0.8,0.3) circle(0.03);
\draw[dashed, thin] (0.8, 0) to (0.8, 0.3) to (0,0.3); 
\node at (0.9, -0.1) {$x$};
\node at (0.7, -0.1) {$1$};
\node at (-0.1, 0.3) {$y$};
\node at (-0.1, 0.7) {$1$};
\draw [->, thin][black] (0, 0) to (1.1,0);
\draw [->, thin][black] (0,0) to (0, 1.1);
\end{tikzpicture}
\qquad\quad 
\begin{tikzpicture}[scale=2.2]
\draw[thin] (0.7, 0) to (0.7,0.7) to (0,0.7); 
\draw[thin] (0, 0) to (0.7, 0.7); 
\draw[very thick] (0, 0) to (0.7,0.7) to (0.8,0.9); 
\fill (0,0) circle(0.03);
\fill (0.7,0.7) circle(0.03);
\fill (0.8,0.9) circle(0.03);
\draw[dashed, thin] (0.8, 0) to (0.8, 0.9) to (0,0.9); 
\node at (0.9, -0.1) {$x$};
\node at (0.7, -0.1) {$1$};
\node at (-0.1, 0.9) {$y$};
\node at (-0.1, 0.7) {$1$};
\draw [->, thin][black] (0, 0) to (1.1,0);
\draw [->, thin][black] (0,0) to (0, 1.1);
\end{tikzpicture}
\caption{Conjectural maximizing path for the fluctuations}
\label{fig:flucconj}
\end{figure}
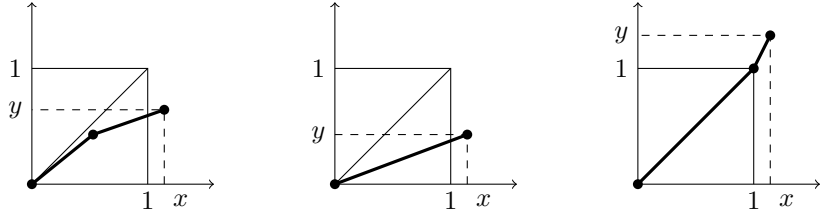

We assume that Theorem \ref{thm:diagfluc} has already been established, and now argue for the conjecture.
Suppose that $\LPP(aN, bN)=\lv N$ for some $\lv> \bar{\LPP}(a,b)$, and let $(x,y)\in \R_+^2\setminus \{(t,t): 0<t\le 1\}$. 
We follow the same reasoning as in Conjecture \ref{conj:LLN}, but now include the next-order asymptotic terms. 
Additionally, we consider more general ``mid-way" points. 
In the previous subsection, the mid-way points were $t(aN, bN)$, $t\in (0,1)$; see Figure \ref{fig:pathmax}. 
This time, we consider the points $\mb p^s_t N\in \R_+^2$, where 
\beqq
	\mb p^s_t= t(a, b) +  s(c_1 N^{-1/2}, -c_2 N^{-1/2})  
	\quad \text{with $c_1=  \frac{a(\ell-a+b)\ab}{\ell\sqrt{D}}$ and $c_2= \frac{b(\ell+a-b)\ab}{\ell\sqrt{D}}$}
\eeqq
for some $(t,s)\in (0,1)\times \R$. See the leftmost panel in Figure \ref{fig:flucconj}. 
As in the last subsection, we again expect that $t=t_c$: 
\beqq
	\LPP(xaN, ybN) \approx  \max\{ \LPP(\mb p^s_{t_c} N)+ \LPP_{\mb p^s_{t_c} N}(xaN, ybN) : s\in \R, \,  \mb p^s_{t_c}\in \R_+^2 \} .
\eeqq
From Theorem \ref{thm:diagfluc}, 
\beqq
    \LPP(\mb p^s_{t_c} N) \approx t_c \lv N +  \ab (\B_1(t_c) - |\B_2(t_c) - s| ) N^{1/2} .
\eeqq
On the other hand, using the unconditional fluctuation result \eqref{eq:onepoint}, we expect that
\beqq
	\LPP_{\mb p^s_{t_c} N}(xaN, ybN) \approx \bar{\LPP}_{\mb p^s_{t_c}}(xa, yb) N +  \frac{(\sqrt{a(x-t_c)}+\sqrt{b(y-t_c)})^{4/3}}{ (a(x-t_c)b(y-t_c))^{1/6}}   \TW_2 N^{1/3}.
\eeqq
Using Taylor's theorem, we see that 
\beqq
	\bar{\LPP}_{\mb p^s_{t_c}}(xa, yb) 
	= \bar{\LPP}_{(t_c a, t_c b)}(xa, yb) + \ab s R(t_c)N^{-1/2} + O(N^{-1})
\eeqq
where
\beq \label{eq:QinRQ}
    R(t)= Q \left(\frac{y-t}{x-t} \right), 
    \qquad 
    Q(u)= \frac{\sqrt{ab}}{\ell\sqrt{D}} \left[ \frac{\lv+a-b}{\sqrt{u}} - (\lv-a+b)\sqrt{u} - \frac{(a-b)(\lv-a-b)}{\sqrt{ab}} \right]. 
\eeq
Since $t_c \lv + \bar{\LPP}_{(t_c a, t_c b)}(xa, yb)=h(x,y)$ from \eqref{eq:Htctemp}, we are thus led to conjecture that 
\beqq
	\LPP(xaN, ybN) \approx h(x,y) N+ \ab Z N^{1/2} +  \frac{(\sqrt{a(x-t_c)}+\sqrt{b(y-t_c)})^{4/3}}{ (ab(x-t_c)(y-t_c))^{1/6}}   \TW_2 N^{1/3}
\eeqq
where
\beqq
	Z= \max\{ r(s) : s\in \R, \,  \mb p^s_{t_c}\in \R_+^2 \}, \qquad
	 r(s):= \B_1(t_c) - |\B_2(t_c) - s|   + R(t_c) s. 
\eeqq

We now evaluate $Z$. 
Observe that $Q(u)$ in \eqref{eq:QinRQ} is a monotonically decreasing function of $u>0$. Also, note from the formula $\slope=\frac{\lv-a-b+\sqrt{D}}{\lv-a-b-\sqrt{D}}$ that $Q(\frac1{\slope})=1$ and $Q(\slope)=-1$. 

\begin{itemize}
\item Suppose $(x,y)\in \Omega_1$. In this case, we have $t_c=1$ from \eqref{eq:tcc} and thus $r(s)= -|s|+R(1)s$. See the rightmost panel in Figure \ref{fig:flucconj}. Since the condition $(x,y)\in \Omega_1$ implies that $\frac1{\slope}<\frac{y-1}{x-1}<\slope$, we find that
$R(1)= Q(\frac{y-1}{x-1})\in [Q(\slope), Q(\frac1{\slope})]=[-1,1]$ by the monotonicity of the function $Q$. 
Hence, $|R(1)|\le 1$, and thus, the maximum of $r(s)= -|s|+R(1)s$ is $r(0)=0$. Therefore, $Z=0$ and 
$\LPP(xaN, ybN) \approx h(x,y) N +  \frac{(\sqrt{a(x-1)}+\sqrt{b(y-1)})^{4/3}}{ (ab(x-1)(y-1))^{1/6}}   \TW_2 N^{1/3}$. This is \eqref{eq:flcona}. 

\item Suppose $(x,y)\in \Omega_2\setminus\overline{\Omega}_1$. 
In this case, $t_c= \frac{\slope y-x}{\slope-1}$ if $y<x$, and $t_c= \frac{\slope x-y}{\slope-1}$ if $y>x$.  
See the leftmost panel in Figure \ref{fig:flucconj}.
Thus, $R(t_{c})=Q(\frac1{\slope})=1$ if $y< x$, and $R(t_{c})= Q(\slope)=-1$ if $y> x$. 
Hence, $r(s)=  \B_1(t_c) - |\B_2(t_c) - s|  \pm s$, with the sign $+$ for $y<x$ and $-$ for $y>x$. 
The maximum occurs at $s=\B_2(t_c)$, yielding 
$Z= \B_1(t_{c}) +\B_2(t_{c})$ if $y< x$, and $Z= \B_1(t_{c}) -\B_2(t_{c})$ if $y> x$. 
Thus, using \eqref{eq:defcd}, we find that $\LPP(xaN, ybN) \approx h(x,y) N+ \ab Z N^{1/2}$ with $Z= \sqrt{2} \cd_{\pm} \B_{\pm}(t_c)$ as in \eqref{eq:flconb}.

\item Suppose $(x,y)\in \R^2_+\setminus \overline{\Omega}_2$. In this case, we have $t_c=0$. See the middle panel in Figure \ref{fig:flucconj}.
Since $\mb p^s_0=s(c_1, -c_2)N^{-1/2}$ lies in $\R_+^2$ only for $s=0$, we find that $Z=0$. Therefore, 
$\LPP(xaN, ybN) \approx h(x,y) N+  \frac{(\sqrt{ax}+\sqrt{by})^{4/3}}{ (abxy)^{1/6}}   \TW_2 N^{1/3}$. 
This corresponds to \eqref{eq:flconc}. 
\end{itemize}
This completes our heuristic argument for Conjecture \ref{conj:offdiagfl}.

\subsection{Conjecture on conditional geodesics} 

The diagonal fluctuation result, Theorem \ref{thm:diagfluc}, suggests a conjecture regarding the geodesic. 
The following consideration is analogous to that in \cite[Conjecture 1.11]{LW24} for the conditional KPZ fixed point. 

Let $\pi_*$ be the geodesic from $(1,1)$ to the site $(aN, bN)$, i.e., 
\begin{equation*}
    \LPP(aN, bN) = \max_{\pi\in (1,1)\to (aN, bN)} E(\pi) = E(\pi_*).
\end{equation*}
Since $\omega_\bv$ are continuous random variables, the geodesic is unique almost surely. The path $\pi_*$ is a sequence of points in $\N^2$. We linearly interpolate so that it becomes a collection of a line segments. Using the basis vectors 
\beqq
	\vecv_1 =(a,b), \qquad  \vecv_2= \left(\frac{a(\ell-a+b)\ab}{\ell\sqrt{D}}, -  \frac{b(\ell+a-b)\ab}{\ell\sqrt{D}} \right)
\eeqq
we may write 
\beqq
	\pi_*=  \{ \tau \mathbf{v}_1 +  \pi^*(\tau) \mathbf{v}_2 \}_{\tau\in [0, N]} 
\eeqq
for a function $\pi^*(\tau)$, $\tau\in [0,N]$, satisfying $\pi^*(0)=\pi^*(N)=0$. 
By the geometry of the geodesic, this function is well-defined. 

Now, assume $\LPP(aN, bN)=\lv N$ and consider the geodesic to $(aN, bN)$. 
From the limit in Theorem \ref{thm:diagfluc}, 
we observe that the function $x\mapsto \B_1(t) - |\B_2(t) - x|$ achieves its maximum at $x= \B_2(t)$ with maximum value $\B_1(t)$. 
This observation leads us to the following conjecture.

\begin{conjecture}
Using the same notation as in Theorem \ref{thm:diagfluc}, we conjecture that 
\beqq \begin{split} 
    &\law \left( \left(  \frac{\pi^*(t N)}{N^{1/2}}  , \frac{\LPP \left( tN \mathbf{v}_1+ \pi^*(t N) \mathbf{v}_2 \right)- t\lv N}{\ab N^{1/2}}  \right)_{t\in (0,1)}
    \, \bigg| \,  \LPP(aN, bN)= \lv N \right) \fddto \law \left( (\B_2(t), \B_1(t) )_{t\in (0,1)} \right)
\end{split} \eeqq
where $\B_1$ and $\B_2$ are correlated Brownian bridges given by \eqref{eq:defcd}. 
\end{conjecture}

The transversal fluctuation exponent under the conditioning event is shown to be $1/2$ in \cite{BG23}. 
For the directed landscape, the convergence of a quantity similar to $\frac{\pi^*(t N)}{N^{1/2}}$ to a Brownian bridge was proved in  \cite{GHZ23}.

\section{Conditional multi-point distributions} \label{sec:finitemulti}

As mentioned in the Introduction, we prove Theorems \ref{thm:diagfluc} and \ref{thm:offdiagflucsameside} by computing the limits of an explicit formula for the conditional multi-point distributions.
The exponential LPP is equivalent to the continuous-time totally asymmetric simple exclusion process (TASEP) with step initial condition: for $(M,N) \in \N^2$ and $T \geq 0$,
\begin{equation} \label{eq:equiDLPPTASEP}
    \prob(\mc L(M,N) > T) = \prob ( \mathsf{x}_N(T) < M-N)
\end{equation}
where $\mathsf{x}_k(T)$ denotes the position of the $k^{\text{th}}$ particle in the TASEP at time $T$. 
In \cite{Liu22a}, Liu obtained an explicit formula for multi-time distributions for the TASEP. 
Using the relation \eqref{eq:equiDLPPTASEP}, the case $I = \{1, \cdots, m - 1\}$ in Proposition 2.3 of \cite{Liu22a}, specialized to the step initial condition, gives a formula for the probabilities 
\beqq
    \PP( \mc L(M_1,N_1) > T_1, \cdots , \mc L(M_{m-1},N_{m-1}) > T_{m-1},\mc L(M_m,N_m) \leq T_m) 
\eeqq
of the exponential LPP. 
We can thus find a formula for the multi-point conditional distributions by computing 
\begin{equation} \label{eq:condpasder}
\begin{split}
        &\PP( \mc L(M_1,N_1) > T_1, \cdots, \mc L(M_{m-1},N_{m-1}) > T_{m-1}|\mc L(M_m,N_m) = T_m) \\
        &= 
        \frac{ \frac{\partial}{\partial T_m}\PP( \mc L(M_1,N_1) > T_1, \cdots,  \mc L(M_{m-1},N_{m-1}) > T_{m-1}, \mc L(M_m,N_m) \leq T_m) }{ \frac{\partial}{\partial T_m} \PP(\mc L(M_m,N_m) \leq  T_m) }
\end{split}
\end{equation}

In this section, we state explicit formulas for multi-point conditional distributions.
We begin by introducing several notations in Subsection \ref{sec:conddef}. The main formula is presented in Proposition \ref{prop:cond} in Subsection \ref{sec:condmpd}.
A few special cases of the formula are discussed in Subsection \ref{sec:condspecial}.
Throughout this section, we fix a positive integer $m$.

\subsection{Definitions} \label{sec:conddef}

Let
\beq \label{eq:Cauchydd}
    \K_n(\bfr | \bfs) = \det\left[ \frac1{r_i-s_j}\right]_{i,j=1}^n = \frac{\prod_{1\le i<j\le n}(r_i-r_j)(s_j-s_i)}{\prod_{i,j=1}^n (r_i-s_j)}
\eeq
be the Cauchy determinant for the vectors $\bfr= (r_1,\cdots,r_n)$ and $\bfs= (s_1,\cdots,s_n)$ in $\C^n$. 
Define 
\beq \label{eq:STdf}
    \ST_n(\bfr | \bfs)=  \sum_{i=1}^{n} (r_i-s_i). 
\eeq
We often suppress the subscript $n$ if the sizes of the vectors are clear from context, and simply write $\K(\bfr | \bfs)$ and $\ST(\bfr | \bfs)$ instead. 
For $\bn = (n_1, \cdots, n_m)\in \N^m$, define the rational function 
\beq \label{eq:Pi_n} \begin{split}
        \Pi_{\vecn}(\bs \xi, \bs \eta) 
        =& \K_{n_1}(\bseta^{1}|\bsxi^{1}) \left[ \prod_{i=1}^{m-1} \K_{n_i+n_{i+1}}(\bsxi^{i}, \bseta^{i+1}|\bseta^{i}, \bsxi^{i+1}) \right] \K_{n_m}(\bsxi^{m}|\bseta^{m}) \ST_{n_m}(\bsxi^{m}|\bseta^{m})
\end{split} \eeq
where $\bs \xi= (\bs \xi^{1}, \cdots, \bs \xi^{m})$ and $\bs \eta= (\bs \eta^{1}, \cdots, \bs \eta^{m})$ 
with $\bsxi^i, \bseta^i\in \C^{n_i}$.  
When $m=1$, the above formula becomes, for $n\in \N$, 
\beqq
     \Pi_{n}(\bs \xi, \bs \eta) := \K_n(\bseta|\bsxi) \K_n(\bsxi|\bseta) \ST_n(\bsxi|\bseta)
     \qquad \text{for $\bs \xi, \bs \eta \in \C^n$.}
\eeqq

For $M,N\in \N$ and $T\in \R_+$, define the function\footnote{Throughout the paper $\log$ denotes the branch of the logarithm function that is analytic in $\C\setminus \ii \overline{ \R_-}$ and satisfies $\log 1=0$.} 
\beq \label{eq:f_expr}
    f_{M,N,T}(z) = \frac{z^N e^{Tz}}{(z+1)^M} = e^{ N\log z - M\log(z+1) + Tz} .
\eeq
For $\bM = (M_1,\cdots, M_m)\in \N^m$, $\bN = (N_1,\cdots, N_m)\in \N^m$, $\bT = (T_1,\cdots, T_m)\in \R_+^m$, and $\bn=(n_1, \cdots, n_m)\in \N^m$, define the function 
\beq \label{eq:defFF}
    \FF^{(\bn)}_{\bM, \bN, \bT} (\bsxi, \bseta) 
    =  \prod_{i=1}^m \prod_{k_i=1}^{n_i} \frac{\ff_i(\xi_{k_i}^{i})}{\ff_i(\eta_{k_i}^{i})}, 
    \qquad \ff_i(z)= \frac{f_{M_i, N_i,T_i}(z)}{f_{M_{i-1},N_{i-1},T_{i-1}}(z)}
\eeq
where $\bs \xi= (\bs \xi^{1}, \cdots, \bs \xi^{m})$ and $\bs \eta= (\bs \eta^{1}, \cdots, \bs \eta^{m})$ as before, with $\bsxi^i=(\xi^i_1, \cdots, \xi^i_{n_i})$ and $\bseta^i = (\eta^i_1, \cdots, \eta^i_{n_i})$ in $\C^{n_i}$.  
In the above formula, we set $M_0=N_0=T_0=0$. 

Let  
\beqq           C^{\text{in}}_{m,\text{left}},\cdots,C^{\text{in}}_{2,\text{left}}, C_{1,\text{left}}, C^{\text{out}}_{2,\text{left}}, \cdots, C^{\text{out}}_{m,\text{left}}
\eeqq 
be $2m-1$ small circles, nested from inside to outside, that enclose the point $-1$.  Similarly, let 
\beqq   
    C^{\text{in}}_{m,\text{right}},\cdots, C^{\text{in}}_{2,\text{right}}, C_{1,\text{right}}, C^{\text{out}}_{2,\text{right}}, \cdots, C^{\text{out}}_{m,\text{right}}
\eeqq
be $2m-1$ small circles, also nested from inside to outside, that enclose the point $0$ and are disjoint from the previous circles.  
See Figure \ref{fig:expLPP} for the case when $m=2$.
The circles are oriented counter-clockwise.\footnote{All closed contours in this paper are oriented counterclockwise, unless otherwise specified. The orientations of infinite contours will be stated explicitly.}  

\begin{figure}[h]
\begin{center}
\begin{tikzpicture}[scale=1.3]
    \draw[->] (-3.5,0) -- (1.5,0); 
    \draw[->] (0,-1.25) -- (0,1.25); 
    \node[below] at (-2,0) {$-1$};  
    \filldraw[black] (-2,0) circle (1pt);
    \draw[thick] (-2,0) circle [radius=0.45];
    \draw[thick] (-2,0) circle [radius=0.65];
    \draw[thick] (-2,0) circle [radius=0.85];
    \node[below] at (0,0) {$0$};    
    \filldraw (0,0) circle (1pt);
    \draw[thick] (0,0) circle [radius=0.45];
    \draw[thick] (0,0) circle [radius=0.65];
    \draw[thick] (0,0) circle [radius=0.85];
\end{tikzpicture}
\end{center}
\caption{Contours for $m = 2$: The three circles on the left are $C^{\text{in}}_{2,\text{left}}$, $C_{1,\text{left}}$, $C^{\text{out}}_{2,\text{left}}$ listed from inside to outside. The three circles on the right are $C^{\text{in}}_{2,\text{right}}$, $C_{1,\text{right}}$, $C^{\text{out}}_{2,\text{right}}$, also listed from inside to outside.}
\label{fig:expLPP}
\end{figure}
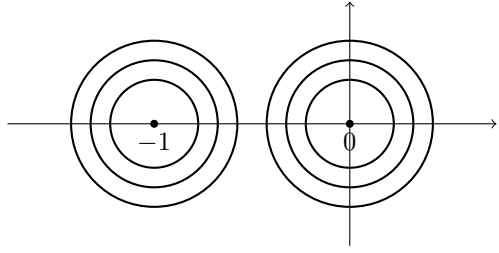

For $\vecn\in \N^m$, $\bM\in \N^m$, $\bN\in \N^m$, and $\bT\in \R_+^m$, 
define the polynomial $\DD^{(\bn)}_{\bM, \bN, \bT}(\bz)$ in $\bz=(z_1, \cdots, z_{m-1})$ of degree $2|\bn|-2n_1$ by    
\beq \label{eq:D_hat_n} \begin{split}
    \DD^{(\bn)}_{\bM, \bN, \bT}(\bz)= & \frac{1 }{(2\pi \ii)^{2|\bn|}}  
    \prod_{i=2}^m  \prod_{k_i = 1}^{n_i} 
    \left[ \int_{C_{i,\tn{left}}^{\tn{in}}} \dd\xi_{k_i}^{i}  +  z_{i-1} \int_{C_{i,\tn{left}}^{\tn{out}}} \dd\xi_{k_i}^{i} \right] 
    \left[ \int_{C_{i,\tn{right}}^{\tn{in}}} \dd\eta_{k_i}^{i} + z_{i-1} \int_{C_{i,\tn{right}}^{\tn{out}}} \dd\eta_{k_i}^{i} \right]  	\\
    &\times  \prod_{k_1 =1}^{n_1}  \left[  \int_{C_{1,\tn{left}}} \dd\xi_{k_1}^{1}  \right] \left[  \int_{C_{1,\tn{right}}} \dd\eta_{k_1}^{1}  \right] 
    \, \Pi_{\bn}(\bs \xi, \bs \eta) 
    \FF^{(\bn)}_{\bM, \bN, \bT}(\bsxi, \bseta) 
\end{split} \eeq 
where 
\beqq   
    |\bn|= n_1+\cdots + n_m \qquad \text{for $\bn = (n_1, \cdots, n_m)\in \N^m$.} 
\eeqq
The coefficients of this polynomial are linear combinations of $2|\bn|$-fold contour integrals. 
Note that when $m=1$, $\DD^{(n)}_{M, N, T}$  is a constant. 

\subsection{Formula for conditional multi-point distributions}\label{sec:condmpd}

We now state an explicit formula for the  conditional multi-point distributions. 
The formula is similar to that for the KPZ fixed point studied in \cite[Lemma 2.2 and Lemma 3.1]{LW24}, and the proof is also nearly identical, since the multi-point distribution formulas share similar structures. 
 
\begin{prop} \label{prop:cond}
Consider the exponential LPP in Definition \ref{def:explpp}. Fix an integer $m\ge 2$. 
Let $\bM = (M_1,\cdots, M_m)\in \N^m$, $\bN = (N_1,\cdots, N_m)\in \N^m$, and $\bT = (T_1,\cdots, T_m)\in \R_+^m$. 
Assume that $0< T_1 \leq \cdots \leq T_m$  and $(N_1,T_1), \cdots, (N_m,T_m)$ are all distinct.  
Then, 
\beq \label{eq:exp_cond}
   	\PP( \LPP(M_1,N_1) > T_1,\cdots , \LPP(M_{m-1},N_{m-1}) > T_{m-1}| \LPP(M_m,N_m) = T_m) 
    	= \frac{ \QQ_m(\bM, \bN, \bT)}{\QQ_1(M_m, N_m, T_m)}
\eeq
where 
\beq \label{eq:Q_hat}
    	\QQ_m(\bM, \bN, \bT) = \sum_{\bn\in\N^m}\frac{1}{(\bn!)^2} \QQ^{(\bn)}_m(\bM, \bN,\bT)
\eeq
with
\beq\label{eq:Q_hat_n} \begin{split} 
    	\QQ^{(\bn)}_m(\bM, \bN,\bT) = \frac{(-1)^{|\bn|+m-1}}{(2\pi \ii)^{m-1}} \oint_{>1}\cdots \oint_{>1} 
   	\DD^{(\bn)}_{\bM, \bN, \bT}(\bz)  \prod_{i=1}^{m-1} \frac{(z_i+1)^{n_i-n_{i+1}-1}}{z_i^{n_{i+1}+1}} \dd z_i . 
\end{split} \eeq
The function $\DD^{(\bn)}_{\bM, \bN, \bT}(\bz)$ is defined in \eqref{eq:D_hat_n}, and the contours are circles centered at the origin with  radii greater than $1$. 
\end{prop}

\begin{proof}
Recall the notation $\N_0:=\{0\}\cup \N$. 
Using the relation \eqref{eq:equiDLPPTASEP}, the formula for multi-point distributions for TASEP in \cite[Proposition 2.3]{Liu22a}, specialized to the case $I = \{1, \cdots, m - 1\}$, implies that\footnote{We have replaced $z_i$ with $-z_i$ in the formula (3) of \cite{Liu22a}.}  
\beq\label{eq:multExp} \begin{split}
    &\PP( \mc L(M_1,N_1) > T_1,..., \mc L(M_{m-1},N_{m-1}) > T_{m-1},\mc L(M_m,N_m) \leq T_m) 
    \\ 
    &= \frac{1}{(2\pi \ii)^{m-1}} 
    \sum_{\bn\in\N_0^m}\frac{(-1)^{|\bn|+m-1}}{(\bn!)^2} 
    \oint_{>1}\cdots \oint_{>1} \DDz^{(\bn)}_{\bM, \bN, \bT}(\bz) 
    \prod_{i=1}^{m-1} \frac{(z_i+1)^{n_i-n_{i+1}-1}}{z_i^{n_{i+1}+1}} \dd z_i
\end{split} \eeq
where $\DDz^{(\bn)}_{\bM, \bN, \bT}(\bz)$ is\footnote{We use different conventions from \cite{Liu22a}. By carefully accounting for the measures in Proposition 2.10 of \cite{Liu22a} and using the Cauchy determinant formula \eqref{eq:Cauchydd}, we find that $\mathsf D^{(\bn)}_{\bM, \bN, \bT}(z_1,\cdots,z_{m-1})$ is equal to $(-1)^{|\bn|}  \mathcal{D}_{\boldsymbol{n},{Y_{\mathrm{step}}}}(-z_1,\cdots,-z_{m-1}) \prod_{j=1}^{m-1}  \frac{z_j^{n_{j+1}}}{(1+z_j)^{n_j-n_{j+1}}}$ in terms of the notation in Section 2.1.3.2 of \cite{Liu22a}.} 
the same as $\DD^{(\bn)}_{\bM, \bN, \bT}(\bz)$, except that the term $\Pi_{\bn}(\bs \xi, \bs \eta)$ is replaced by  $\widetilde{\Pi}_{\bn}(\bsxi, \bseta)$, which is given by \eqref{eq:Pi_n} without the factor $\ST(\bsxi^{m}|\bseta^{m})$. 
The assumptions that $0< T_1 \leq \cdots \leq T_m$ and that the pairs $(N_1,T_1), \cdots, (N_m,T_m)$ are all distinct are necessary since \cite[Proposition 2.3]{Liu22a} requires similar conditions.

We insert the above formula into equation \eqref{eq:condpasder}. 
Noting that $T_m$ appears only in the function $f_{M_m,N_m, T_m}$, we have 
\beqq
    \frac{\partial}{\partial T_m} \widetilde{\Pi}_{\bn}(\bs \xi, \bs \eta)  \FF^{(\bn)}_{\bM, \bN, \bT}(\bsxi, \bseta) 
    = \Pi_{\bn}(\bs \xi, \bs \eta)  \FF^{(\bn)}_{\bM, \bN, \bT}(\bsxi, \bseta) . 
\eeqq
Thus, we arrive at the formula \eqref{eq:exp_cond}, but with the series in both the numerator and denominator taken over $\bn\in \N_0^m$.  

Now Lemma \ref{lem:trivial} below implies that $\QQ_m^{(\bn)}(\bM, \bN, \bT)=0$ if $\bn \in (\N_0^{m-1}\setminus \N^{m-1}) \times \N$. 
Furthermore, since $\Pi_{\bn}(\bs \xi, \bs \eta) =0$ for $\bn\in \N_0^{m-1}\times \{0\}$ by formula \eqref{eq:Pi_n}, 
we find that $\QQ^{(\bn)}(\bM, \bN, \bT)=0$ for $\bn\in \N_0^m\setminus \N^m$. 
Therefore, the series over $\bn \in\N_0^m$ reduces to a series over $\bn\in \N^m$. 
The series for the denominator is similar. 
\end{proof}


The following lemma is used in the proof of the above proposition. 
\begin{lem}\label{lem:trivial}
Let $H_{\bn}(\bs \xi, \bs \eta)$ be an integrable function that does not depend on $\bz$. Then, the function 
\beqq \begin{split}
    G^{(\bn)}(\bz) = & \prod_{i=2}^m \prod_{k_i = 1}^{n_i} 
    \left[ \int_{C_{i,\tn{left}}^{\tn{in}}} \dd\xi_{k_i}^{i}  +  z_{i-1} \int_{C_{i,\tn{left}}^{\tn{out}}} \dd\xi_{k_i}^{i} \right] 
    \left[ \int_{C_{i,\tn{right}}^{\tn{in}}} \dd\eta_{k_i}^{i} + z_{i-1} \int_{C_{i,\tn{right}}^{\tn{out}}} \dd\eta_{k_i}^{i} \right] 	
    H_{\bn}(\bs \xi, \bs \eta)
\end{split} \eeqq
satisfies 
\beqq
    \oint_{>1}\cdots \oint_{>1} G^{(\bn)}(\bz)  \prod_{i=1}^{m-1} \frac{(z_i+1)^{n_i-n_{i+1}-1}}{z_i^{n_{i+1}+1}} \dd z_i 
    =0  
    \qquad \text{for $\bn \in (\N_0^{m-1}\setminus \N^{m-1}) \times \N$.}
\eeqq 
\end{lem}

\begin{proof}
Note that $G^{(\bn)}(\bz)$ is a polynomial of degree $2n_{i+1}$ in each variable $z_i$ for $i=1, \cdots, m-1$. 
If $\bn \in (\N_0^{m-1}\setminus \N^{m-1}) \times \N$, then there exists $i\in \{1, \cdots, m-1\}$ such that $n_{i}=0$ and $n_{i+1} \ge 1$. In this case, $G^{(\bn)}(\bz) = O(z_{i}^{2n_{i+1}})$ and $\frac{(z_{i}+1)^{n_{i}-n_{i+1}-1}}{z_{i}^{n_{i+1}+1}}= O(z_{i}^{-2n_{i+1}-2})$ as $z_{i}\to \infty$. 
Therefore, the integrand decays sufficiently fast at infinity, and the result follows from Cauchy's theorem. 
\end{proof}

\subsection{Formulas for two special cases} \label{sec:condspecial}

The case where $\bn= (1, \cdots, 1)=:\bone$ will play a special role. The formula \eqref{eq:Q_hat_n} simplifies in this case. 

\begin{lem} \label{lem:Q_hat_111}  
We have  
\beq \label{eq:Q_hat_111}
    \QQ^{(\bone)}_m(\bM,\bN,\bT)  
    = -\frac{1}{(2\pi \ii)^{2m}} \int_{\vec{\gamma}} \dd \bsxi \int_{\vec{\Gamma}}\dd \bseta
    \,\, \Pi_{\bone}(\bs\xi, \bseta)  \FF^{(\bone)}_{\bM,\bN,\bT}(\bs\xi, \bseta) 
\eeq
where $\bsxi= (\xi^1, \cdots, \xi^m) \in \C^m$ and $\bseta= (\eta^1, \cdots, \eta^m) \in \C^m$. The contours are 
\beq \label{eq:Cveccont}
    \vec{\gamma} = \gamma_1\times \cdots \times \gamma_m, 
    \qquad
    \vec{\Gamma} = \Gamma_1\times \cdots\times \Gamma_m, 
\eeq
where $\gamma_1, \cdots, \gamma_m$ are small circles around the point $z=-1$, nested from inside to outside, and $\Gamma_1, \cdots, \Gamma_m$ are small circles around the point $z=0$, also nested from inside to outside, such that all circles are mutually disjoint. 
\end{lem}

\begin{proof}
When $\bn= \bone$, the formula \eqref{eq:Q_hat_n} becomes
\beqq \begin{split} 
    \QQ^{(\bone)}_m(\bM, \bN,\bT) = - \frac{1}{(2\pi \ii)^{m-1}} \oint_{>1}\cdots \oint_{>1} 
    \DD^{(\bone)}_{\bM, \bN, \bT}(\bz)  \prod_{i=1}^{m-1} \frac{1}{z_i^{2} (z_i+1)} \dd z_i . 
\end{split} \eeqq
The function $\DD^{(\bone)}_{\bM, \bN, \bT}(\bz)$ is a polynomial of degree $2$ in each $z_i$ for $i=1, \cdots, m-1$. 
Thus, the $z_i$-integrals retain only the leading coefficients of the polynomial, which effectively removes all $C_{i,\tn{left}}^{\tn{in}}$- and $C_{i,\tn{right}}^{\tn{in}}$-integrals. 
We then relabel the contours as follows:  $C_{1,\tn{left}}= \gamma_1$, $C_{1,\tn{right}}=\Gamma_1$, and 
$C_{i,\tn{left}}^{\tn{out}}=\gamma_i$ and $C_{i,\tn{right}}^{\tn{out}}=\Gamma_i$ for $i=2, \cdots, m$. 
The same calculation was also carried out for the KPZ fixed point in \cite[Lemma 3.5]{LW24}. 
\end{proof}

For later use, we note that
\beq \label{eq:Pi_111}
    \Pi_{\bone}(\bs{ \xi}, \bs{\eta})
    = \frac{(-1)^{m}}{\xi^{m}-\eta^{m}} \prod_{i=1}^{m-1} \frac{(\xi^{i}-\eta^{i+1})(\eta^{i}-\xi^{i+1})}{(\xi^{i}-\xi^{i+1})(\eta^{i}-\eta^{i+1})(\xi^{i}-\eta^{i})^2}.
\eeq

\medskip

For the proof of Theorem \ref{thm:offdiagflucsameside}, we also need to evaluate the limit of 
$\QQ^{(\bn)}_m(\bM,\bN,\bT) $ when $m=3$ and $\bn=(1,2,1)$. 
This term has the following explicit formula. 

\begin{lem} \label{lem:Q_hat_121}
Let 
$\vec{\gamma}=\gamma_1\times \gamma_2\times \gamma_3\times \gamma_4$ and 
$\vec{\Gamma}=\Gamma_1\times \Gamma_2\times \Gamma_3\times \Gamma_4$, 
where $\gamma_1, \cdots, \gamma_4$ are small circles around $-1$, nested from inside to outside, $\Gamma_1, \cdots, \Gamma_4$ are small circles around $0$, also nested from inside to outside, with all circles mutually disjoint.
We have
\beq \label{eq:Q123df} \begin{split}
    \QQ^{(1,2,1)}_3(\bM,\bN,\bT)  
    =& \frac{1}{(2\pi \ii)^8} \int_{\vec{\gamma}} \dd \bsxi^{3122} \int_{\vec{\Gamma}} \dd \bseta^{1223} \, 
    \Pi_{(1,2,1)}(\bsxi, \bseta)  \FF^{(1,2,1)}_{\bM,\bN,\bT}(\bsxi, \bseta) \\
    &+ \frac{1}{(2\pi \ii)^8}  \int_{\vec{\gamma}} \dd \bsxi^{1223} \int_{\vec{\Gamma}} \dd \bseta^{3122} \, 
    \Pi_{(1,2,1)}(\bsxi, \bseta) \FF^{(1,2,1)}_{\bM,\bN,\bT}(\bsxi, \bseta)   
\end{split} \eeq
where $\bsxi=(\xi^{1}, \xi^{2}_1, \xi^{2}_2, \xi^{3})$,  $\bsxi^{3122}= (\xi^{3}, \xi^{1}, \xi^{2}_1, \xi^{2}_2)$,  $\bsxi^{1223}= (\xi^{1}, \xi^{2}_1, \xi^{2}_2, \xi^{3})$, and similarly for $\bseta, \bseta^{3122}, \bseta^{1223}$. 
\end{lem}

\begin{proof}
When $m=3$ and $\bn=(1,2,1)$, we need to compute 
\beqq \begin{split}
    \oint_{>1} \oint_{>1} 
    &\left[ \int_{C_{3,\tn{left}}^{\tn{in}}} \dd\xi^{3}  +  z_{2} \int_{C_{3,\tn{left}}^{\tn{out}}} \dd\xi^{3} \right] 
    \left[ \int_{C_{3,\tn{right}}^{\tn{in}}} \dd \eta^{3}  +  z_{2} \int_{C_{3,\tn{right}}^{\tn{out}}} \dd \eta^{3} \right] \\
    &	\prod_{i=1}^2 \left[ \int_{C_{2,\tn{left}}^{\tn{in}}} \dd\xi_{i}^{2}  +  z_{1} \int_{C_{2,\tn{left}}^{\tn{out}}} \dd\xi_{i}^{2} \right] 
    \left[ \int_{C_{2,\tn{right}}^{\tn{in}}} \dd \eta_{i}^{2}  +  z_{1} \int_{C_{2,\tn{right}}^{\tn{out}}} \dd \eta_{i}^{2} \right] 
    \frac{\dd z_1 \dd z_2}{(z_1+1)^2 z_1^3 z_2^2} . 
\end{split} \eeqq
Evaluating the $z_1$ and $z_2$-integrals, we obtain the result. 
\end{proof}

\section{Miscellaneous lemmas} \label{sec:threelemmas}

We record the following two lemmas, which will be used in several places throughout this paper.

\begin{lem}[\cite{LW24}] \label{lem:bridge}
Let $m\ge 2$. 
Let $\Gamma_1, \cdots, \Gamma_m$ be disjoint contours, listed from left to right, each parallel to the $y$-axis with upwards orientation. 
Let $\vec\Gamma \equiv \Gamma_1\times\cdots\times\Gamma_m$. 
For every $0=a_0<a_1<\cdots<a_m = A$ and $b_1,\cdots,b_{m-1}\in\R$ with $b_0=b_m=0$,  
\beqq
    \frac{\sqrt{2\pi A}}{(2\pi\ii)^m} \int_{\vec\Gamma} 
    \frac{\prod_{i=1}^m e^{\frac{1}{2}(a_i-a_{i-1})u_i^2+(b_{i}-b_{i-1})u_i} }{\prod_{i=1}^{m-1}(u_{i+1}-u_i)}  \dd \bu    
    = \prob\left({\sqrt{A} \, \B\left(\frac{{a_{i}}}{A}\right)>b_i, \,\,  i=1,\dots,m-1}\right)
\eeqq
where $\dd\bu= \dd u_1\cdots\dd u_m$  with each $u_i\in\Gamma_i$, and $\B$ is a standard Brownian bridge. 
\end{lem}

\begin{proof}
The equality can be verified by expressing the right-hand side in terms of the usual density function for a Brownian bridge, and then taking derivatives with respect to $b_1, \cdots, b_{m-1}$. The details can be found in Lemma 3.4 of \cite{LW24}. 
\end{proof} 

When proving the main theorems, we first establish them for parameters lying outside certain hypersurfaces.
We then extend the results to the full set of parameters using the next lemma.
The proof essentially follows that of Lemma 3.6 in \cite{LW24}, although we present the result here in a slightly different form.

\begin{lem}\label{lem:bootstrap} 
Let $I$ be an open interval in $\R$ and let $y_0\in I$. 
For each $n\in \N$, let $A_n$ be an event, and let $\{Y_n(y)\}_{y\in I}$ be a stochastic process. 
Let $r\in\R$. 
Suppose that the following two conditions hold:
\begin{enumerate}[(a)]
\item There is a continuous function $f$ on $I$ such that
\beq \label{eq:convergenced2}
    \lim_{n\to\infty} \prob(\{Y_n(y)\le r \} \cap A_n) =f(y)
    \qquad \text{for every} \quad y \in I\setminus \{y_0\}.
\eeq
\item There is a continuous function $g$ on $I\times I$ satisfying $g(y,y)=0$ for $y\in I$,  such that
\beqq
    \lim_{n\to\infty} \prob(Y_n(y)\le r, \, Y_n(y')>r  ) =g(y,y') 
    \qquad \text{for every} \quad  y, y'\in I \quad \text{with} \quad y\neq y'. 
\eeqq
\end{enumerate}
Then, \eqref{eq:convergenced2} also holds for $y=y_0$. 
\end{lem}

\begin{proof}
Let $y \in I\setminus \{y_0\}$.  
Noting that  
\beqq \begin{split}
    \prob(\{Y_n(y_0)\le r \} \cap A_n) 
    \le \prob(\{Y_n(y)\le r \} \cap A_n) +  \prob(Y_n(y_0)\le r, Y_n(y)> r ), 
\end{split} \eeqq
we find that
\beqq \begin{split}
    \limsup_{n\to \infty} \prob(\{Y_n(y_0)\le r \} \cap A_n) \le f(y)+ g(y_0, y) .
\end{split} \eeqq
Similarly, since 
\beqq \begin{split}
    \prob(\{Y_n(y_0)\le r \} \cap A_n) 
    \ge \prob(\{Y_n(y)\le r  \} \cap A_n) -  \prob(Y_n(y_0)> r, Y_n(y)\le  r ), 
\end{split} \eeqq
we find that  
\beqq \begin{split}
    \liminf_{n\to \infty} \prob(\{Y_n(y_0)\le r \} \cap A_n) \ge f(y)- g(y, y_0). 
\end{split} \eeqq
Taking the limit as $y\to y_0$ and using the continuity of $f$ and $g$ and the fact that $g(y_0, y_0)=0$, we conclude that $\prob(\{Y_n(y_0)\le r \} \cap A_n)$ converges to $f(y_0)$ as $n\to \infty$. 
\end{proof}

\section{Asymptotic analysis of a function} \label{sec:asfunction}

When we evaluate the limits of the formulas in Proposition \ref{prop:cond}, we require the asymptotic properties of the functions $f_{M,N,T}(z) = e^{ N\log z - M\log(z+1) + Tz}$, defined in \eqref{eq:f_expr}, as the parameters $M, N, T$ tend to infinity. 
In this section, we summarize the relevant asymptotic results for this function.

\subsection{Asymptotic properties}

Let $\alpha_1, \alpha_2, \alpha_3 \in \R \setminus \{0\}$ and $\beta_1, \beta_2, \beta_3\in \R$. 
For every $L>0$, let $\delta^1_L$ and $\delta^2_L$ be real numbers such that 
\beqq
    \alpha_1 L + \beta_1 L^{1/2} + \delta^1_L \in \Z, \qquad 
    \alpha_2 L + \beta_2 L^{1/2} + \delta^2_L \in \Z, 
\eeqq
and assume that $\delta^1_L, \delta^2_L$ are uniformly bounded for all $L> 0$. 
Define the functions 
\beq \label{eq:Gaph} \begin{split}
    &\GG(z)=   - \alpha_1 \log (z+1) + \alpha_2\log z + \alpha_3 z, \\
    &\HH(z)= -\beta_1 \log(z+1)+ \beta_2 \log z  + \beta_3 z, \\
    &\EE_L(z)= -\delta^1_L \log (z+1) + \delta^2_{L}\log z, 
\end{split} \eeq
and 
\begin{equation} \label{eq:fdef22}
    \ff_{L}(z) := e^{L \GG(z) + L^{1/2} \HH(z)+ \EE_L(z)} . 
\end{equation}

\begin{lem} \label{lem:asymptotics_f1}
Let $\cp$ be a critical point of $\GG(z)$. 
Then, for every $\epsilon \in (0, 1/2)$, there exists a constant $L_0>0$ such that for all $L>L_0$ and $|w|\leq L^{\epsilon/3}$, 
\begin{equation} \label{eq:z_nbd_asym}
    \ff_{L}( \cp + w L^{-1/2}) 
    = \ff_{L}( \cp) e^{\frac{1}2 \GG''(\cp)  w^2 + \HH'(\cp) w}
    \big(1 + O(L^{-1/2+\epsilon}) \big).
\end{equation}
\end{lem}

\begin{proof}
This follows from Taylor's theorem expanded at $z=\cp$: 
\beqq
    \ff_{L}(z) = f_{L}(\cp) 
    e^{L \left[ \frac12 \GG''(\cp)  (z-\cp)^2+ O( | z-\cp|^3) \right] + L^{1/2} \left[ \HH'(z_c) (z- \cp)  + O(|z-\cp|^2) \right]
    + O(|z-\cp|)}. 
\eeqq
\end{proof}

\begin{lem} \label{lem:basic}
The critical points of $\GG(z)$ are
\beq \label{eq: z_pm}
    \cp^{\pm}= \frac{-\alpha_3+\alpha_1-\alpha_2 \pm \sqrt{\Qt}}{2\alpha_3} 
    \qquad \text{where} \quad 
    \Qt= \alpha_3^2 -2(\alpha_1 + \alpha_2)\alpha_3 + (\alpha_1-\alpha_2)^2, 
\eeq
with the following cases: 
\begin{enumerate}[(a)]
\item If $\alpha_1,\alpha_2 > 0$ and $\alpha_3>(\sqrt{\alpha_1}+\sqrt{\alpha_2})^2$, then  $\Qt>0$,  
$\GG''(\cp^{+}) <0$, $\GG''(\cp^{-}) >0$, and $-1< \cp^-< \cp^+<0$.  

\item If $\alpha_1 < 0$ and $\alpha_2,\alpha_3 > 0$, then $\Qt > 0$, $\GG''(\cp^{+}) <0$, $\GG''(\cp^{-}) <0$, and $\cp^- < -1 < \cp^+ < 0$.  

\item If $\alpha_2 < 0$ and $\alpha_1,\alpha_3 > 0$ then $\Qt > 0$, $\GG''(\cp^{+}) >0$, $\GG''(\cp^{-}) >0$, and $-1 < \cp^- < 0 < \cp^+$.  
\end{enumerate}
Furthermore,  
\beq \label{eq:Atpm}
    \GG''(\cp^{\pm})= \mp \frac{\sqrt{\Qt}}{2\alpha_1\alpha_2} \left[ (\alpha_1+\alpha_2)\alpha_3-  (\alpha_1-\alpha_2)^2 \pm  (\alpha_1-\alpha_2) \sqrt{\Qt} \right]. 
\eeq
\end{lem}

\begin{proof}
Since 
\beqq
    \GG'(z)= - \frac{\alpha_1}{z+1} + \frac{\alpha_2}{z} + \alpha_3 = \frac{\alpha_3 z^2 + (\alpha_3-\alpha_1+\alpha_2)z + \alpha_2 }{(z+1)z}, 
\eeqq
we obtain \eqref{eq: z_pm}. 
It is also direct to verify \eqref{eq:Atpm}. 
Set $\At_{\pm}= \mp \GG''(\cp^{\pm})$. 
Note that $\At_+\At_- =  \frac{\alpha_3^2\Qt}{\alpha_1\alpha_2}$. 

\begin{enumerate}[(a)]
    \item Suppose $\alpha_1,\alpha_2 > 0$ and $\alpha_3 > (\sqrt{\alpha_1}+\sqrt{\alpha_2})^2$. 
Since $\Qt= (\alpha_3- (\sqrt{\alpha_1}+\sqrt{\alpha_2})^2)(\alpha_3- (\sqrt{\alpha_1}-\sqrt{\alpha_2})^2)$, 
the condition $\alpha_3> (\sqrt{\alpha_1}+\sqrt{\alpha_2})^2$ implies  $\Qt>0$. 
The same condition also implies that 
\beqq 
    \frac{\alpha_1\alpha_2}{\sqrt{\Qt}}(\At_++\At_-) 
    = (\alpha_1+\alpha_2)\alpha_3-  (\alpha_1-\alpha_2)^2 > 2\sqrt{\alpha_1\alpha_2}(\sqrt{\alpha_1}+\sqrt{\alpha_2})^2 > 0. 
\eeqq 
Since $\At_+\At_- =  \frac{\alpha_3^2\Qt}{\alpha_1\alpha_2}>0$, we find that $\At_{\pm}>0$. 
The inequalities $-1< \cp^-< \cp^+<0$ follow from the inequalities $\sqrt{\Qt} < \alpha_3\pm (\alpha_1-\alpha_2)$, which can be checked by squaring both sides. 

\item Suppose $\alpha_1 < 0$ and $\alpha_2,\alpha_3 > 0$. 
Since $\Qt = (\alpha_3 +\alpha_1 - \alpha_2)^2 - 4\alpha_1\alpha_3$, we find that $\Qt > 0$. 
In this case, $\At_+\At_- =  \frac{\alpha_3^2\Qt}{\alpha_1\alpha_2}<0$ and $\At_+>\At_-$. Thus, $A_+ > 0$ and $A_- < 0$. 
The property $\cp^- < -1 < \cp^+ < 0$ follows from the inequalities $|\alpha_3 +\alpha_1-\alpha_2|< \sqrt{\Qt} < \alpha_3-\alpha_1+\alpha_2$.

\item 
Suppose $\alpha_2 < 0$ and $\alpha_1,\alpha_3 > 0$. 
From $\Qt = (\alpha_3 -\alpha_1 + \alpha_2)^2 - 4\alpha_2\alpha_3$, we see that $\Qt > 0$. 
Since $\At_+\At_- =\frac{\alpha_3^2\Qt}{\alpha_1\alpha_2}<0$ and $\At_+<\At_-$, it follows that $\At_+ < 0$ and $\At_- > 0$. 
The property $-1 < \cp^- < 0 < \cp^+$ follows from noting that $ |\alpha_3 - \alpha_1+\alpha_2|< \sqrt{\Qt} < \alpha_3 + \alpha_1 - \alpha_2$. 
\end{enumerate}
\end{proof}

\begin{lem} \label{lem:asymptotics_f2} 
Let $\cp^\pm$ be the critical points of $\GG(z)$ as given in \eqref{eq: z_pm}, and let $b\in \mathbb{R}$. For all sufficiently large $L>0$, define the circles
\beq
    \con^{b, L}_{\lt}= \{ z\in \C \, : \, |z+1|= | \cp^- +1|+b L^{-\frac12} \}, 
    \qquad
    \con^{b, L}_{\rt}= \{ z\in \C \, : \, |z|= |\cp^+| +b L^{-\frac12} \}. 
\eeq
\begin{itemize}
\item If $\alpha_1,\alpha_2 > 0$ and $\alpha_3 > (\sqrt{\alpha_1}+\sqrt{\alpha_2})^2$, then both  statements \ref{eq:propertya}  and \ref{eq:propertyb}  below hold.  
\item If $\alpha_1 < 0$ and $\alpha_2, \alpha_3 > 0$, then statement \ref{eq:propertyb}  holds. 
\item If $\alpha_2 < 0$ and $\alpha_1, \alpha_3>0$, then statement \ref{eq:propertya}  holds.  
\end{itemize}

\begin{enumerate}[(a)] 
\item\label{eq:propertya} 
There exists a constant $c>0$ such that for every $0 < \epsilon < 1/2$, 
\beq \label{eq:f21}
    \left| \frac{\ff_{L}(z)}{\ff_{L}(\cp^-)} \right| =   O(e^{-cL^{2\epsilon/3}})  
    \qquad \text{for $z \in \Sigma^{b, L}_{\lt}\cap \{ z\in \C : |z-\cp^-| \ge L^{-\frac12 + \frac{\epsilon}{3}}\}$}
\eeq
as $L \to \infty$. Furthermore, there exist  $L_0>0$ and $C>0$ such that for all $L>L_0$,
\beq \label{eq:f2aub_xi}
    \left| \frac{\ff_{L}(z)}{\ff_{L}(\cp^-)} \right| \le C \quad \text{for $z \in \con^{b, L}_{\lt}$}. 
\eeq

\item \label{eq:propertyb} 
There exists a constant $c>0$ such that for every $0 < \epsilon < 1/2$,
\beq \label{eq:f22}
    \left| \frac{\ff_{L}(\cp^+)}{\ff_{L}(z)} \right| =   O(e^{-cL^{2\epsilon/3}})  
    \qquad \text{for $z \in \Sigma^{b, L}_{\rt}\cap \{ z\in \C : |z- \cp^+| \ge L^{-\frac12 + \frac{\epsilon}{3}}\}$}
\eeq
as $L\to \infty$. 
Furthermore, there exist $L_0>0$ and $C>0$ such that for all $L>L_0$,
\beq \label{eq:f2aub_eta}
    \left| \frac{\ff_{L}(\cp^+)}{\ff_{L}(z)} \right| \le C \quad \text{for  $z \in \con^{b, L}_{\rt}$}.
\eeq
\end{enumerate}   
\end{lem}

\subsection{Proof of Lemma \ref{lem:asymptotics_f2}} \label{sec:pf_asymptotics_f}

We use the following result. 

\begin{lem} \label{lemma:G1decay}
\begin{enumerate}[(a)]
\item If $\alpha_1,\alpha_2 > 0$ and $\alpha_3 > (\sqrt{\alpha_1}+\sqrt{\alpha_2})^2$, then 
\beq \label{eq:derGspm}
    |\cp^-+1| < s_-:=1-\sqrt{\alpha_2/\alpha_3}, \qquad |\cp^+|< s_+:=1- \sqrt{\alpha_1/\alpha_3},
\eeq
and, for every $s \in (0, \mr s_-)$,  
\beq\label{eq:derG1the1} \begin{split}
    &\frac{\partial }{\partial \theta}  \re \GG(-1+se^{\tn{\i} \theta })  < 0 \quad \text{for $\theta\in (0,\pi)$;} \qquad 
    \frac{\partial }{\partial \theta}  \re \GG(-1+se^{\tn{\i} \theta })  > 0 \quad \text{for $\theta\in (-\pi, 0)$,}
\end{split} \eeq
and, for every $s \in (0, \mr s_+)$,  
\beq \label{eq:derG1the2} \begin{split}
            &\frac{\partial }{\partial \theta}  \re \GG(se^{\tn{\i} \theta })  < 0 \quad \text{for $\theta\in (0,\pi)$;} \qquad 
            \frac{\partial }{\partial \theta}  \re \GG(se^{\tn{\i} \theta })  > 0 \quad \text{for $\theta\in (-\pi, 0)$.}
\end{split} \eeq
\item If $\alpha_1 < 0$ and $\alpha_2,\alpha_3 > 0$, then \eqref{eq:derG1the2} holds with $\mr s_+ = 1$.  
\item If $\alpha_2 < 0$ and $\alpha_1,\alpha_3 > 0$, then \eqref{eq:derG1the1} holds with $\mr s_- = 1$. 
\end{enumerate}
\end{lem}

\begin{proof}
(a) Suppose $\alpha_1,\alpha_2 > 0$ and $\alpha_3 > (\sqrt{\alpha_1}+\sqrt{\alpha_2})^2$. 
From the formula \eqref{eq: z_pm} for $\cp^\pm$, 
the properties $\mr s_-> |\cp^-+1|=\cp^-+1$ and $\mr s_+> |\cp^+|=-\cp^+$ hold since 
\beqq \begin{split}
    (\alpha_3 - \alpha_1 + \alpha_2+\sqrt{\Qt})^2 -  4\alpha_2\alpha_3
    &= 2 \Qt + 2(\alpha_3 - \alpha_1+\alpha_2) \sqrt{ \Qt} >0,  \\
    (\alpha_3 + \alpha_1- \alpha_2+\sqrt{\Qt})^2 -  4\alpha_1\alpha_3
    &= 2 \Qt + 2(\alpha_3 + \alpha_1-\alpha_2) \sqrt{ \Qt} >0.
\end{split} \eeqq
From the formula of $\GG$, we have 
\beqq \begin{split}
    \frac{\partial }{\partial \theta} \re\GG( -1+se^{\tn{\i} \theta } ) 
    = s \sin \theta \left(  \frac{\alpha_2}{1+s^2 - 2s\cos \theta} - \alpha_3 \right), 
    \quad 
    \frac{\partial }{\partial \theta} \re\GG(se^{\tn{\i} \theta } ) 
    = s \sin \theta \left( \frac{\alpha_1}{1+s^2 + 2s\cos \theta} - \alpha_3\right). 
\end{split}  \eeqq
Note that $0<s_\pm<1$. If $s\in (0, \mr s_-)$, then 
\beqq
    \frac{\alpha_2}{1+s^2 - 2s\cos \theta} - \alpha_3 \le \frac{\alpha_2 }{(1-s)^2}-  \alpha_3 <0
\eeqq
for every $\theta$. Similarly, if $s\in (0, \mr s_+)$, then 
\beqq
    \frac{\alpha_1}{1+s^2 +2s\cos \theta} - \alpha_3 \le \frac{\alpha_1 }{(1-s)^2}-  \alpha_3  <0 
\eeqq
for every $\theta$. Hence, \eqref{eq:derG1the1} and  \eqref{eq:derG1the2} hold. 

(b) If $\alpha_1 < 0$ and $\alpha_2,\alpha_3 > 0$, then 
$ \frac{\alpha_1}{1+s^2 + 2s\cos \theta} - \alpha_3 \leq \frac{\alpha_1}{(1+s)^2} - \alpha_3
< 0$
for every $s \in (0,1)$ and all $\theta$. Thus, \eqref{eq:derG1the2} holds with $s_+=1$. 

(c) If $\alpha_2 < 0$ and $\alpha_1,\alpha_3 > 0$, then 
$ \frac{\alpha_2}{1+s^2 - 2s\cos \theta} - \alpha_3 \leq \frac{\alpha_2}{(1+s)^2} - \alpha_3
< 0$
for every $s \in (0,1)$ and all $\theta$. Thus, \eqref{eq:derG1the1} holds with $s_-=1$. 
\end{proof}

\medskip

\begin{proof}[Proof of Lemma \ref{lem:asymptotics_f2}] 
$\bullet$ Suppose $\alpha_1,\alpha_2 > 0$ and $\alpha_3 > (\sqrt{\alpha_1}+\sqrt{\alpha_2})^2$. 
Let $\At_\pm = \mp \GG''(\cp^\pm)$. 
By Lemma \ref{lem:basic} (a), we have $A_\pm > 0$. 
By Taylor's theorem at $z=\cp^{\pm}$, there exists $\delta>0$ such that 
\beqq
    \GG(z) - \GG(\cp^{\pm})  =  \mp \frac{\At_\pm}{2}  (z-\cp^{\pm})^2 +  \mc E_{1,\pm} (z) 
\eeqq
where the function $\mc E_{1,\pm} (z)$ satisfies  
$|\mc E_{1,\pm} (z) | \le \frac{\At_\pm}{8} |z -  z^{\pm}_c|^2$ for $|z-\cp^\pm|\le \delta$. 
It is straightforward to verify by direct computation that there exists $\epsilon>0$ such that, for every $r,a\in\R$,  $\re(w^2) + \frac12 |w|^2 \le 3a^2 $ for every $w=r-(r+a)e^{\ii\theta}$ with $|\theta|<\epsilon$.
Any point $z \in \con^{b,L}_{\pm}$ can be written as
$z-\cp^\pm = \pm \mr r_0^{\pm} \mp (\mr r_0^{\pm} + bL^{-1/2})e^{\ii \theta}$
for some $\theta\in (-\pi, \pi)$, where $\mr r_0^{+} = -\cp^+$ and $\mr r_0^{-} = 1+\cp^-$.
Thus, after making $\delta>0$ smaller if needed, we have 
\begin{equation*}
\label{eq:dircomptdeldel}
    \re [(z-\cp^{\pm})^2] + \frac{1}{2}|z-\cp^\pm|^2
    \le 3b^2L^{-1}   
\end{equation*}
for all sufficiently large $L$ and all $z\in \con^{b,L}_{\pm}$ satisfying $|z-\cp^\pm| \le \delta$. 
Hence, 
\beq \label{eq:Goa} \begin{split}
    &\mp \re (\GG(z) - \GG(\cp^{\pm}))  = \frac{\At_\pm}{2} \re [(z-\cp^{\pm})^2] +   \re (\mp \mc E_{1,\pm}(z) )
    \le - \frac{\At_\pm}{8} |z-\cp^{\pm}|^2 +  \frac{3\At_\pm b^2}{2}L^{-1}
\end{split}
\eeq
for all sufficiently large $L$ and all $z \in \con^{b,L}_{\pm}$ satisfying $|z-\cp^\pm|\le \delta$.

Moreover, possibly after shrinking $\delta>0$, there exists $C>0$  
such that 
\beq \label{eq:HHHH}
        |\HH(z) - \HH(z^{\pm}_c) |\le C |z- \cp^{\pm}| \quad \text{for $|z- \cp^\pm|\le \delta$,}
\eeq
and 
\beq \label{eq:EEEE}
        |\EE_L(z) - \EE_L(\cp^{\pm}) |\le C \quad \text{for $|z-\cp^\pm|\le \delta$ and for every $L>0$.} 
\eeq

Fix $\epsilon\in (0, 1/2)$, and divide the circle $\con^{b, L}_{\lrt}$ into two parts:  
\beq \begin{split}
    &\con^{b, L, 1}_{\lrt}:= \con^{b, L}_{\lrt} \cap \{z\in \C : 0 \le |z-z_c^\pm| \le \delta \},\qquad 
    \con^{b, L, 2}_{\lrt}:= \con^{b, L}_{\lrt} \cap \{z\in \C :  |z-z_c^\pm| \ge \delta \} .
\end{split} \eeq
From \eqref{eq:Goa}, \eqref{eq:HHHH},  and \eqref{eq:EEEE}, there exist $L_0>0$ such that
\beq \label{eq:Fes1} 
    \mp \log \left| \frac{\ff_{L}(z)}{\ff_{L}(z^\pm_c)} \right|
    \le  - \frac{\At_\pm}{8} | z- z^\pm_c|^2 L +  C |z- z^\pm_c| L^{1/2}  +C 	
    \qquad \text{for $z\in \con^{b, L, 1}_{\lrt}$}
\eeq 
and for every $L \ge L_0$. 
Thus, $\mp \log \left| \frac{\ff_{L}(z)}{\ff_{L}(z^\pm_c)} \right|$ is uniformly bounded from the above on $ \con^{b, L, 1}_{\lrt}$ for all $L\ge L_0$. We also note that there exists $L_1>0$ such that 
\beqq
     - \frac{\At_\pm}{8} | z- z^\pm_c|^2 L +  C |z- z^\pm_c| L^{1/2}
     \le - \frac{\At_\pm}{16} | z- z^\pm_c|^2 L \le - \frac{\At_\pm}{16} L^{\frac{2\epsilon}{3}}
\eeqq
if $|z-z^\pm_c|\ge L^{-\frac12+ \frac{\epsilon}{3}}$, for all $L\ge L_1$. 

Now we consider the part $\con^{b, L,2}_{\lrt} $. 
Let $z_1^\pm$ denote the endpoint of the arc $\con^{b, L, 1}_{\lrt}$ in the upper half-plane. 
Because $\alpha_3 > (\sqrt{\alpha_1}+\sqrt{\alpha_2})^2$, 
the inequalities \eqref{eq:derGspm} holds. 
Thus, setting $s= |1+\cp^-|+ bL^{-1/2} $ or $s=|\cp^+|+ bL^{-1/2}$, 
the properties \eqref{eq:derG1the1} and \eqref{eq:derG1the2} hold for all $z \in \Sigma_{\pm}^{b,L,2}$ by Lemma \ref{lemma:G1decay}, for all large enough $L$. 
Hence, noting that $\re \GG(z)= \re \GG(\bar{z})$, we find that there exists $L_2 > 0$ such that
\beq \label{eq:Fesintermediate}
    \mp \re \left( \GG(z) - \GG(z_c^\pm)\right) \le \mp \re \left( \GG(z_1^\pm) - \GG(z_c^\pm)\right) \text{ for every } z\in \con^{b, L, 2}_{\lrt}
\eeq
for all $L \geq L_2$.
From \eqref{eq:Goa}, we see that 
$\mp \re \left( \GG(z^\pm_1) - \GG(z_c^\pm)\right) - \frac{3\At_\pm b^2}{2}L^{-1} \le - \frac{\At_\pm}{8} |z^\pm_1-z^\pm_c|^2 = - \frac{\At_\pm}{8} \delta^2$. 
Since $\con^{b, L,2}_{\lrt}$ lies in a compact subset of $\C\setminus \{-1, 0\}$ for all sufficiently large $L$, we find that  
there exist $L_3>0$ and $K>0$ such that 
\beq \label{eq:Fes2} 
    \mp \log \left| \frac{\ff_{L}(z)}{\ff_{L}(\mr z^\pm_c)} \right| \le - \frac{\At_\pm}{8} \delta^2  L + K L^{1/2} + \frac{3\At_\pm b^2}{2}
\eeq
for every $z$ in the arc $\con^{b, L,2}_{\lrt} $, whenever $L\ge L_3$. 

The estimates \eqref{eq:f21}, \eqref{eq:f2aub_xi}, \eqref{eq:f22}, and \eqref{eq:f2aub_eta} follow from the above computations.

\medskip 
$\bullet$ Suppose $\alpha_1 < 0$ and $\alpha_2,\alpha_3 > 0$. By Lemma \ref{lem:basic} (b), we have $A_+ > 0$. 
The proof of (b) for the $(+)$-case applies here as well, and thus the result follows.

\medskip
$\bullet$ Suppose $\alpha_2 < 0$ and $\alpha_1,\alpha_3 > 0$. 
by Lemma \ref{lem:basic} (c), we have $A_- > 0$. 
The proof of (a) for the $(-)$-case applies here as well, and thus the result follows.
\end{proof}

\section{Proof of Theorem \ref{thm:diagfluc}} \label{sec:proofofdiag}

To prove the theorem, we show that for every $m\ge 2$,
\beq \label{eq:diagpr}
    \prob \left( \bigcap_{i=1}^{m-1} \left\{  \frac{ \LPP( \td_i a \LL +  \xd_i\frac{a(\ell-a+b)}{\ell\sqrt{D}} \ab  \LL^{1/2}, \td_i b\LL -  \xd_i\frac{b(\ell+a-b)}{\ell\sqrt{D}} \ab  \LL^{1/2})- \td_i \lv \LL}{\ab \LL^{1/2}}  > \hd_i \right\}\, \bigg| \, \LPP(a\LL, b\LL)= \ell \LL \right)  
\eeq
converges, as $\LL\to \infty$, to 
\beqq
	\Ql:= \prob\left( \bigcap_{i=1}^{m-1} \left\{ \B_1(\td_i) - \left| \B_2(\td_i)-\xd_i \right|  \} > \hd_i\right\} \right)
\eeqq
for every 
\beq \label{eq:txhvecd}
    \mb\td=(\td_1, \cdots,\td_{m-1})\in (0,1)^{m-1}, 
    \quad \mb \xd=(\xd_1, \cdots,\xd_{m-1})\in \R^{m-1}, 
    \quad \mb \hd=(\hd_1, \cdots,\hd_{m-1})\in \R^{m-1}. 
\eeq
Here we use $\LL$ as the large parameter, whereas in the theorem we used $N$. 
Using the identity $\min(a,b)= \frac{a+b}{2}- \frac{|a-b|}{2}$, 
\beq \label{eq:limitofca1}
    \Ql= \prob \left( \bigcap_{i=1}^{m-1} \left\{  \min \{\sqrt{2} \cd_+ \B_+ (\td_i) - \xd_i, \sqrt{2} \cd_-\B_-(\td_i) +\xd_i\} > \hd_i\right\} \right) 
\eeq
where $\cd_\pm$ are defined in \eqref{eq:abdf} and $\B_{\pm}$ are independent Brownian bridges.

Since the function $\Ql$ in \eqref{eq:limitofca1} is continuous in $(\td_1,\dots,\td_{m-1}) \in (0,1)^{m-1}$, it suffices to show that \eqref{eq:diagpr} converges to \eqref{eq:limitofca1} when the times $\td_i$ are pairwise distinct. Indeed, if some of the times $\td_i$ coincide, the desired convergence follows by induction using Lemma \ref{lem:bootstrap}. More precisely, once the convergence has been established for pairwise distinct times, Lemma \ref{lem:bootstrap} (b) follows from the one-time case and the case of two distinct times, while Lemma \ref{lem:bootstrap} (a) provides the inductive step. Thus, it is enough to consider the case in which the times $\td_i$ are all distinct. By relabeling the indices if necessary, we may further assume that
$\td_1 < \cdots < \td_{m-1}$.

We now prove that \eqref{eq:diagpr} converges to \eqref{eq:limitofca1} under this assumption.

\bigskip

Fix an integer $m\geq 2$ and fix the numbers \eqref{eq:txhvecd}, assuming now that 
\beqq
    0  < \td_1 <  \cdots < \td_{m-1} < 1. 
\eeqq  
We use $L$ as the large parameter instead of $N$. For sufficiently large $\LL>0$, define  
\beqq \begin{split}
    &\bM_L = (M_{L,1}, \cdots, M_{L,{m}})\in \N^m, \qquad \bN_L = (N_{L,1}, \cdots, N_{L,{m}})\in \N^m, \qquad \bT_L = (T_{L,1}, \cdots, T_{L,{m}}) \in \R_+^m
\end{split} \eeqq
where\footnote{Recall that $\lceil s \rceil$ denotes the least integer greater than or equal to $s$.} for $i=1,  \cdots, m-1$, 
\beq\label{eq:j_scaling}
    M_{L,i} = \big\lceil \td_i  a \LL + \xd_i \frac{a( \lv -  a +  b)\ab}{ \lv \sqrt{D}}  \LL^{1/2}  \rceil, 
    \quad  N_{\LL, i} =  \big\lceil  \td_i b L - \xd_i \frac{ b( \lv +  a -  b)\ab}{ \lv \sqrt{D}} \LL^{1/2}  \big\rceil , 
    \quad T_{\LL, i} = \td_i \lv L + \hd_i \ab  \LL^{1/2}, 
\eeq
and 
\beqq
    M_{L,m}= \lceil aL \rceil, \qquad N_{L,m}=\lceil bL \rceil, \qquad T_{L,m}= \ell L. 
\eeqq
We also set $M_{L,0}=N_{L,0}=T_{L,0}=0$.

Recalling $\eqref{eq:LPPforreal}$, Proposition $\ref{prop:cond}$ implies that Theorem $\ref{thm:diagfluc}$ is proved if we show that 
\beq \label{eq:Qhatlimit}
    \lim_{L\to \infty} \frac{\QQ_m(\bM_L, \bN_L, \bT_L)}{\QQ_1(M_{L,m}, N_{L,m}, T_{L,m})} = \Ql.  
\eeq
Recall that 
\beq \label{eq:QQtempa} 
    \QQ_m(\bM_L, \bN_L, \bT_L) =  \sum_{\bn\in\N^m}\frac{1}{(\bn!)^2}  \QQ_m^{(\bn)}(\bM_L, \bN_L, \bT_L)
\eeq
where $\QQ_m^{(\bn)}$ is given by the formula \eqref{eq:Q_hat_n}. 
The following lemma shows that the term with $\bn = (1,\cdots,1)$ is responsible for the limit.  
For $L>0$, define the constant
\beq \label{eq:cKdef}
    \cK_L:= \left(\frac{\lv+ a- b+\sqrt{D}} {\lv+ a- b -\sqrt{D}}  \right)^{\lceil aL \rceil} 
    \left(\frac{\lv -a +b+\sqrt{D}} {\lv - a + b -\sqrt{D}}  \right)^{\lceil bL \rceil} 
    e^{- \sqrt{D} L}, 
\eeq
where $D$ is defined in \eqref{eq:Ddefn}.

\begin{lem}\label{lem:Q_hat_111_limit}
Set $\bone = (1,\cdots,1)$. We have 
\beqq
    \lim_{L \to \infty} \frac{2\pi  LD} {\sqrt{ab} \cK_L} \QQ_m^{(\bone)}(\bM_L, \bN_L, \bT_L) 
    = \Ql .
\eeqq
\end{lem}

The next result shows that the remaining terms in the sum are negligible by comparison.

\begin{lem}\label{lem:diag_error}
There exists a constant $c>0$ such that
\beqq
    \frac1{\cK_L} \sum_{\bn\in\N^m\setminus\{\bone\}}\frac1{(\bn!)^2}\abs{\QQ_m^{(\bn)}(\bM_L,\bN_L,\bT_L)} \leq e^{-c L}
\eeqq
for all sufficiently large $L>0$.
\end{lem}

The same analysis applies to the case when $m=1$. Note that in this case, $\Ql=1$. 

\begin{lem} \label{cor:Qhat1scalar}
We have
\beqq
    \lim_{L \to \infty} \frac{2\pi  LD} {\sqrt{ab} \cK_L}  \QQ_1(\lceil aL \rceil, \lceil bL \rceil, \lv L) = 1. 
\eeqq
\end{lem}

The above three lemmas complete the proof of Theorem \ref{thm:diagfluc}. 
We prove Lemmas \ref{lem:Q_hat_111_limit} and \ref{lem:diag_error} in Subsections \ref{sec:Q_hat_111_limit} and \ref{sec:diag_error}, respectively, following a preliminary discussion of some functions in Subsection \ref{sec:diag_basic}. 
Lemma \ref{cor:Qhat1scalar} is the special case  $m=1$ of these two lemmas, and we omit its proof. 

\subsection{Formula for $\ff_{L,i}$} \label{sec:diag_basic}

The quantity $\QQ_m^{(\bn)}$ in \eqref{eq:QQtempa} is expressed in terms of $\DD^{(\bn)}_{\bM_L, \bN_L, \bT_L}(\bz)$ from  \eqref{eq:D_hat_n}, which involves the functions 
\beqq
    \ff_{L,i}(z):= \frac{f_{M_{L,i}, N_{L,i},T_{L,i}}(z)}{f_{M_{L,i-1},N_{L,i-1},T_{L,i-1}}(z)}, 
    \qquad
    f_{M,N,T}(z)=  \frac{z^Ne^{Tz}}{(z+1)^M} = e^{- M\log(z+1) +N\log z+ Tz}. 
\eeqq
From \eqref{eq:j_scaling}, we have 
\beq\label{eq:j_scaling2}
    M_{L,i} =  \td_i  a \LL + \xd_i \frac{a( \lv -  a +  b)\ab}{ \lv \sqrt{D}}  \LL^{1/2}  + \epsilon_{L, i}
    \qquad  
    N_{\LL, i} =  \td_i b L - \xd_i \frac{ b( \lv +  a -  b)\ab}{ \lv \sqrt{D}} \LL^{1/2}  + \epsilon'_{L, i}
\eeq
for real numbers $ \epsilon_{L, i},  \epsilon_{L, i}'\in [0,1)$. 
Thus, 
\beq \label{eq:fLide}
    \ff_{L,i}(z) = e^{(t_i-t_{i-1}) \GG(z) L + \HH_i(z) L^{1/2} + \EE_{L, i}(z) } 
\eeq
where
\beq \label{eq:GGHHE} \begin{split}
    &\GG(z)= -a\log(z+1)+b\log z +\ell z, \\
    &\HH_i(z)= - (\xd_i-\xd_{i-1}) \frac{a( \lv -  a +  b)\ab}{ \lv \sqrt{D}} \log(z+1) - (\xd_i-\xd_{i-1}) \frac{ b( \lv +  a -  b)\ab}{ \lv \sqrt{D}} \log z + (\hd_i-\hd_{i-1}) \ab z, \\
    &\EE_{L, i}(z)= - \delta^1_{L,i}\log (z+1)+\delta^2_{L,i} \log z, 
\end{split} \eeq
with real numbers 
\beq
    \delta^1_{L,i}, \, \delta^2_{L,i} \in (-1, 1). 
\eeq
Here, we set $x_0=t_0=h_0=x_m=h_m=0$ and $t_m=1$. 
Note that $\fLi(z)$ are analytic except possibly  at $z=0$ and $z= -1$. 

We list a few properties:
\begin{itemize}
\item 
From Lemma \ref{lem:basic}, the critical points of $\GG$ are
\beq \label{eq:c3+-}
    z^{\pm} = \frac{- \lv+ a- b  \pm \sqrt{D}}{2\lv}, \qquad D=\lv^2-2(a+b)\lv + (a-b)^2, 
\eeq
and they satisfy the inequalities $-1<z^-<z^+<0$. 
\item 
It is straightforward to check  (see \eqref{eq:Atpm}) that 
\beq \label{eq:GGpaHp}
    \GG''(z^\pm)= \mp\frac{\sqrt{D}}{2ab}\left[(a+b)\lv- (a-b)^2 \pm (a-b) \sqrt{D}\right]  =  \mp 2\cd_\pm^2 \ab^2, 
\eeq
and that 
\beq \label{eq:GGpaHp2}
    \HH_i'(z^\pm)=  \ab \left( \pm (\xd_i-\xd_{i-1}) + \hd_i-\hd_{i-1} \right). 
\eeq
\item 
Since $\prod_{i=1}^m \ff_{L,i}(z)= f_{M_{L,m}, N_{L,m},T_{L,m}}(z)$, we see that 
\beq \label{eq:finzpm}
    \prod_{i=1}^m \frac{\fLi(z^-)}{ \fLi(z^+)}
    = \frac{f_{M_{L,m}, N_{L,m},T_{L,m}}(z^-)}{f_{M_{L,m}, N_{L,m},T_{L,m}}(z^+)} 
    = \cK_L, 
\eeq
where $\cK_L$ is defined in \eqref{eq:cKdef}. 
\item 
It is direct to see that 
\beq 
    \GG(z^+)-\GG(z^-)= \sqrt{D} + a \log \bigg( \frac{\lv+a-b-\sqrt{D}}{\lv+a-b+\sqrt{D}}\bigg) + b \log \bigg( \frac{\lv-a+b-\sqrt{D}}{\lv-a+b+\sqrt{D}}\bigg)  =\J(\lv), 
\eeq
using the notation from \eqref{eq:large_deviation_fn}. In particular,  $\J(\lv)>0$. 
Since $t_i-t_{i-1}>0$ for every $i$, there exists $L_0>0$ such that 
\beq \label{eq:frati}
    \frac{\ff_{L, i}(z^-)}{\ff_{L, i}(z^+)} \le e^{-\frac{1}{2} (t_i-t_{i-1}) (\GG(z^+)- \GG(z^-)) L}  \le  e^{- \frac{1}{2} \tau \J(\lv) L} , \qquad 
    \tau := \min_{1 \le i \le m} (t_i - t_{i-1}) > 0, 
\eeq
for every $L>L_0$ and $i=1, \cdots , m$. 
\end{itemize}

\subsection{Proof of Lemma \ref{lem:Q_hat_111_limit}} \label{sec:Q_hat_111_limit}

By Lemma \ref{lem:Q_hat_111}, 
\begin{equation} \label{eq:Qinpr41}
    \QQ_m^{(\bone)}(\bM_L,\bN_L, \bT_L) 
    = -\frac{1}{(2\pi \ii)^{2m}} \int_{\vec{\gamma}} \dd \bsxi \int_{\vec{\Gamma}}\dd \bseta \,\, 
    \Pi_{\bone}(\bs{\xi}, \bs{\eta}) \FF^{(\bone)}_{\bM_L,\bN_L,\bT_L}(\bs{\xi}, \bs{\eta})  .
\end{equation}
By Cauchy's theorem, we can deform the contours, without changing the value of the integral, to 
\beqq
    \vec{\gamma}= \con_{\lt}^{1,L}\times \cdots \times \con_{\lt}^{m,L}, 
    \qquad 
    \vec{\Gamma}=  \con_{\rt}^{1,L} \times \cdots \times \con_{\rt}^{m, L}. 
\eeqq
where $\con^{i, L}_{\lrt}$ are  the circles  in Lemma \ref{lem:asymptotics_f2} with $z^{\pm}_c = z^{\pm}$ and $b=i$. 
Note that all circles  $\con_{\lrt}^{i,L}$ are contained in the disk $\{z\in \C : |z|\le 2\}$ for all sufficiently large $L>0$. 
Fix $\epsilon\in (0,1/2)$ and define  $D_{L,-}^{\epsilon}=\{z\in \C : |z-z^-|\le L^{-\frac12+\frac{\epsilon}{3} }\}$ and $D_{L,+}^{\epsilon}=\{z\in \C : |z-z^+|\le L^{-\frac12+\frac{\epsilon}{3} }\}$. Set 
\beqq
    \vec{\gamma}^\epsilon= (\con_{\lt}^{1,L}\cap D_{L,-}^\epsilon) \times \cdots \times (\con_{\lt}^{m,L}\cap D_{L,-}^\epsilon), 
    \qquad
    \vec{\Gamma}^\epsilon =  (\con_{\rt}^{1,L}\cap D_{L,+}^\epsilon) \times \cdots \times (\con_{\rt}^{m, L}\cap D_{L,+}^\epsilon). 
\eeqq

Since $a,b>0$ and $\ell> (\sqrt{a}+\sqrt{b})^2$, Lemma \ref{lem:asymptotics_f2} (a) and (b) apply to $\ff_{L,i}(z)$. 
Thus, using \eqref{eq:finzpm}, 
\beqq 
    \frac{ |\FF^{(\bone)}_{\bM_L,\bN_L,\bT_L}(\bs{\xi}, \bs{\eta}) | }{\cK_L} 
    = \prod_{i=1}^m \left| \frac{\ff_{L,i}(\xi^{i})\ff_{L,i}(z^+)}{\ff_{L,i}(z^-)\ff_{L,i}(\eta^{i})} \right| 
    = O(e^{-cL^{2\epsilon/3}})
    \quad \text{uniformly for $(\bs{\xi}, \bs{\eta})\in (\vec{\gamma}\times \vec{\Gamma}) \setminus 
	(\vec{\gamma}^\epsilon \times \vec{\Gamma}^\epsilon)$.}
\eeqq
On the other hand, note that 
\beqq
    \dist(\con^{\lt}_{L,i}, \con^{\lt}_{L,j} )\ge L^{-\frac12}, \qquad
    \dist( \con^{\rt}_{L,i}, \con^{\rt}_{L,j})\ge L^{-\frac12}
\eeqq
for every $i\neq j$, and $\dist( \con^{\lt}_{L,i}, \con^{\rt}_{L,j})$ is bounded below by a constant for all $i,j$ and sufficiently large $L$. 
Since all circles are contained in the disk $\{z\in \C : |z|\le 2\}$ when $L$ is large enough, 
we find from \eqref{eq:Pi_111} that
\beqq
    \Pi_{\bone}(\bs{\xi}, \bs{\eta}) = O(L^{m-1})         
    \qquad  \text{uniformly for $(\bs{\xi}, \bs{\eta})\in \vec{\gamma}\times \vec{\Gamma}$}
\eeqq
as $L\to \infty$. 
Thus, 
\beq \label{eq:Pi1e}
    \frac{1}{\cK_L L^{m-1} }  
    \int_{ (\vec{\gamma}\times \vec{\Gamma}) \setminus (\vec{\gamma}^\epsilon \times \vec{\Gamma}^\epsilon) } \dd \bsxi  \, \dd \bseta \, 
    \Pi_{\bone}(\bs{\xi}, \bs{\eta})  \FF^{(\bone)}_{\bM_L,\bN_L,\bT_L}(\bs{\xi}, \bs{\eta})
    = O(e^{-cL^{2\epsilon/3}}).  
\eeq
 
We now evaluate the integral over $\vec{\gamma}^\epsilon \times \vec{\Gamma}^\epsilon$. Changing the variables as $\xi^{i}\mapsto u_i$ and $\eta^{i}\mapsto v_i$ given by
\beq \label{eq:xietacofL}
    \xi^{i} = z^- + \frac{u_i}{\ab L^{1/2}}, \qquad
    \eta^{i} = z^+ - \frac{v_i}{\ab L^{1/2}},
\eeq
we have
\beq \label{eq:nohattohat}
    \int_{ \vec{\gamma}^\epsilon \times \vec{\Gamma}^\epsilon} \dd \bsxi  \, \dd \bseta \, 
    \Pi_{\bone}(\bs{\xi}, \bs{\eta})  \FF^{(\bone)}_{\bM_L,\bN_L,\bT_L}(\bs{\xi}, \bs{\eta})
    = 
    \frac{(-1)^m}{(\ab^{2}L)^{m}}  \int_{\vec{\Sigma}^-_L\times \vec{\Sigma}^+_L}  \hat\Pi_{\bone}(\mathbf{u}, \mathbf{v})  \hat \FF_{\LL}(\mathbf{u}, \mathbf{v})
    \, \, \dd \mathbf{u} \, \dd\mathbf{v} 
\eeq
where $\hat\Pi_{\bone}(\mathbf{u}, \mathbf{v}) = \Pi_{\bone}(\bs{\xi}(\mb u), \bs{\eta}(\mb v))$ 
and $\hat \FF_{\LL}(\mathbf{u}, \mathbf{v})= \FF^{(\bone)}_{\bM_L,\bN_L,\bT_L}(\bs{\xi}(\mb u), \bs{\eta}(\mb v))$, 
and the contours $\vec{\Sigma}^-_L$ and $\vec{\Sigma}^+_L$ are the images of the contours $\vec{\gamma}^\epsilon$ and $\vec{\Gamma}^\epsilon$ under the change of variables. 
Noting that $z^+-z^-= \frac{\sqrt{D}}{\lv}$, we find from \eqref{eq:Pi_111} that 
\beqq  
    \hat \Pi_{\bone}(\mathbf{u}, \mathbf{v}) 
    =  \frac{\ell (\ab^2L)^{m-1}}{\sqrt{D}} \left[ \prod_{i=1}^{m-1} \frac{1}{(u_i-u_{i+1})(v_{i+1}-v_{i})} \right] \left( 1+  O(L^{-\frac12+\frac{\epsilon}{3}})  \right) 
    \qquad  \text{for $(\mathbf{u}, \mathbf{v})\in \vec{\Sigma}^-_L \times \vec{\Sigma}^+_L$.}
\eeqq
On the other hand, using  \eqref{eq:GGpaHp} and  \eqref{eq:GGpaHp2}, Lemma \ref{lem:asymptotics_f1} gives 
\beqq 
    \frac{\hat \FF_L(\mathbf{u}, \mathbf{v})}{\cK_L}
    = \FF(\mathbf{u}, \mathbf{v}) \left(1 + O(L^{-1/2+\epsilon}) \right)
    \qquad \text{for $(\mathbf{u}, \mathbf{v})\in \vec{\Sigma}^-_L\times \vec{\Sigma}^+_L$,}
\eeqq
where 
\beq \label{eq:mrF1}
    \FF(\mathbf{u}, \mathbf{v})
    = \prod_{i=1}^m \frac{e^{ \cd_-^2 (t_i-t_{i-1}) u_i^2 + [ - (\xd_i-\xd_{i-1}) + (\hd_i-\hd_{i-1})] u_i } }{e^{ - \cd_+^2 (t_i-t_{i-1}) v_i^2 -  [(\xd_i-\xd_{i-1}) + (\hd_i-\hd_{i-1})] v_i}}. 
\eeq
Because the function $\FF(\mathbf{u}, \mathbf{v})$ decays super-exponentially fast in each variable as it tends to infinity in any closed sector strictly contained in $\{z\in \C: \arg (z)\in (\frac{\pi}{4}, \frac{3\pi}{4})\cup (-\frac{3\pi}{4}, -\frac{\pi}4)\}$, extending the contours $\vec{\Sigma}^-_L$ and $\vec{\Sigma}^+_L$, and applying the dominated convergence theorem, we find that 
\beqq 
    \lim_{L\to\infty} 
    \frac{\sqrt{D}}{\ell (\ab^2L)^{m-1} \cK_L}  \int_{\vec{\Sigma}^-_L} \dd \mathbf{u}  
    \int_{\vec{\Sigma}^+_L}  \dd\mathbf{v}  \, \hat\Pi_{\bone}(\mathbf{u}, \mathbf{v})  \hat \FF_{\LL}(\mathbf{u}, \mathbf{v})
\eeqq
converges, as $L\to \infty$, to 
\beq \label{eq:Pi1c}
    (-1)^{m-1} \int_{\vec{\Sigma}^-} \dd \mathbf{u} 
    \int_{\vec{\Sigma}^+}  \dd \mathbf{v}  \, \left[ \prod_{i=1}^{m-1} \frac{1}{(u_{i+1}-u_{i})(v_{i+1}-v_{i})} \right]  \FF(\mathbf{u}, \mathbf{v})
\eeq
where $\vec{\Sigma}^-= \Sigma_1^{\lt}\times \cdots \times \Sigma_m^{\lt}$ and
$\vec{\Sigma}^+=\Sigma_1^{\rt}\times \cdots \times \Sigma_m^{\rt}$, 
with $\Sigma_i^{\lrt}= i+\ii \R$ for $1 \le i \le m$. 
All contours $\Sigma_i^{\pm}$ are oriented upwards.

From \eqref{eq:Qinpr41}, \eqref{eq:nohattohat}, and \eqref{eq:Pi1c}, we conclude that 
\beqq
    \lim_{L \to \infty} \frac{4\pi \cd_+\cd_- \ab^2 L \sqrt{D}}{\lv   \cK_L}   \QQ^{(\bone)}_m(\bM_L, \bN_L, \bT_L) 
    = \mathsf P_1 \mathsf P_2
\eeqq
where
\beqq
    \mathsf P_1:= \frac{\sqrt{4\pi} \cd_-}{(2\pi \ii)^m} \int_{\vec{\Sigma}^{\lt}} 
    \frac{\prod_{i=1}^m e^{ \cd_-^2 (t_i-t_{i-1}) u_i^2 + [ - (\xd_i-\xd_{i-1}) + (\hd_i-\hd_{i-1})]  u_i }}{\prod_{i=1}^{m-1}(u_{i+1} - u_{i})} \d \mathbf{u}
\eeqq
and
\beqq
    \mathsf P_2:= \frac{\sqrt{4\pi} \cd_+}{(2\pi \ii)^m}  \int_{\vec{\Sigma}^{\rt}}  
    \frac{\prod_{i=1}^m e^{\cd_+^2 (t_i-t_{i-1})  v_i^2 +  [(\xd_i-\xd_{i-1}) + (\hd_i-\hd_{i-1})] v_i}  }{\prod_{i=1}^{m-1}(v_{i+1} - v_{i})}\d \mathbf{v}. 
\eeqq
By Lemma \ref{lem:bridge}, 
\beqq
    \mathsf P_1= \prob\left( \bigcap_{i=1}^{m-1}\left\{ \sqrt{2}\cd_-\B_-(t_i ) > -\xd_i+\hd_i \right\} \right), 
    \qquad 
    \mathsf P_2= \prob\left( \bigcap_{i=1}^{m-1}\left\{ \sqrt{2}\cd_+\B_+(t_i ) > \xd_i+\hd_i \right\} \right)
\eeqq
for independent Brownian bridges $\B_+$ and $\B_-$. 
Noting $\frac{4\pi \cd_+\cd_- \ab^2 L\sqrt{D}}{\lv \cK_L} = \frac{2\pi  LD} {\sqrt{ab} \cK_L}$, we obtain Lemma \ref{lem:Q_hat_111_limit}. 

\subsection{Proof of Lemma \ref{lem:diag_error}} \label{sec:diag_error}

We take the $z_i$-contours in the formula \eqref{eq:Q_hat_n} of $\QQ^{(\bn)}_m(\bM_L, \bN_L, \bT_L)$ to be circles of fixed radii greater than $1$. For concreteness, we set them to be the circles of radii $2$ centered at the origin. Then, 
\beq \label{eq:QbdD}
    \left| \QQ^{(\bn)}_m(\bM_L, \bN_L, \bT_L) \right| 
    \le 3^{|\bn|} \max_{|z_i|=2, \, i=1, \cdots, m-1}
   \left| \DD^{(\bn)}_{\bM_L, \bN_L, \bT_L}(\bz)  \right|. 
\eeq
Consider now the formula \eqref{eq:D_hat_n} for $\DD^{(\bn)}_{\bM, \bN, \bT}(\bz)$. 
Recall the circles $\con^{b, L}_{\lrt}$ in Lemma \ref{lem:asymptotics_f2}. 
We take the contours as
$C_{1,\tn{left}} = \con^{0, L}_{\lt}$, $C_{1,\tn{right}} = \con^{0, L}_{\rt}$, and, 
for $i=2, \cdots, m$, 
\beqq
    C_{i,\tn{left}}^{\tn{in}} = \con^{-(i-1), L}_{\lt}, \quad 	C_{i,\tn{left}}^{\tn{out}} = \con^{i-1, L}_{\lt}, \quad 
    C_{i,\tn{right}}^{\tn{in}} = \con^{-(i-1), L}_{\rt}, \quad 	C_{i,\tn{right}}^{\tn{out}} = \con^{i-1, L}_{\rt} .
\eeqq
Since the lengths of all contours are at most $2\pi$ and $|z_i|=2$, we find that 
\beq \label{eq:DDpf}
    \left| \DD^{(\bn)}_{\bM_L, \bN_L, \bT_L}(\bz)  \right|
    \le 3^{2|\bn|} \max_{(\bs \xi, \bs \eta)\in \vec{C}_{\tn{left}} \times \vec{C}_{\tn{right}}} 
    | \Pi_{\bn}(\bs \xi, \bs \eta)|  | \FF^{(\bn)}_{\bM_L, \bN_L, \bT_L}(\bs \xi, \bs \eta) |, 
\eeq
where we set 
\beqq
    \vec{C}_{\tn{left}} = \left(C_{1,\tn{left}}\right)^{n_1} \times \left( C^{\tn{in}}_{2,\tn{left}} \cup C^{\tn{out}}_{2,\tn{left}}\right)^{n_2} \times \cdots \times \left( C^{\tn{in}}_{m,\tn{left}} \cup C^{\tn{out}}_{m,\tn{left}} \right)^{n_m}, 
\eeqq
and
\beqq
    \vec{C}_{\tn{right}} = \left(C_{1,\tn{right}}\right)^{n_1} \times \left( C^{\tn{in}}_{2,\tn{right}} \cup C^{\tn{out}}_{2,\tn{right}} \right)^{n_2} \times \cdots \times \left( C^{\tn{in}}_{m,\tn{right}} \cup C^{\tn{out}}_{m,\tn{right}} \right)^{n_m} .
\eeqq

Consider the term $\left| \Pi_{\bn}(\bs \xi, \bs \eta) \right|$ given in \eqref{eq:Pi_n}. 
By Hadamard's inequality, 
\beq \label{eq:Cayd1}
    |\K(\mb w | \mb w')| = \left| \det \left( \frac1{w_i-w_j'} \right) \right| 
    \le \prod_{i=1}^n \left( \sum_{j=1}^n \frac1{|w_i-w_j'|^2} \right)^{1/2}
    \le \frac{n^{n/2}}{d^n} 
\eeq
for every $\mb w=(w_1, \cdots, w_n)\in \C^n$ and $\mb w'=(w_1', \cdots, w_n')\in \C^n$, provided that 
$\min_{i,j\in \{1, \cdots, n\} } |w_i-w_j'|\ge d >0$. 
Thus, for every $(\bsxi, \bseta)\in \vec{C}_{\tn{left}} \times \vec{C}_{\tn{right}}$, using $d=L^{-1/2}$ in \eqref{eq:Cayd1}, 
\beqq \begin{split}
    \left| \K(\bseta^{1}|\bsxi^{1})
    \left[ \prod_{i=1}^{m-1} \K(\bsxi^{i}, \bseta^{i+1}|\bseta^{i}, \bsxi^{i+1}) \right]
    \K(\bsxi^{m}|\bseta^{m}) \right| 
    \le n_1^{\frac{n_1}{2}}\left[\prod_{i=1}^{m-1}(n_i+n_{i+1})^{\frac{n_i+n_{i+1}}{2}}\right]n_m^{\frac{n_m}{2}} L^{|\bn|}. 
\end{split} \eeqq 
Recall the basic bound of factorials: $n!\ge n^n e^{-n}$ for $n\in \N$. Thus, $n^n\le e^n n! \le 4^n n!$, and hence, $(a+b)^{a+b} \le 4^{a+b} (a+b)!\le 8^{a+b}a!b!$ for every $a,b\in \N$. Therefore, 
\beqq
    n_1^{\frac{n_1}{2}}\left[\prod_{i=1}^{m-1}(n_i+n_{i+1})^{\frac{n_i+n_{i+1}}{2}}\right]n_m^{\frac{n_m}{2}} 
    \le \frac{8^{|\bn|}}{2^{(n_1+n_m)/2}} \prod_{i=1}^m n_i !
    = \frac{8^{|\bn|}\bn !}{2^{(n_1+n_m)/2}}  .
\eeqq
Now, for all large enough $L$, the contours $C^{\text{in/out}}_{\text{left/right}}$ are contained a disk of radius $2$. Hence, 
\beqq
    |\ST(\bsxi^m|\bseta^m)|= \left| \sum_{k_m = 1}^{n_m}(\xi^{m}_{L,k_m} - \eta^{m}_{L,k_m}) \right| \leq 4 n_m . 
\eeqq
Thus, 
\beq \label{eq:Piest12}
    \max_{(\bs \xi, \bs \eta)\in \vec{C}_{\tn{left}} \times \vec{C}_{\tn{right}}} \left| \Pi_{\bn}(\bs \xi, \bs \eta) \right| 
    \leq 4 n_m \frac{8^{|\bn|}\bn!}{2^{(n_1+n_m)/2}}  L^{|\bn|} 
    \le 8^{|\bn|+1} \bn! L^{|\bn|} 
\eeq
for all sufficiently large $L$.  

By Lemma \ref{lem:asymptotics_f2} and using \eqref{eq:finzpm}, there exist constants $C>0$ and $L_0>0$ such that 
\beqq
    \left| \FF^{(\bn)}_{\bM_L, \bN_L, \bT_L}(\bs \xi, \bs \eta) \right|
    \le C^{2|\bn|} \prod_{i=1}^m  \left| \frac{\ff_{L, i}(z^-)}{ \ff_{L, i}(z^+)} \right|^{n_i}
    = C^{2|\bn|} \cK_L\prod_{i=1}^m  \left| \frac{\ff_{L, i}(z^-)}{ \ff_{L, i}(z^+)} \right|^{n_i-1}
\eeqq
for every $L\ge L_0$ and $(\bsxi, \bseta)$ on the contours. Thus, by \eqref{eq:frati}, 
\beq \label{eq:fpCj}
    \frac1{\cK_L} \left| \FF^{(\bn)}_{\bM_L, \bN_L, \bT_L}(\bsxi, \bseta) \right|
    \le C^{2|\bn|} e^{ -  \frac{ \tau}{2} (|\bn|-m)  \J(\lv) L}. 
\eeq

From \eqref{eq:QbdD}, \eqref{eq:DDpf}, \eqref{eq:Piest12}, and \eqref{eq:fpCj}, we find that there exist constants $C>0$ and $L_1>0$ such that 
\beqq
    \frac1{\cK_L} \left| \QQ^{(\bn)}_m(\bM_L, \bN_L, \bT_L) \right| 
    \le C^{|\bn|} L^{|\bn|} e^{-  \frac{ \tau}{2} (|\bn|-m)  \J(\lv) L} \bn!  
\eeqq
for all $L\ge L_1$ and $\bn\in \N^m$. 
Now, if $\bn \neq \bone$, then $|\bn|\ge m+1$ and thus $|\bn|- m \ge \frac1{m+1} |\bn|$. 	
Hence, 
\beq 
    \frac1{\cK_L} \left| \QQ^{(\bn)}_m(\bM_L, \bN_L, \bT_L) \right| 
    \le C^{|\bn|} L^{|\bn|} e^{- \frac{\tau \J(\lv) }{2(m+1)} |\bn| L} \bn! 
    \qquad \text{for $\bn \in \N^m\setminus \{\bone\}$.} 
\eeq
Therefore, there exist constants $c>0$ and $L_2>0$ such that 
\beq \label{eq:aassbba}
    \frac1{\cK_L} \sum_{\bn\in\N^m\setminus\{\bone\}}\frac1{(\bn!)^2}\abs{\QQ_m^{(\bn)}(\bM_L,\bN_L, \bT_L)}  \leq e^{ -cL}
\eeq
for all $L\ge L_2$. This proves Lemma \ref{lem:diag_error}.


\section{Proof of Theorem \ref{thm:offdiagflucsameside}}\label{sec:proofoffdiagflucsameside}

In the proof of Theorem \ref{thm:diagfluc}, the leading terms in the exponent of the functions $\ff_{L,i}(z)$ in \eqref{eq:fLide} are given by the same function $\GG(z)$ from \eqref{eq:GGHHE} for every $i = 1, \cdots, m$. However, for Theorem \ref{thm:offdiagflucsameside}, the leading functions depend on $i$. This implies that each function has different critical points and thus requires different contours. 
Because of the nesting structure of the original contours, 
which may not be in a suitable order, and the form of the rational function $\Pi_{\bn}(\bsxi, \bseta)$, it becomes necessary to account for the poles. Keeping track of the residues coming from these poles introduces technical difficulties in proving Theorem \ref{thm:offdiagflucsameside}. 
For these reasons, we prove Theorem \ref{thm:offdiagflucsameside} only for two-point distributions, leaving the problem of multi-point distribution convergence to future work. 

Fix $a,b > 0$ and $\lv>  \bar{\LPP}(a,b)$. 
Let $(x_1, y_1)$ and $(x_2, y_2)$ be distinct points in the square $(0,1)^2$ satisfying $\frac1{\slope} < \frac{y_1}{x_1}, \frac{y_2}{x_2} < 1$ or $1 < \frac{y_1}{x_1}, \frac{y_2}{x_2} < \slope$, where, recalling from \eqref{eq:critical_slope},  
\beqq
    \slope = \frac{\lv - a -b +\sqrt{D}}{\lv - a-b-\sqrt{D}} , \qquad D= \lv^2-2(a+b)\lv + (a-b)^2.
\eeqq
Since $\LPP(m, n) \eqind \LPP(n,m)$, it suffices to consider one of these cases. Without loss of generality, we assume  
\beq \label{eq:x1x2y1y2_sameside}
    \frac{1}{\slope} < \frac{y_1}{x_1}, \frac{y_2}{x_2} < 1.  
\eeq

Recall the function $h(x,y)$ from  \eqref{eq:mvxyshaded}. The points $(x_1, y_1)$ and $(x_2, y_2)$ satisfy one of the following three possibilities: $h(x_1,y_1) < h(x_2,y_2)$, $h(x_1,y_1) > h(x_2,y_2)$, or $h(x_1,y_1) = h(x_2,y_2)$. 
The case $h(x_1,y_1) = h(x_2,y_2)$ follows from the results of the other two cases and Lemma \ref{lem:bootstrap};  see Subsection \ref{sec:equalt}. 
The second case,  $h(x_1,y_1) > h(x_2,y_2)$, can be reduced to the first by relabeling the points.
Thus, we focus on the first case.

\subsection{Setup} 

The assumption 
\begin{equation} \label{eq:x2y2region00}
    h(x_1,y_1) < h(x_2,y_2) 
\end{equation}
is equivalent to
\begin{equation} \label{eq:x2y2region}
    (x_2-x_1)\hide{(\lv +a-b-\sqrt{D})} + \mu (y_2-y_1)\hide{(\lv -a + b +\sqrt{D})} > 0
    \quad \text{where $\mu := \frac{\lv - a+b+\sqrt{D}}{\lv+a-b-\sqrt{D}}$.}
\end{equation}
We use the notation
\beq \label{eq:hi12}
    h_i := h(x_i, y_i)=\frac12 \left[ x_i (\lv+a-b-\sqrt{D})  + y_i (\lv-a+b+\sqrt{D}) \right] 
\eeq
for $i=1,2$. Let $\mr r_1, \mr r_2 \in \R$ be fixed numbers as in Theorem \ref{thm:offdiagflucsameside}. 

We again use $L$ as the large parameter instead of $N$. 
For every $L>0$, set 
\beq \label{eq:MNTLiorj} \begin{split}
    M_{L,i} = \lceil x_i aL \rceil, \qquad N_{L,i} =\lceil y_i bL \rceil, \qquad T_{L,i} = h_iL + \sqrt{2} \ab \mr r_i L^{1/2} \qquad  \text{for $i=1,2$,}
\end{split} \eeq
and $M_{L,3} = \lceil a  L \rceil$, $N_{L,3} = \lceil b L\rceil$, $T_{L,3} = \lv L$, with  $\ab$ in \eqref{eq:abdf}. 
We also set $M_{L,0}=N_{L,0}=T_{L,0}=0$.
Note that $0<T_{L,1}<T_{L,2}<T_{L,3}$ for all large enough $L$, and we always assume that $L$ is large enough so that these inequalities hold. 
Thus, Proposition \ref{prop:cond} implies that 
\beqq
    \prob \left(\mc L(M_{L,i}, N_{L,i}) > T_{L,i}, \, i=1,2 \, \big| \, \mc L(aL, bL) = \lv L\right) = \frac{\QQ_3(\bM_L, \bN_L, \bT_L)}{\QQ_1(\lceil aL \rceil , \lceil bL \rceil,\lv L)}
\eeqq 
where $\bM_L = (M_{L,1}, M_{L,2}, M_{L,3})\in \N^3$, $\bN_L = (N_{L,1}, N_{L,2}, N_{L,3})\in \N^3$, and $\bT_L = (T_{L,1}, T_{L,2}, T_{L,3})\in \R_+^3$.
The goal is to prove that, with  $\cd_{+}$ as in \eqref{eq:abdf}, 
\beq \label{eq:Qlimitoffdiagsameside}
    \lim_{L\to \infty} 
    \frac{\QQ_3(\bM_L, \bN_L, \bT_L)}{\QQ_1(\lceil aL \rceil, \lceil bL \rceil,\lv L)} 
    =\prob \left[ \cd_{+} \B\left(\frac{\slope y_i- x_i}{\slope -1} \right) > \mr r_i,  i=1,2 \right],
\eeq
where $\B$ is a standard Brownian bridge.

From Proposition \ref{prop:cond}, 
\beq \label{eq:Qld}
    \QQ_3(\bM_L, \bN_L, \bT_L)
    = \sum_{\vecn\in \N^3} \frac{1}{(\vecn!)^2} \QQ_L^{(\vecn)}
    \qquad \QQ_L^{(\vecn)}:= \QQ^{(\vecn)}_3(\bM_L, \bN_L, \bT_L). 
\eeq
Noting that $m=3$,  we have 
\beq \label{eq:Qnold}
    \QQ^{(\bn)}_{L} 
    = \frac{(-1)^{|\bn|}}{(2\pi \ii)^2} \oint_{>1} \oint_{>1}
    \DD^{(\bn)}_{L}(z_1,z_2)  \prod_{i=1}^{2} \frac{(z_i+1)^{n_i-n_{i+1}-1}}{z_i^{n_{i+1}+1}} \dd z_i , 
\eeq
where $\DD^{(\vecn)}_L(z_1, z_2) := \DD^{(\vecn)}_{\bM_L, \bN_L, \bT_L}(\bz)$ is given by 
\beq \label{eq:Dvnz2} \begin{split}
    \DD^{(\vecn)}_L(z_1, z_2) = 
    &\frac{1 }{(2\pi \ii)^{2|\bn|}}  \prod_{i=2}^3 \prod_{k_i = 1}^{n_i}  \left[ \int_{C_{i, \tn{left}}^{\tn{in}}} \dd\xi_{k_i}^{i}  +  z_{i-1} \int_{C_{i,\tn{left}}^{\tn{out}}} \dd\xi_{k_i}^{i} \right] 
    \left[ \int_{C_{i,\tn{right}}^{\tn{in}}} \dd\eta_{k_i}^{i} + z_{i-1} \int_{C_{i,\tn{right}}^{\tn{out}}} \dd\eta_{k_i}^{i} \right] 	\\
    &\qquad\qquad \prod_{k_1 =1}^{n_1}  \left[ \int_{C_{1,\tn{left}}} \dd\xi_{k_1}^{1}  \right] \left[  \int_{C_{1,\tn{right}}} \dd\eta_{k_1}^{1}  \right]  
    \, \Pi_{\bn}(\bsxi, \bseta)  \FF_L^{(\bn)}(\bsxi, \bseta).
\end{split} \eeq
Here, recalling \eqref{eq:Pi_n} and \eqref{eq:defFF}, 
\beq \label{eq:Pi_n_m3} \begin{split}
    \Pi_{\vecn}(\bsxi, \bseta)  = \K(\bseta^{1}|\bsxi^{1})
    \left[ \prod_{i=1}^{2} \K(\bsxi^{i}, \bseta^{i+1}|\bseta^{i}, \bsxi^{i+1}) \right]
    \K(\bsxi^{3}|\bseta^{3}) \ST(\bsxi^{3}|\bseta^{3})
\end{split} \eeq
and, with the functions $f_{M,N,T}(z)=e^{ N\log z - M\log(z+1) + Tz} $ from \eqref{eq:f_expr}, 
\beq \label{eq:FLfL}
    \FF_L^{(\bn)}(\bsxi, \bseta)  := \FF^{(\bn)}_{\bM_L,\bN_L,\bT_L}(\bsxi, \bseta)
    =\prod_{i=1}^3 \prod_{k_i=1}^{n_i}\frac{\ff_{L,i}(\xi_{k_i}^{i})}{\ff_{L,i}(\eta_{k_i}^{i})}, 
    \qquad \ff_{L,i}(z):= \frac{f_{M_{L,i},N_{L,i},T_{L,i}}(z)}{f_{M_{L,i-1},N_{L,i-1},T_{L,i-1}}(z)} 
\eeq
with $\bsxi= (\bs \xi^{1}, \bsxi^{2}, \bs \xi^{3})$ and $\bseta= (\bs \eta^{1}, \bseta^{2}, \bs \eta^{3})$, where
$\bsxi^{i}= (\xi^{i}_1, \cdots, \xi^{i}_{n_i})$ and $\bseta^{i}=(\eta^{i}_1, \cdots, \eta^{i}_{n_i})$ for $i=1,2,3$. 
Note that for each $i=1, 2,3$, the functions $\FF_L^{(\bn)}(\bsxi, \bseta)$ and $ \Pi_{\vecn}(\bsxi, \bseta)$ are symmetric in the variables $\xi^{i}_1, \cdots, \xi^{i}_{n_i}$ and also symmetric in $\eta^{i}_1, \cdots, \eta^{i}_{n_i}$.

\medskip

The rational function $\Pi_{\bn}(\bsxi, \bseta)$ has simple poles at $\xi^{i}_j = \xi^{i+1}_k$ and $\eta^{i}_j = \eta^{i+1}_k$ for every $i,j,k$. 
We will need to consider the residues at these various poles. The resulting expressions involve new functions
\beq \label{eq:ftwoth}
    \ff_{L, 12}(z) := \ff_{L, 1}(z) \ff_{L,2}(z), \quad 
    \ff_{L, 23}(z) := \ff_{L, 2}(z) \ff_{L,3}(z), \quad 
    \ff_{L, 123}(z) := \ff_{L, 1}(z) \ff_{L,2}(z) \ff_{L,3}(z). 
\eeq
We observe that (cf. Subsection \ref{sec:diag_basic}) 
\beq \label{eq:fLstard}
    \ff_{L,*}(z) = e^{\mcG_{*}(z)L + \mcH_{*}(z)L^{1/2}+\EE_{L,*}(z)}, \qquad *\in \{1,2,3,12,23,123\}, 
\eeq 
where
\beq \label{eq:mcGdf} \begin{split}
    &\mcG_{1}(z) = -a x_1\log(z+1) + by_1\log z + h_1z, \\
    &\mcG_{2}(z) = -a (x_2-x_1)\log(z+1) + b (y_2-y_1) \log z+ (h_2-h_1)z, \\
    &\mcG_{3}(z) = -a (1-x_2)\log(z+1) + b (1-y_2) \log z + (\lv - h_2) z 
\end{split} \eeq
and
\beq \label{eq:Gsubide} \begin{split}
    &\mcG_{12}(z) := \mcG_{1}(z) + \mcG_{2}(z) =  -ax_2\log(z+1) + by_2\log z+ h_2z, \\
    &\mcG_{23}(z) := \mcG_{2}(z) + \mcG_{3} (z)=  -a(1-x_1)\log(z+1) + b(1-y_1)\log z + (\lv -h_1)z, \\
    &\mcG_{123} (z):= \mcG_{1}(z)+ \mcG_{2}(z) +\mcG_{3} (z)= -a\log(z+1) + b\log z +\lv z. 
\end{split} \eeq
The functions $\mcH_*$ are given by 
\beq \label{eq:mcHf} \begin{split}
    &\mcH_1(z) = \sqrt{2} \ab \mr r_1 z, \qquad
    \mcH_2(z) = \sqrt{2} \ab (\mr r_2 - \mr r_{1}) z, \qquad
    \mcH_3(z) = - \sqrt{2} \ab \mr r_{2} z,\\
    & \mcH_{12}(z) :=\mcH_{1}(z) +\mcH_{2}(z) = \sqrt{2} \ab \mr r_2 z, \qquad
    \mcH_{23}(z) :=\mcH_{2}(z) +\mcH_{3}(z)=  - \sqrt{2} \ab \mr r_{1} z, 
\end{split} \eeq
and $\mcH_{123}(z) :=\mcH_{1}(z) +\mcH_{2}(z) +\mcH_3(z)= 0$.  
Finally, the functions $\EE_{L,*}$ are 
\beq \begin{split}
    \EE_{L,*}(z) =  - \delta^1_{L,*}\log (z+1)+\delta^2_{L,*} \log z
\end{split} \eeq
with real numbers satisfying 
\beq
    \delta^1_{L,*}, \delta^2_{L,*} \in [-3,3]
\eeq
so that $\ff_{L,*}(z)$ are meromorphic with  possible poles only at $z=-1$ and $z=0$. 
All six functions $f_{L, *}$ in \eqref{eq:fLstard} are of the form \eqref{eq:fdef22}. In Lemma \ref{lem:fl_asym_sameside} in Subsection \ref{sec:cpa}, we check the applicability of Lemma \ref{lem:asymptotics_f2} to these functions. 

\subsection{Integrals} \label{sec:intetrals}

We will express the integrals appearing in $ \eqref{eq:Dvnz2} $ as sums of contributions from various residues. 
To this end, we introduce notation for the types of integrals that will appear in these expressions.

\begin{definition} \label{def:Bnset}
Define the set  
\beqq 
    \mc A_3 = \{ 1,2,3,12,23,123 \}.
\eeqq 
For $\vecn=(n_1, n_2, n_3)\in \N^3$, define $\listn$ to be the set of lists $\bsigma=\sigma_1\sigma_2\cdots \sigma_k$ of elements $\sigma_j\in \mc A_3$ such that, for each $i=1,2,3$, the total
number of times $i$ appears in any of $\sigma_1, \sigma_2, \cdots,  \sigma_k$ is equal to $n_i$. We denote
\beqq
    |\bsigma|=k \qquad \text{if $\bsigma= \sigma_1\sigma_2\cdots \sigma_k$.}
\eeqq
Let $\lista=\cup_{\vecn\in \N^3} \listn$. 
The type of a list $\bsigma\in \lista$ is the vector 
\beq \label{eq:psigma}
    \type(\bsigma) = (a_{123}, a_{12}, a_{23}, a_1, a_2, a_3) \in \N_0^6
\eeq
where $a_*$ is the number of $\sigma_i$ in $\bs \sigma = \sigma_1 \cdots \sigma_k$ such that $\sigma_i = *$ for each $* \in \mc A_3$.  
\end{definition}    

Typically, we write a list $\bsigma$ as 
\beq \label{eq:bsex}
    \bsigma = \alpha_1^{m_1} \alpha_2^{m_2}\alpha_3^{m_3} \cdots
\eeq 
where for each $i$, $\alpha_i$ and $\alpha_{i+1}$ are distinct elements of $\mc A_3$, and $\alpha^m$ denotes the list consisting of $m$ 
consecutive copies of  $\alpha$. 
If there is a possibility of confusion, we use parentheses for the numbers $12$, $23$, and $123$, writing them as $(12)$, $(23)$, or $(123)$, respectively. We also omit the superscript $1$ when $m_i=1$. 
For example, $3^2(23)^11^22^1=3^2(23) 1^22= 33(23)112$ is an element of $\listset_{(2,2,3)}$ of type $(0,0,1,2,1,2)$. 
Similarly, $3^2(12)^23=33(12)(12)3$ is also an element of $\listset_{(2,2,3)}$ but of type $(0,2,0,0,0,3)$. 
We have 
\beqq
    \listset_{(1,1,1)}=  \{ 123, 132, 213, 231, 312, 321, 1(23), (23)1, (12)3, 3(12), (123) \}. 
\eeqq

Note that if $\bsigma\in \listn$ has $\type(\bsigma)= \veca=(a_{123}, a_{12}, a_{23}, a_1, a_2, a_3)$, then 
\beq \label{eq:setBdf}
    a_1 + a_{12} + a_{123} = n_1, \quad  
    a_2 + a_{12}+a_{23}+a_{123} = n_2, \quad 
    a_3 + a_{23}+a_{123} = n_3. 
\eeq	


\begin{definition} \label{def:Pisigmaxi}
Let $\vecn\in \N^3$. For $\bsigma, \btau\in \listn$, define the functions 
\beq \label{eq:Pisigmaxi} \begin{split}
    \Pi^{\bsigma}_{\btau}(\bsxi, \bseta) 
    = &\K(\bea, \beot, \beo  | \bxa, \bxot, \bxo) \, 
    \K(\bxo, \betr, \bet | \beo,  \bxtr, \bxt) \, \\
    &\times \K( \bxot, \bxt, \ber | \beot, \bet, \bxr) \,
    \K( \bxa, \bxtr, \bxr | \bea, \betr,  \ber) 
    \ST( \bxa, \bxtr, \bxr | \bea, \betr,  \ber), 
\end{split} \eeq
and
\beq
    \FF_{L}^{\bsigma|\btau} (\bsxi, \bseta) 
    = \frac{\prod_{*\in \mc A_3} \prod_{i=1}^{a_*} \ff_{L, *} (\xi_i^*)}{\prod_{*\in \mc A_3} \prod_{i=1}^{b_*} \ff_{L, *} (\eta_i^*)}
\eeq
where 
\beqq   
    (a_{123}, a_{12}, a_{23}, a_1,a_2,a_3)=\type(\bsigma), \qquad 
    (b_{123}, b_{12}, b_{23}, b_1,b_2,b_3)=\type(\btau). 
\eeqq
Here,  $\bsxi = ( \bsxi^{123}, \bsxi^{12}, \bsxi^{23}, \bsxi^1, \bsxi^2, \bsxi^3)$ 
and $\bseta = ( \bseta^{123}, \bseta^{12}, \bseta^{23}, \bseta^1, \bseta^2, \bseta^3)$, 
with $\bsxi^{*} = (\xi^*_1,\cdots,\xi^*_{a_*}) \in \C^{a_*}$ and $\bseta^{*} = (\eta^*_1,\cdots,\eta^*_{b_*}) \in \C^{b_*}$
for each $* \in \mc A_3$. 
\end{definition}   

We note that $\Pi^{\bsigma}_{\btau}$ and $\FF_{L}^{\bsigma|\btau}$ depend only on $\type(\bsigma)$ and $\type(\btau)$, and not on the exact form of $\bsigma$ and $\btau$. 

The first $\K$ and the last $\K$ in \eqref{eq:Pisigmaxi} are  determinants of Cauchy matrices of sizes $n_1$ and $n_3$, respectively. 
The second $\K$ is the  determinant of a Cauchy matrix of size $n_2-n_1+a_1+b_1$, which is equal to $a_{1}+b_{23}+b_2$ and also to $b_1+a_{23}+a_2$ since $a_1-a_2-a_{23}=n_1-n_2=b_1-b_2-b_{23}$. 
Similarly, the third $\K$ is the  determinant of a Cauchy matrix of size $n_2-n_3+a_3+b_3$ which is equal to $a_{12}+a_{2}+b_{3}$ and also to $b_{12}+b_{2}+a_3$ since
$a_3-a_2-a_{12}=n_3-n_2=b_3-b_2-b_{12}$.  

We note that for each $*\in \mc A_3$, the functions $\Pi^{\bsigma}_{\btau}(\bsxi, \bseta)$ and $\FF_{L}^{\bsigma|\btau} (\bsxi, \bseta)$ are symmetric functions in the variables $\xi^*_1, \cdots, \xi^*_{a_*}$ and also symmetric in  $\eta^*_1,\cdots,\eta^*_{b_*}$. 

\medskip

Let $\bsigma\in \listn$ with $\type(\bsigma)=\veca$. 
For $\bsxi\in \C^{|\veca|}$, 
we define $\bsxi^{\bsigma} \in \C^{|\veca|}$ 
as follows. 
We can always write $\bsigma= \bsigma_1 \cdots \bsigma_r$, where each sub-list 
$\bsigma_i= 123^{s^{123}_i}12^{s^{12}_i}23^{s^{23}_i}1^{s^1_i}2^{s^2_i}3^{s^3_i}$ with $s_i^*\ge 0$ for every $i$ and for each superscript $*$. 
Define
\beqq
    \bsxi^{\bsigma} = (\bsxi_1, \cdots, \bsxi_r)
\eeqq
where, setting $k^{*}_i= s^*_1+\cdots +s^*_{i-1}$ with $k^*_{1}=0$, 
\beqq
    \bsxi_i= ( \underbrace{\xi_{k^{123}_i+1}^{123}, \cdots, \xi_{k^{123}_{i}+ s^{123}_i}^{123}}_{s^{123}_i}, 
    \underbrace{\xi_{k^{12}_i+1}^{12}, \cdots, \xi_{k^{12}_{i}+ s^{12}_i}^{12}}_{s^{12}_i}, \cdots, \underbrace{\xi_{k^{3}_i+1}^{3}, \cdots, \xi_{k^{3}_{i}+ s^{3}_i}^{3}}_{s^{3}_i}) 
\eeqq
for each $i$. 
For example, 
$\bsxi^{2(23)1}= (\xi^2_1,  \xi^{23}_1, \xi^1_1)$ and 
$\bsxi^{2(12)22}= (\xi^2_1 , \xi^{12}_1 , \xi^2_2, \xi^2_3)$. 

\begin{definition}
For $\bsigma, \btau\in \listn$ and $L>0$, define the integral 
\beq \label{eq:ingdef}
    \ing^{\bsigma}_{\btau}
    =  \frac{1}{(2\pi \ii)^{|\bsigma|+|\btau|}} \int \dd \bsxi^{\bsigma} \int \dd \bseta^{\btau} \, 
    \Pi^{\bsigma}_{\btau}(\bsxi, \bseta) \FF_{L}^{\bsigma|\btau} (\bsxi, \bseta) 
\eeq
where the contour for $\bsxi^{\bsigma}$ is a product of $|\bsigma|$ small circles centered at $-1$, nested from inside to outside, 
and the contour for $\bseta^{\btau}$ is a product of $|\btau|$ small circles centered at $0$, also nested from inside to outside. All circles are mutually disjoint. 
\end{definition}   

For example, both 
$\bsigma=22(123)$ and $\btau=2(12)32$ are elements of  $\listset_{(1,3,1)}$, with $\type(\bsigma)=(1,0,0,0,2,0)$ and $\type(\btau)=(0,1,0,0,2,1)$. We have 
\beqq
    \ing^{\bsigma}_{\btau}
    =  \frac{1}{(2\pi \ii)^{7}}  \int_{\gamma_1} \dd \xi^2_1  \int_{\gamma_2} \dd \xi^{2}_2 \int_{\gamma_3} \dd \xi^{123}_1 
    \int_{\Gamma_1} \dd \eta^2_1 \int_{\Gamma_2} \dd \eta^{12}_1 \int_{\Gamma_3} \dd \eta^3_1 \int_{\Gamma_4} \dd \eta^2_2
    \, \, 
    \Pi^{\bsigma}_{\btau}(\bsxi, \bseta) \FF_{L}^{\bsigma|\btau} (\bsxi, \bseta) 
\eeqq
where $\gamma_{1}, \gamma_{2},  \gamma_3$ are nested circles centered at $-1$ of radii $0<r_1<r_2<r_3<1/2$, 
$\Gamma_{1}, \Gamma_2, \Gamma_3, \Gamma_4$ are nested circles centered at $0$ of radii $0<R_1<R_2<R_3<R_4<1/2$; all circles are mutually disjoint. 
Note that in this example, we can take  $\gamma_1=\gamma_2$ without changing the integral since the integrand is analytic at $\xi_1^2=\xi_2^2$. However, we cannot take $\Gamma_3$ and $\Gamma_4$ to be the same since $\eta^3_1=\eta^2_2$ is a pole of $\Pi^{\bsigma}_{\btau}(\bsxi, \bseta)$, arising from the third $\K$ term in \eqref{eq:Pisigmaxi}. 

\medskip

In terms of the notations introduced above, \eqref{eq:Qnold} can be written as follows. 

\begin{lem} \label{lem:Dbaseint2}
For $\vecn=(n_1, n_2, n_3)\in \N^3$ and $L>0$, 
\beqq
    \QQ^{(\bn)}_{L} 
    = (-1)^{|\bn|}  \sum_{i=0 \vee (2n_2-n_1+1)}^{2n_2} \sum_{j=0 \vee (2n_3-n_2+1)}^{2n_3}  
    \left[ \oint_{>1}  \frac{(z_1+1)^{n_1-n_{2}-1}}{z_1^{n_{2}-i+1}} \frac{\dd z_1}{2\pi \ii} \right]  
    \left[ \oint_{>1} \frac{(z_2+1)^{n_2-n_{3}-1}}{z_2^{n_{3}-j+1}} \frac{\dd z_2}{2\pi \ii} \right]
    \alpha_{ij} 
\eeqq
where $\alpha_{ij}$ is a sum of $\binom{2n_2}{i}\binom{2n_3}{j}$ terms,  each of the form $\ing^{\bsigma}_{\btau}$ with $\bsigma, \btau\in \listn$ specified by
\beq \label{eq:baseint2}
    \bsigma= 3^{n_{31}} 2^{n_{21}} 1^{n_1} 2^{n_{22}} 3^{n_{32}}, 
    \qquad     \btau = 3^{n_{31}'} 2^{n'_{21}} 1^{n_1} 2^{n_{22}'} 3^{n_{32}'} 
\eeq
for $n_{21}, n_{22}, n_{31}, n_{32},n_{21}', n_{22}', n_{31}', n_{32}'\in \N_0$ subject to the constraints  
\beqq \begin{split}
	n_{21}+n_{22}=n_{21}'+n_{22}'=n_2, \quad 
	n_{31}+n_{32}= n_{31}'+n_{32}'= n_3, \quad
	n_{22} + n'_{22} = i, \quad
	n_{32} + n_{32}' = j.
\end{split} \eeqq
\end{lem}

\begin{proof}
Multiplying out the formula \eqref{eq:Dvnz2}, and using the invariance of $\Pi_{\bn}(\bsxi, \bseta)$ and $\FF_L^{(\bn)}(\bsxi, \bseta)$ under suitable permutations of the variables,
we find that
\beq
    \DD^{(\vecn)}_L(z_1, z_2)
    = \sum_{i=0}^{2n_2} \sum_{j=0}^{2n_3} \alpha_{ij} z_1^i z_2^j 
\eeq
where $\alpha_{ij}$ is a sum of $\binom{2n_2}{i}\binom{2n_3}{j}$ terms of the form $\ing^{\bsigma}_{\btau}$ with $\bsigma, \btau\in \listn$ as in \eqref{eq:baseint2}. 
Inserting this formula into \eqref{eq:Qnold}, we obtain the result since 
$\oint \frac{(z+1)^{n-n'-1}}{z^{n'-i+1}} \dd z =0$ whenever $i\le 2n'-n$.
\end{proof}

In Subsection \ref{sec:deform}, the integrals $\ing^{\bsigma}_{\btau}$ with $\bsigma, \btau$ as in \eqref{eq:baseint2} will be further rewritten in terms of $\ing^{\bsigma}_{\btau}$ with other choices of $\bsigma, \btau$, which are more amenable to the application of the method of steepest descent. 

\subsection{Critical point analysis} \label{sec:cpa}

We now consider the critical points of the functions $\GG_{*}$ in \eqref{eq:mcGdf} and  \eqref{eq:Gsubide}, and these will be used  in the asymptotic evaluation of the integrals $\ing^{\bsigma}_{\btau}$. 
All these functions are special cases of the function given in \eqref{eq:Gqmcd} below. 

As before, let $a, b > 0$ and $\lv  > \bar{\LPP}(a,b)$ be fixed. We set $D =\lv^2-2(a+b)\lv +(a-b)^2$, $\slope= \frac{\lv-a-b+\sqrt{D}}{\lv-a-b-\sqrt{D}}$, and $ \mu = \frac{\lv-a+b+\sqrt{D}}{\lv +a-b - \sqrt{D}}$.

\begin{lem} \label{lem:general_crit_behavior} 
For every $X, Y \in \R \setminus \{0\}$ satisfying $X+\mu Y \ne 0$, the critical points of the function
\beq \label{eq:Gqmcd}
    \GG_{X,Y}(z) = -aX \log(1+z) + bY\log z + \frac12 \left[ X(\lv+a-b-\sqrt{D}) + Y(\lv-a+b+\sqrt{D}) \right] z
\eeq
are
\beq \label{eq:mean_levelcurve_slope} 
    \pq = -\frac{\frac{Y}{X}}{\frac1{\mu}+ \frac{Y}{X}} \quad \text{and} \quad   
    \pp = - \frac{\lv-a+b-\sqrt{D}}{2\lv}. 
\eeq
Furthermore, they satisfy the following properties: 
\begin{enumerate}[(a)]
    \item If $ \frac{Y}{X}<- \frac1{\mu}$, then $\pq < -1 < \pp < 0$. 
    \item If $-\frac{1}{\mu}<\frac{Y}{X}<0$, then  $-1< \pp < 0 < \pq$.
    \item If $0< \frac{Y}{X}<\frac{1}{\slope}$, then $ -1 < \pp < \pq < 0$. 
    \item If $\frac{Y}{X}> \frac{1}{\slope}$, then $ -1 < \pq < \pp < 0$. 
\end{enumerate}
Furthermore, 
\beq \label{eq:Aratio}
    \frac{ \GG_{X,Y}''(\pp)}{ \GG_{1,1}''(\pp)} 
    = \frac{\slope Y- X}{\slope-1}. 
\eeq
\end{lem}

\begin{proof}
This lemma in principle follows from Lemma \ref{lem:basic}. However, the explicit computation of the critical points can be tedious using the formula \eqref{eq: z_pm}. Instead we proceed as follows. We have  
\beqq
    \GG_{X,Y}'(z) 
    = X \left[-\frac{a}{1+z} + \frac{\lv+a-b-\sqrt{D}}{2} \right] + Y \left[ \frac{b}{z} + \frac{\lv-a+b+\sqrt{D}}{2} \right]. 
\eeqq
Noting that $\frac{1}{1+\pp} = \frac{\lv+a-b-\sqrt{D}}{2a}$ and $\frac{1}{\pp} = -\frac{\lv-a+b+\sqrt{D}}{2b}$, we see  that $ \GG_{X,Y}'(\pp) = 0$. 
Also, since
\beqq
    \GG_{X,Y}'(z) 
    = \frac{hz^2 + (bY-aX+h)z + bY}{z(1+z)}
    \qquad 
    \text{where $h:=\frac{1}{2} [X(\lv+a-b-\sqrt{D}) + Y(\lv-a+b+\sqrt{D})]$,}
\eeqq
we find that $\GG$ has another critical point given by  
\beqq
    \pq = \frac{bY}{h \pp}= -\frac{\frac{Y}{X}}{\frac1{\mu}+ \frac{Y}{X}}. 
\eeqq

The conditions $\lv > \bar{\LPP}(a,b) =  (\sqrt{a} + \sqrt{b})^2$ and $a,b>0$ imply that $\lv \pm (a-b) \pm \sqrt{D}>0$ for all four choices of signs.
These inequalities show that $-1<\pp<0$. 
Properties (a) and (b) follow directly from the formula of $\pq$. 
Properties (c) and (d) are obtained by noting that $\pq=\pp$ if and only if $\frac{Y}{X}=\frac1{\slope}$. 
Finally, \eqref{eq:Aratio} can be derived by direct computation.  
\end{proof}

Recalling \eqref{eq:hi12}, the functions $\mcG_{1}, \mcG_{2}, \mcG_{3}, \mcG_{12}, \mcG_{23}$, and $\mcG_{123}$ are all equal to the function $\GG_{X,Y}(z)$ in \eqref{eq:Gqmcd} with the parameters
\beq \label{eq:XYpa}
    (X,Y)= (x_1,y_1), \quad (x_2-x_1, y_2-y_1), \quad (1-x_2,1-y_2), \quad (x_2, y_2), \quad (1-x_1, 1-y_1), \quad (1,1), 
\eeq
respectively. 
The ordering of their critical points depends on the relative positions of $(x_1, y_1)$ and $(x_2, y_2)$. 
Set
\beq\label{def:wandRregion}
    \w= (x_1, y_1, x_2, y_2).  
\eeq
From the assumptions \eqref{eq:x1x2y1y2_sameside} and \eqref{eq:x2y2region}, $\w$ belongs to the region 
\beq \label{eq:rg}
    \rg = \left\{ (x_1,y_1, x_2,y_2) \in (0,1)^4: \frac{1}{\slope} < \frac{y_1}{x_1}, \frac{y_2}{x_2}  < 1 \text{ and } (x_2-x_1) +\mu (y_2-y_1) > 0 \right\}. 
\eeq
Since $(x_1, y_1)\in (0,1)^2$ satisfies $\frac1{\slope} < \frac{y_1}{x_1} < 1$, we find that  
\beq \label{eq:xoyospl}
    -\frac1{\mu}< 0<\frac1{\slope}<\frac{y_1}{x_1} <1<  \frac{1-y_1}{1-x_1} . 
\eeq	
The six numbers above divide the real line into seven intervals. 
We define the following seven disjoint sub-regions of $\rg$ according to which interval the value $\frac{y_2-y_1}{x_2-x_1}$ belongs to: 
\beq \label{eq:all_regions} \begin{split}
    &\rgo := \left\{ \w \in \rg:  \frac{y_2-y_1}{x_2-x_1}< - \frac1{\mu} \right\}, 
    \quad\quad\quad\quad \rgt := \left\{ \w \in \rg :  \frac{y_2-y_1}{x_2-x_1} >\frac{1-y_1}{1-x_1} \right\}, \\
    &\rgth := \left\{ \w \in \rg :1< \frac{y_2-y_1}{x_2-x_1}< \frac{1-y_1} {1-x_1} \right\}, 
    \quad \rgf := \left\{ \w \in \rg :  \frac{y_1}{x_1}< \frac{y_2-y_1}{x_2-x_1}<1  \right\}, \\
    &\rgfi := \left\{ \w \in \rg :\frac1{\slope}<  \frac{y_2-y_1}{x_2-x_1}< \frac{y_1} {x_1}  \right\}, 
    \quad\quad \rgs := \left\{ \w \in \rg : 0< \frac{y_2-y_1}{x_2-x_1}< \frac1{\slope}  \right\}, \\
    &\rgse :=  \left\{ \w \in \rg: -\frac1{\mu}< \frac{y_2-y_1}{x_2-x_1} < 0 \right\}.
\end{split}  \eeq
The region $\rg$ is the union of these seven regions and a finite collection of hypersurfaces. 
Since $(x_2-x_1) +\mu (y_2-y_1) > 0 $ for $\w \in \rg$, the regions $\rgo$ and $\rgse$ simplify to  
\beq \label{eq:signofxy}
    \rgo = \left\{ \w \in \rg:  x_2<x_1\, \, y_2>y_1 \right\}
    \qquad \text{and} \qquad \rgse =  \left\{ \w \in \rg: x_2> x_1\, \, y_2<y_1 \right\}.
\eeq
For $\w \in \rgt \cup \cdots \cup \rgs$, we have $x_2 - x_1>0$ and $y_2 -y_1>0$. 

We now state the following result regarding the orderings of the critical points.

\begin{lem} \label{cor:criticalptofGlo} 
Let $\pp=- \frac{\lv-a+b-\sqrt{D}}{2\lv}$, as in Lemma \ref{lem:general_crit_behavior}. 
Let  $z_{*}^{-} \le z_{*}^{+}$ denote the critical points of $\mcG_*$  for each $* \in \mc A_3$. Then, for every $\w\in \rgo\cup\cdots \cup \rgse$, 
\beq \label{eq:cpotha}
    z^+_{1} = z^+_3 = z^+_{12}  = z^+_{23}  = z_{123}^+ = \pp. 
\eeq
Furthermore, the following results hold: 
\begin{enumerate}[(a)]
    \item If $\w  \in \rgo$, then $ z^-_2<-1 < z^-_{23} < z^-_{3} < z_{123}^- < z^-_{12} < z^-_1 < \pp = z^+_2 < 0$.
    \item If $\w  \in \rgt$, then $-1 < z^-_2 < z^-_{23} < z^-_{3} < z_{123}^- < z^-_{12} < z^-_1 < \pp =z^+_2 < 0$.
    \item If $\w  \in \rgth$, then $-1 < z^-_3 < z^-_{23} < z^-_{2} < z_{123}^- < z^-_{12} < z^-_1 < \pp = z^+_2 < 0$.
    \item If $\w  \in \rgf$, then $-1 < z^-_3 < z^-_{23} < z_{123}^- < z^-_{2} < z^-_{12} < z^-_1 < \pp = z^+_2 < 0$.
    \item If $\w  \in \rgfi$, then $-1 < z^-_3 < z^-_{23} < z_{123}^- < z^-_{1} < z^-_{12} < z^-_2 <  \pp = z^+_2 < 0$.
    \item If $\w  \in \rgs$, then $-1 < z^-_3 < z^-_{23} < z_{123}^- < z^-_{1} < z^-_{12} < z^-_2 = \pp  < z^+_2 < 0$.
    \item If $\w \in \rgse$, then $-1 < z^-_3 < z^-_{23} < z_{123}^- < z^-_{1} < z^-_{12} < z^-_2 = \pp < 0 < z^+_2$.
\end{enumerate}
\end{lem}

\begin{proof} 
From Lemma \ref{lem:general_crit_behavior}, one of the critical points is $\pp$ for every $\GG_*$; this critical point does not depend on $\w$. 
Since $\w\in \rg$, we have 
\beq \label{eq:xi1i1}
    \frac1{\slope}< \frac{y_i}{x_i}< 1< \frac{1-y_i}{1-x_i}
\eeq
for both $i=1,2$. 
Thus, the parameters $(X,Y)$ in \eqref{eq:XYpa} satisfy $Y, X>0$ and $\frac{Y}{X}>\frac1{\slope}$ for every $*\neq 2$. 
Hence, Lemma \ref{lem:general_crit_behavior} (d) implies that $\pp$ is the larger critical point of $\mcG_*$ for $*\neq 2$, implying \eqref{eq:cpotha}, and the smaller critical points are: 
\beqq \begin{split}
    z^-_1= -\frac{\frac{y_1}{x_1}}{\frac{1}{\mu} + \frac{y_1}{x_1}}, \quad
    z^-_3= -\frac{\frac{1-y_2}{1-x_2}}{\frac{1}{\mu}+ \frac{1-y_2}{1-x_2}}, \quad 
    z^-_{12}= -\frac{\frac{y_2}{x_2}}{\frac{1}{\mu} + \frac{y_2}{x_2}}, \quad 
    z^-_{23}= -\frac{\frac{1-y_1}{1-x_1}}{\frac{1}{\mu} + \frac{1-y_1}{1-x_1}}, \quad
    z^-_{123}= -\frac{1}{\frac{1}{\mu} + 1} . 
\end{split} \eeqq

Note that the function $r\mapsto -\frac{r}{\frac{1}{\mu} + r}$ is decreasing in $r$. 
Hence, using \eqref{eq:xi1i1} we find that for every $\w\in \rgo\cup\cdots \cup \rgse$,  
\beqq
    \max\{ z_{23}^-, z_3^-\} < z^-_{123} < \min \{ z_1^-, z_{12}^-\} .
\eeqq
From \eqref{eq:xoyospl} and the definitions of the sub-regions, we see that
\beqq
    \text{$\frac{y_2-y_1}{x_2-x_1}< \frac{y_1}{x_1}$ for $\w \in \rgo \cup \rgfi\cup\rgs\cup \rgse$} \quad\text{and}\quad \text{$\frac{y_2-y_1}{x_2-x_1}> \frac{y_1}{x_1}$ for $\w \in \rgt \cup \rgth \cup \rgf$.}
\eeqq 
Using \eqref{eq:signofxy}, these inequalities imply that $(y_2-y_1)x_1>(x_2-x_1)y_1$ if $\w \in \rgo\cup \cdots \cup \rgf$,  
and $(y_2-y_1)x_1<(x_2-x_1)y_1$ if $\w \in \rgfi\cup\rgs\cup \rgse$. Hence, $\frac{y_2}{x_2}>\frac{y_1}{x_1}$ in the former case,  and $\frac{y_2}{x_2}<\frac{y_1}{x_1}$ in the latter case, implying that $z_{12}^-<z_1^-$ in the former case and $z_{12}^->z_1^-$ in the latter case. 
Similarly, also from the definitions of the sub-regions, we see that
\beqq
    \text{$\frac{y_2-y_1}{x_2-x_1}< \frac{1-y_1}{1-x_1}$ for $\w \in \rgo \cup (\rgth\cup\cdots\cup \rgse)$,} \quad\text{and}\quad \text{$\frac{y_2-y_1}{x_2-x_1}> \frac{1-y_1}{1-x_1}$ for $\w \in \rgt$.}
\eeqq   
These inequalities imply that $\frac{1-y_2}{1-x_2}<\frac{1-y_1}{1-x_1}$ for $\w\in \rgo\cup \rgt$, and $\frac{1-y_2}{1-x_2}>\frac{1-y_1}{1-x_1}$ for $\w\in \rgth\cup\cdots\cup \rgse$. Thus, $z_{3}^->z_{23}^-$ in the former case and $z_{3}^-<z_{23}^-$ in the latter case. 
Thus, we have obtained all inequalities for the critical points that do not involve $z_2^{\pm}$. 

We now consider $z_2^{\pm}$. 
Let $w_c=-\frac{\frac{y_2-y_1}{x_2-x_1}}{\frac{1}{\mu} + \frac{y_2-y_1}{x_2-x_1}}$.
From Lemma \ref{lem:general_crit_behavior}, we have 
\beqq \begin{split}
    &z_2^-=w_c, \quad z_2^+= \pp \qquad \text{for $\w\in \rgo\cup\cdots\cup \rgfi$;} \qquad
    z_2^-=\pp, \quad z_2^+=w_c \qquad \text{for $\w\in \rgs\cup \rgse$.}\\
\end{split} \eeqq
Furthermore, $w_c<-1$ for $\w\in \rgo$, $w_c>0$ for $\w\in \rgse$,  $w_c\in (\pp, 0)$ for $\w\in \rgs$, and $w_c\in (-1, \pp)$ for $\w\in \rgt\cup \cdots\cup \rgfi$. 
Since $\frac{y_2-y_1}{x_2-x_1}>\frac{1-y_1}{1-x_1}$ for $\w\in \rgt$ and $1<\frac{y_2-y_1}{x_2-x_1}<\frac{1-y_1}{1-x_1}$ for $\w\in \rgth$, 
we find that $w_c<z_{23}^-$ in the former case and $w_c\in (z_{23}^-, z_{123}^-)$ in the latter case. 
Finally, observe that if $x_2-x_1>0$, then
$\frac{y_1}{x_1} <\frac{y_2-y_1}{x_2-x_1}$ if and only if $y_1x_2< y_2x_1$, which is equivalent to $\frac{y_2}{x_2} < \frac{y_2-y_1}{x_2-x_1}$. 
Thus,  $\frac{y_2}{x_2} <\frac{y_2-y_1}{x_2-x_1}<1$ for $\w\in \rgf$, and  $\frac{y_2-y_1}{x_2-x_1} < \frac{y_2}{x_2}$ for $\w\in \rgfi$, so that $w_c\in (z_{123}^-, z_{12}^-)$ in the former case and $w_c>z_{12}^-$ in the latter. 
This completes the proof. 
\end{proof}

The functions $\ff_{L,*}$ are of the form \eqref{eq:fdef22}. 
We conclude this subsection by verifying that Lemma \ref{lem:asymptotics_f2} is applicable to $\ff_{L,*}$. 

\begin{lem} \label{lem:fl_asym_sameside}
If $\w \in \rgo \cup \cdots \cup \rgse$, then Lemma \ref{lem:asymptotics_f2} \ref{eq:propertya}  and \ref{eq:propertyb} apply to $\ff_{L,*}(z)$ for every $* \in \{1, 3, 12, 23, 123\}$.
Moreover, if $\w \in \rgo \cup \cdots \cup \rgs$, then Lemma~\ref{lem:asymptotics_f2} \ref{eq:propertyb} applies to $\ff_{L,2}(z)$; if $\w \in \rgt \cup \cdots \cup \rgse$, then Lemma~\ref{lem:asymptotics_f2} \ref{eq:propertya}  applies to $\ff_{L,2}(z)$. 
\end{lem}

\begin{proof}
The parameters $S_* = (\al_1,\al_2,\al_3)$  in \eqref{eq:Gaph} for $\ff_{L,*}$ are given by 
\beqq\begin{split}
    &S_1 = (ax_1,by_1,h_1), \quad S_2 = (a(x_2-x_1),b(y_2-y_1),h_2-h_1), \quad S_3 = (a(1-x_2),b(1-y_2),\ell-h_2),\\
    &S_{12} = (ax_2,by_2,h_2), \quad S_{23} = (a(1-x_1),b(1-y_1),\ell-h_1), \quad S_{123} = (a,b,\ell), 
\end{split} \eeqq
where the $h_i$ are given in \eqref{eq:hi12}. In all cases, $\al_3 = h(\al_1/a,\al_2/b)$, where (see \eqref{eq:mvxyshaded}) 
\beqq
    \mv(x,y)=  \frac12 \left[ (\lv +a-b-\sqrt{D})x + (\lv -a+b+\sqrt{D})y \right].
\eeqq

Consider $*\neq 2$. 
Since $x_1, y_1, x_2, y_2\in (0,1)$, we have $0< h_1, h_2<\lv$, and thus $\al_1, \al_2, \al_3>0$ in all relevant cases.   
For $x, y>0$, the arithmetic-geometric mean inequality implies that 
\beqq
    \hat{h}(x,y):= (\ell-a-b-\sqrt{D})x + (\ell-a-b+\sqrt{D})y > 4\sqrt{abxy} 
    \qquad  \text{if}\quad  \frac{y}{x} \neq \frac1{\slope}. 
\eeqq 
Thus, $h(x,y)= \frac12 \hat{h}(x,y) + ax+by> (\sqrt{ax}+\sqrt{by})^2$ if $ \frac{y}{x} \neq \frac1{\slope}$. Therefore,  
\beqq
    \alpha_3  = h(\al_1/a,\al_2/b) > (\sqrt{\alpha_1} + \sqrt{\alpha_2})^2 \quad \text{if} \quad \frac{a\al_2}{b\al_1}\neq \frac1{\slope} .
\eeqq
By definition (cf. \eqref{eq:xoyospl}), $ \frac{y_1}{x_1}, \frac{y_2}{x_2}, \frac{1-y_1}{1-x_1}, \frac{1-y_2}{1-x_2} > \frac1{\slope}$, and thus $\frac{a\al_2}{b\al_1}\neq \frac1{\slope}$. 
Hence, $\alpha_3 > (\sqrt{\alpha_1} + \sqrt{\alpha_2})^2$ for all  $*\neq 2$. 
Therefore, Lemma~\ref{lem:asymptotics_f2} \ref{eq:propertya} and \ref{eq:propertyb} hold.  

Consider $*=2$. 
Then, by assumption \eqref{eq:x2y2region00}, $\alpha_3= h_2-h_1 >0$. 
From  \eqref{eq:signofxy}, $\al_1, \al_2>0$ if $\w\in \rgt \cup \cdots \cup \rgs$. 
From \eqref{eq:xoyospl} and \eqref{eq:all_regions}, we also see that $\frac{y_2-y_1}{x_2-x_1}\neq \frac1{\slope}$ for all $\w\in \rgo \cup \cdots \cup \rgse$. 
Hence, the argument of the previous paragraph applies, and  we find that Lemma~\ref{lem:asymptotics_f2} \ref{eq:propertya}  and \ref{eq:propertyb} apply to $f_{L,2}$ if $\w\in \rgt \cup \cdots \cup \rgs$. 
Additionally, from  \eqref{eq:signofxy}, we have $\al_1<0,\al_2>0$ for $\w \in \rgo$, and $\al_1>0,\al_2<0$ for $\w \in \rgse$. 
Therefore, if $\w \in \rgo$, then Lemma~\ref{lem:asymptotics_f2} \ref{eq:propertyb} holds, and if $\w \in \rgse$, then Lemma~\ref{lem:asymptotics_f2} \ref{eq:propertya} holds. 
\end{proof}

\subsection{Leading contributions} \label{sec:leading}

In this section, we evaluate several integrals, which we will later show provide the leading contributions to the limit.

Consider $\QQ_{L}^{(1,1,1)}$ in \eqref{eq:Qld}. From Lemma \ref{lem:Q_hat_111} with $m=3$, we have  
\beqq 
    \QQ^{(1,1,1)}_{L}
    = -\frac{1}{(2\pi \ii)^{6}} \int_{\gamma_1}\dd\xi^1  \int_{\gamma_2} \dd \xi^2 \int_{\gamma_3} \dd \xi^3 
    \int_{\Gamma_1} \dd \eta^1 \int_{\Gamma_2} \dd \eta^2 \int_{\Gamma_3} \dd \eta^3 
    \,\, \Pi_{(1,1,1)}(\bs\xi, \bseta)  \FF^{(1,1,1)}_{\bM_L,\bN_L,\bT_L}(\bs\xi, \bseta) 
\eeqq
where $\gamma_1$, $\gamma_2$, $\gamma_3$ are small circles around the point $z = -1$, nested from inside to outside, and $\Gamma_1$, $\Gamma_2$, $\Gamma_3$ are small circles around the point $z = 0$, also nested from inside to outside; all circles are disjoint. Using the notations of Subsection \ref{sec:intetrals}, 
\beqq \begin{split} 
    \Pi_{\bone}(\bs{ \xi}, \bs{\eta})
    &= \K_1(\eta^1  | \xi^1) \K_2(\xi^1, \eta^2 | \eta^1, \xi^2)  \K_2( \xi^2, \eta^3 | \eta^2, \xi^3 ) \K_1( \xi^3 | \eta^3) \ST_1( \xi^3 | \eta^3) 
    = \Pi^{123}_{123}(\bs{ \xi}, \bs{\eta}) 
\end{split} \eeqq
and $\FF^{(\bone)}_{\bM_L,\bN_L,\bT_L}(\bs\xi, \bseta)=\FF^{123|123}_{L}(\bs\xi, \bseta)$. Thus,  
\beqq
    \QQ^{(1,1,1)}_{L}= -\ing^{123}_{123}. 
\eeqq

Note that $\Pi^{123}_{123}(\bs{ \xi}, \bs{\eta})$ is a rational function, and $\FF^{123|123}_{L}(\bs\xi, \bseta)$ is analytic except possibly at $-1$ and $0$. 
We deform the contours and repeatedly apply Cauchy's residue theorem. For example, for fixed $\eta^1, \eta^2, \eta^3$, if we swap the contours for  $\xi^2$ and $\xi^3$, then 
due to the pole at $\xi^2=\xi^3$,
\beqq \begin{split}
    \int_{\gamma_1}\dd\xi^1  \int_{\gamma_2} \dd \xi^2 \int_{\gamma_3} \dd \xi^3 
    \,\, \Pi^{123}_{123}(\bs{ \xi}, \bs{\eta}) \FF^{123|123}_{L}(\bs\xi, \bseta) = \int_{\gamma_1}\dd\xi^1  \int_{\gamma_2} \dd \xi^3  \int_{\gamma_3} \dd \xi^2 
    \,\,  \Pi^{123}_{123}(\bs{ \xi}, \bs{\eta}) \FF^{123|123}_{L}(\bs\xi, \bseta) \\
    + 2\pi \ii \int_{\gamma_1}\dd\xi^1  \int_{\gamma_2} \dd \xi^{23} 
    \,\,  \Pi^{1(23)}_{123}(\xi^1, \xi^{23},  \bseta) \FF^{123|123}_{L}(\xi^1, \xi^{23}, \xi^{23}, \bseta), 
\end{split} \eeqq
since, recalling the Cauchy determinant formula \eqref{eq:Cauchydd}, 
\beqq \begin{split}
    &\lim_{\xi^2\to \xi^3} (\xi^2-\xi^3) \Pi^{123}_{123}(\bs{ \xi}, \bs{\eta}) \bigg|_{\xi^3=\xi^{23}} \\
    &= \K(\eta^1  | \xi^1) \K(\xi^1, \eta^2 | \eta^1, \xi^{23}) \left[ \lim_{\xi^2\to \xi^3} (\xi^2-\xi^3) \K( \xi^2, \eta^3 | \eta^2, \xi^3 ) \right]_{\xi^3=\xi^{23}} \K( \xi^{23} | \eta^3) \ST( \xi^{23} | \eta^3) \\
    &= - \K(\eta^1  | \xi^1) \K(\xi^1, \eta^2 | \eta^1, \xi^{23})  \K( \eta^3 | \eta^2) \K( \xi^{23} | \eta^3) \ST( \xi^{23} | \eta^3) 
    = - \Pi^{1(23)}_{123}(\xi^1, \xi^{23}, \bs{\eta}). 
\end{split} \eeqq
Thus, we see that $\ing^{123}_{123} = \ing^{132}_{123} + \ing^{1(23)}_{123}$. 
Similar computations yield the following basic rules. 

\begin{table}[h]
\centering
\renewcommand{\arraystretch}{1.8}
\setlength{\tabcolsep}{8pt}
\begin{tabular}{|c|c|c|}
\hline
$\bsigma_\pm$ & Rule & $\bsigma$ \\
\hline
$\bsigma_+ = \cdots 12 \cdots$ or $\bsigma_- = \cdots 21 \cdots$
&
$\ing^{\bsigma_\pm}_{\btau}
= \ing^{\bsigma_\mp}_{\btau}
\mp (-1)^{b_1+a_{23}}\ing^{\bsigma}_{\btau}$
&
$\bsigma = \cdots (12) \cdots$ \\
\hline
$\bsigma_+ = \cdots 23 \cdots$ or $\bsigma_- = \cdots 32 \cdots$
&
$\ing^{\bsigma_\pm}_{\btau}
= \ing^{\bsigma_\mp}_{\btau}
\mp (-1)^{b_2+a_{12} + b_{12}}\ing^{\bsigma}_{\btau}$
&
$\bsigma = \cdots (23) \cdots$ \\
\hline
$\bsigma_+ = \cdots 1(23) \cdots$ or $\bsigma_- = \cdots (23)1 \cdots$
&
$\ing^{\bsigma_\pm}_{\btau}
= \ing^{\bsigma_\mp}_{\btau}
\mp (-1)^{b_1}\ing^{\bsigma}_{\btau}$
&
$\bsigma = \cdots (123) \cdots$ \\
\hline
$\bsigma_+ = \cdots (12)3 \cdots$ or $\bsigma_- = \cdots 3(12) \cdots$
&
$\ing^{\bsigma_\pm}_{\btau}
= \ing^{\bsigma_\mp}_{\btau}
\mp (-1)^{b_2+b_{12}}\ing^{\bsigma}_{\btau}$
&
$\bsigma = \cdots (123) \cdots$ \\
\hline
\end{tabular}
\caption{Some rules for swapping a pair of left contours. Here, $\type(\sigma_\pm)=\veca$ and $\type(\tau)=\vecb$.}
\label{table:sigma_rule}
\end{table}

\begin{table}[h] \label{table2}
\centering
\renewcommand{\arraystretch}{1.8}
\setlength{\tabcolsep}{8pt}
\begin{tabular}{|c|c|c|}
\hline
$\btau_\pm$ & Rule & $\btau$ \\
\hline
$\btau_+ = \cdots 12 \cdots$ or $\btau_- = \cdots 21 \cdots$
&
$\ing^{\sigma}_{\tau_\pm}
= \ing^{\sigma}_{\tau_\mp}
\pm (-1)^{a_1+b_{23}}\ing^{\sigma}_{\tau}$
&
$\btau = \cdots (12) \cdots$ \\
\hline
$\btau_+ = \cdots 23 \cdots$ or $\btau_- = \cdots 32 \cdots$
&
$\ing^{\sigma}_{\tau_\pm}
= \ing^{\sigma}_{\tau_\mp}
\pm (-1)^{a_2+a_{12} + b_{12}}\ing^{\sigma}_{\tau}$
&
$\btau = \cdots (23) \cdots$ \\
\hline
\end{tabular}
\caption{Some rules for swapping a pair of right contours. Here, $\type(\sigma)=\veca$ and $\type(\tau_\pm)=\vecb$.}
\label{table:tau_rule}
\end{table}

Additionally, since the integrand of \eqref{eq:ingdef} is analytic at $\xi^1_i=\xi^3_j$ and $\xi^2_i=\xi^{(23)}_j$ for all $i,j$, we also find the following rules. 
\begin{table}[h] \label{table2}
\centering
\renewcommand{\arraystretch}{1.8}
\setlength{\tabcolsep}{8pt}
\begin{tabular}{|c|c|}
\hline
$\bsigma_\pm$ & Rule \\
\hline
$\bsigma_+ = \cdots 13 \cdots$ or $\bsigma_- = \cdots 31 \cdots$
&
$\ing^{\sigma_\pm}_{\tau}
= \ing^{\sigma_\mp}_{\tau}$ \\
\hline
$\bsigma_+ = \cdots 2(23) \cdots$ or $\bsigma_- = \cdots (23)2 \cdots$
&
$\ing^{\sigma_\pm}_{\tau}
= \ing^{\sigma_\mp}_{\tau}$ \\
\hline
\end{tabular}
\caption{Additional rules for swapping a pair of left contours.}
\label{table:sigma_rule_nores}
\end{table}

Applying these rules, we obtain the following result. 

\begin{lem} \label{lem:QQ111list}
We have $\QQ^{(1,1,1)}_{L}= -\ing^{123}_{123}$,  and the following expressions for $\ing^{123}_{123}$ hold: 
\begin{enumerate}[(a)]
    \item $\ing^{123}_{123}=\ing^{231}_{123}+\ing^{3(12)}_{123}+\ing^{(123)}_{123}$.
    \item $\ing^{123}_{123}=\ing^{321}_{123}+\ing^{(23)1}_{123}+\ing^{3(12)}_{123}+\ing^{(123)}_{123}$.
    \item $\ing^{123}_{123} =\ing^{312}_{123}+\ing^{(23)1}_{123}+\ing^{(123)}_{123}$. 
    \item $\ing^{123}_{123} = \ing^{312}_{213}+ \ing^{(23)1}_{213}+ \ing^{(123)}_{213}-  \ing^{312}_{(12)3}- \ing^{(23)1}_{(12)3}+\ing^{(123)}_{(12)3}$. 
\end{enumerate}
\end{lem}

\begin{proof} 
Using the rules listed in Table \ref{table:sigma_rule}, we obtain the following identities:
\begin{equation}
\begin{split}
(i) \ &\ing^{123}_{123} = \ing^{213}_{123} + \ing^{(12)3}_{123}, \quad (ii) \ \ing^{(12)3}_{123} = \ing^{3(12)}_{123} + \ing^{(123)}_{123}, \quad (iii) \ \ing^{231}_{123} = \ing^{321}_{123} + \ing^{(23)1}_{123} \\
(iv) \ &\ing^{123}_{123} = \ing^{132}_{123} + \ing^{1(23)}_{123}, \quad (v) \ \, \ing^{1(23)}_{123} = \ing^{(23)1}_{123} + \ing^{(123)}_{123}.
\end{split} 
\end{equation}
The rules in Table \ref{table:tau_rule} imply
\begin{equation}
\begin{split}
 (vi) \, \ing^{312}_{123} = \ing^{312}_{213} - \ing^{312}_{(12)3} \qquad (vii) \ \ing^{(23)1}_{123} = \ing^{(23)1}_{213} - \ing^{(23)1}_{(12)3}, \qquad  (viii) \ \ing^{(123)}_{123} = \ing^{(123)}_{213} + \ing^{(123)}_{(12)3}.
\end{split} 
\end{equation}
The lemma is proved since (see Table \ref{table:sigma_rule_nores})
\begin{equation}
\begin{split}
    &(i) + \left[\ing^{213}_{123} = \ing^{231}_{123} \right] + (ii) \implies (a),  \\
    &(i) + \left[\ing^{213}_{123} = \ing^{231}_{123} \right] + (ii) + (iii) \implies (b),  \\
    &(iv) + \left[\ing^{132}_{123} = \ing^{312}_{123} \right] + (v)  \implies (c),  \\
    &(c) + (vi)  + (vii) + (viii)\implies (d).  \\
\end{split}
\end{equation}
\end{proof}

From Lemma \ref{cor:criticalptofGlo}, we see that the integrals on the right-hand side are suitable for applying the method of steepest descent in the following cases: (a) for $\w \in \rgo \cup \rgt$; (b) for $\w \in \rgth \cup \rgf$; (c) for $\w \in \rgfi$; and (d) for $\w \in \rgs \cup \rgse$. 

For $\w \in \rgo \cup \cdots \cup \rgfi$, we evaluate $\ing^{(123)}_{123}$  asymptotically in Proposition \ref{prop:leading} below, and show in Proposition \ref{cor:incor} that other integrals are subdominant.
Furthermore, in Corollary \ref{cor:errorsum}, we show that $\QQ^{(\bn)}_{L}$, $\bn\neq (1,1,1)$, give subdominant contributions. 
The integral $\ing^{(123)}_{123}$ is responsible for the limit in Theorem \ref{thm:offdiagflucsameside} when $\w\in \rgo\cup\cdots\cup \rgfi$. 

\medskip

For $\w \in \rgs \cup \rgse$, we will show that 
$\ing^{(123)}_{(12)3}$ is the largest among the integrals appearing in Lemma \ref{lem:QQ111list}(d). 
It turns out that another term from $\QQ^{(1,2,1)}_{L}$ also contributes to the limit in Theorem \ref{thm:offdiagflucsameside}. 
By Lemma \ref{lem:Q_hat_121}, we have 
\begin{equation} \label{eq:Qinpr411}
   \QQ^{(1,2,1)}_{L}  = \ing^{3122}_{1223} + \ing^{1223}_{3122}. 
\end{equation}
We write these as follows, similarly to the previous lemma. 

\begin{lem} \label{lem:QQ121list}
We have 
\beqq \begin{split}
    \QQ^{(1,2, 1)}_{L}
    = &2\ing^{3122}_{2231} -  2\ing^{3122}_{2(12)3} - 2\ing^{(23)12}_{2231} - 2\ing^{(123)2}_{2231} - 
    2\ing^{3122}_{23(12)} + 4\ing^{(23)12}_{23(12)} - 4\ing^{(123)2}_{23(12)} \\
    & - 2\ing^{3122}_{2(23)1} - 4\ing^{(23)12}_{2(23)1} - 4\ing^{(123)2}_{2(23)1} - 2\ing^{3122}_{(23)(12)} - 4\ing^{(23)12}_{(23)(12)} + 4\ing^{(123)2}_{(23)(12)}. 
\end{split} \eeqq 
\end{lem}

\begin{proof}
Each of the following identities is obtained from applying a single rule in Tables \ref{table:sigma_rule} and \ref{table:tau_rule} once or twice. 
First, $\ing^{1223}_{3122} = \ing^{1223}_{3221} - 2\ing^{1223}_{32(12)}$.
Each term can then further expressed as 
\begin{equation}
\begin{split}
    \ing^{1223}_{3221} = \ing^{1223}_{2231} - 2\ing^{1223}_{2(23)1},  \qquad \ing^{1223}_{32(12)} = \ing^{1223}_{23(12)} + \ing^{1223}_{(23)(12)} . 
\end{split}
\end{equation}
Next, 
\begin{equation}
\begin{split}
     &\ing^{1223}_{2231} = \ing^{1322}_{2231} -2 \ing^{12(23)}_{2231}, \qquad \quad \,\ing^{1223}_{2(23)1} = \ing^{1322}_{2(23)1} + 2\ing^{12(23)}_{2(23)1},\\
     &\ing^{1223}_{23(12)} = \ing^{1322}_{23(12)} -2 \ing^{12(23)}_{23(12)} , \qquad \ing^{1223}_{(23)(12)} = \ing^{1322}_{(23)(12)} + 2\ing^{12(23)}_{(23)(12)}.
\end{split}
\end{equation} 
Finally, also using Table \ref{table:sigma_rule_nores}, we find 
\begin{equation}
\begin{split}
     &\ing^{1322}_{2231} = \ing^{3122}_{2231}, \qquad \qquad \ \ing^{12(23)}_{2231} =\ing^{1(23)2}_{2231} = \ing^{(23)12}_{2231} + \ing^{(123)2}_{2231}, \\
     &\ing^{1322}_{2(23)1} = \ing^{3122}_{2(23)1}, \qquad \quad \, \ing^{12(23)}_{2(23)1} =\ing^{1(23)2}_{2(23)1} = \ing^{(23)12}_{2(23)1} + \ing^{(123)2}_{2(23)1}, \\
     &\ing^{1322}_{23(12)} = \ing^{3122}_{23(12)}, \qquad \quad \, \ing^{12(23)}_{23(12)} =\ing^{1(23)2}_{23(12)} = \ing^{(23)12}_{23(12)} - \ing^{(123)2}_{23(12)}, \\
     &\ing^{1322}_{(23)(12)} = \ing^{3122}_{(23)(12)}, \qquad \ing^{12(23)}_{(23)(12)} =\ing^{1(23)2}_{(23)(12)} = \ing^{(23)12}_{(23)(12)} - \ing^{(123)2}_{(23)(12)}. 
\end{split}
\end{equation}
Combining these identities and 
$\ing^{3122}_{1223} = \ing^{3122}_{2213} - 2\ing^{3122}_{2(12)3}$, we obtain the result. 
\end{proof}

When $\w\in \rgs \cup \rgse$, the integral $\ing^{(123)2}_{(23)(12)}$  has the same leading-order behavior as $\ing^{(123)}_{(12)3}$. 
We will see that two integrals $\ing^{(123)}_{(12)3}$ and $\ing^{(123)2}_{(23)(12)}$ are responsible for the limit in Theorem \ref{thm:offdiagflucsameside} when $\w\in \rgs \cup \rgse$.  

Later, we need to find similar representations for $\ing^{\bsigma}_{\btau}$ when $\bsigma$ and $\btau$ have arbitrary length. In this case, repeated applications of the basic rules become extremely tedious. In Subsection \ref{sec:deform}, we present a unified treatment of the deformations of the integrals. 

\bigskip

In the remainder of this subsection, we evaluate the integral $\ing^{(123)}_{123}$ asymptotically when  $\w \in \rgo \cup \cdots \cup \rgfi$, and the integrals $\ing^{(123)}_{(12)3}$ and $\ing^{(123)2}_{(23)(12)}$ when $\w\in \rgs \cup \rgse$. 

\begin{prop} \label{prop:leading} 
For $L>0$, define the constant (which is the same as \eqref{eq:cKdef})
\beq \label{eq:cKdef2}
    \cK_L= \left(\frac{\lv+ a- b+\sqrt{D}} {\lv+ a- b -\sqrt{D}}  \right)^{\lceil aL \rceil} 
    \left(\frac{\lv -a +b+\sqrt{D}} {\lv - a + b -\sqrt{D}}  \right)^{\lceil bL \rceil} 
    e^{- \sqrt{D} L}. 
\eeq
As $L\to \infty$, the following hold:
\beqq  \begin{split}
    &\ing^{(123)}_{123} 
    = - \frac{\sqrt{ab}\cK_L}{2\pi  LD }  \prob \left[ \cd_{+} \B\left(\frac{\slope y_i- x_i}{\slope -1} \right) > \mr r_i,  i=1,2 \right] (1+o(1)) \qquad  \text{for $\w \in \rgo \cup \cdots \cup \rgfi$,}\\
    & \ing^{(123)}_{(12)3} 
    = - \frac{\sqrt{ab}\cK_L}{2\pi  LD} \prob \left[\cd_{+} \B\left(\frac{\slope y_2- x_2}{\slope -1} \right) > \mr r_2  \right] (1+o(1)) \qquad \text{for $\w \in \rgs \cup \rgse$,}\\
    &\ing^{(123)2}_{(23)(12)} 
    = - \frac{\sqrt{ab}\cK_L}{2\pi  LD}
    \prob \left[ \cd_{+} \B\left(\frac{\slope y_1- x_1}{\slope -1} \right) \le \mr r_1, \, \cd_{+} \B\left(\frac{\slope y_2- x_2}{\slope -1} \right) > \mr r_2 \right]
    (1+o(1)) \quad \text{for $\w \in \rgs \cup \rgse$.}
\end{split} \eeqq
\end{prop}

\begin{proof}
The analysis of the three integrals is similar. In all cases, the critical points $z_{123}^\pm$ of the function 
$\GG_{123}(z)= -a \log(z+1)+b\log z + \lv z$ play a distinguished role. We use the notations (see Lemma \ref{lem:general_crit_behavior})
\beqq
    \pp^-:=z_{123}^- = -\frac{1}{\frac1{\mu}+1}= - \frac{\lv-a+b+ \sqrt{D}}{2\lv}, \qquad
    \pp^+:= z_{123}^+ = - \frac{\lv-a+b-  \sqrt{D}}{2\lv}.
\eeqq
They satisfy $-1<\pp^-<\pp^+<0$. 
It is straightforward to check (cf. \eqref{eq:Atpm}) that 
\beq \label{eq:GG123pp}
    \mp \GG_{123}''(z_c^{\pm}) = \frac{\sqrt{D}}{2ab} \left[ (a+b)\lv-(a-b)^2 \pm  (a-b) \sqrt{D} \right] =2\ab^2\cd^2_\pm . 
\eeq

The integrals we consider are
\beqq  \begin{split}
    &\ing^{(123)}_{123} = \frac{1}{(2\pi \ii)^4} 
    \int_{\gamma_1} \d\xi^{123} \int_{\Gamma_1} \d\eta^{1} \int_{\Gamma_2} \d\eta^{2} \int_{\Gamma_3} \d\eta^{3} \, \, 
    \Pi^{(a)}(\bsxi, \bseta) \FF_{L}^{(a)} (\bsxi, \bseta), \\ 
    & \ing^{(123)}_{(12)3} = \frac{1}{(2\pi \ii)^3} 
    \int_{\gamma_1} \d\xi^{123} \int_{\Gamma_1} \d\eta^{12} \int_{\Gamma_2} \d\eta^{3} \, \, 
    \Pi^{(b)}(\bsxi, \bseta) F_{L}^{(b)} (\bsxi, \bseta), \\ 
    &\ing^{(123)2}_{(23)(12)} = \frac{1}{(2\pi \ii)^4} 
    \int_{\gamma_1}  \d\xi^{123} \int_{\gamma_2} \d \xi^2 
    \int_{\Gamma_1}  \d\eta^{23} \int_{\Gamma_2} \d\eta^{12} \, \, 
    \Pi^{(c)}(\bsxi, \bseta) F_{L}^{(c)} (\bsxi, \bseta) 
\end{split} \eeqq
where $\bsxi=\xi^{123}$ and $\bseta=(\eta^1, \eta^2, \eta^3)$ for the first integral; 
$\bsxi=\xi^{123}$ and $\bseta=(\eta^{12}, \eta^3)$ for the second integral; and 
$\bsxi = (\xi^{123},\xi^2)$ and $\bseta = (\eta^{12}, \eta^{23})$ for the third integral. 
The contours $\gamma_i$  are small circles around $-1$, and $\Gamma_i$ are small circles around $0$; all contours are chosen to be non-intersecting.  
The circle $\gamma_1$ is nested inside $\gamma_2$, and the circles $\Gamma_1, \Gamma_2, \Gamma_3$ are nested from inside to outside. 
The functions are 
\beq \label{eq:Piina}  \begin{split}
    &\Pi^{(a)}(\bsxi, \bseta):= \Pi^{(123)}_{123}(\bsxi, \bseta) = \frac{1}{(\eta^1 - \xi^{123})(\eta^2-\eta^1)(\eta^3-\eta^2)}, \\
    &\Pi^{(b)}(\bsxi, \bseta):= \Pi^{(123)}_{(12)3}(\bsxi, \bseta) = \frac{1}{(\eta^{12} - \xi^{123})(\eta^3-\eta^{12})}, \\	
    &\Pi^{(c)}(\bsxi, \bseta) := \Pi^{(123)2}_{(23)(12)}(\bsxi, \bseta) = \frac{1}{(\eta^{12} - \xi^{123})(\eta^{23}-\xi^2)(\xi^2-\eta^{12})}, 
\end{split} \eeq
and
\beqq  \begin{split}
    &\FF^{(a)}_L(\bsxi, \bseta):= \FF_{L}^{(123)|123} (\bsxi, \bseta) = \frac{\ff_{L,123}(\xi^{123})}{\ff_{L,1}(\eta^1)\ff_{L,2}(\eta^2)\ff_{L,3}(\eta^3)}, \\
    &F_{L}^{(b)} (\bsxi, \bseta):= F_{L}^{(123)|(12)3} (\bsxi, \bseta) = \frac{\ff_{L,123}(\xi^{123})}{\ff_{L,12}(\eta^{12})\ff_{L,3}(\eta^3)}, \\
    &F_{L}^{(c)} (\bsxi, \bseta) := F_{L}^{(123)2|(23)(12)} (\bsxi, \bseta) = \frac{\ff_{L,123}(\xi^{123})\ff_{L,2}(\xi^{2})}{\ff_{L,12}(\eta^{12})\ff_{L,23}(\eta^{23})}. 
\end{split} \eeqq

Lemma \ref{cor:criticalptofGlo} implies that the critical points satisfy 
\beqq  \begin{split}
    &-1<\pp^-<z_1^+=z_2^+=z_3^+=\pp^+<0 \qquad \text{for $\w \in \rgo \cup \cdots \cup \rgfi$,}\\
    &-1<\pp^-<z_2^- =z_{12}^+=z_3^+=z^+_{23}=\pp^+<0 \qquad \text{for $\w\in \rgs\cup \rgse$.}
\end{split} \eeqq
Note that, since $\ff_{L,1}(z)\ff_{L,2}(z)\ff_{L,3}(z)=\ff_{L,12}(z)\ff_{L,3}(z) = \ff_{L,12}(z) \frac{\ff_{L,23}(z)}{\ff_{L,2}(z)}= \ff_{L, 123}(z)$, we have, in terms of \eqref{eq:cKdef2}, 
\beq\label{eq:fmanycK} \begin{split}
    \frac{\ff_{L,123}(\pp^-)}{\ff_{L,1}(\pp^+)\ff_{L,2}(\pp^+)\ff_{L,3}(\pp^+)}
    = \frac{\ff_{L,123}(\pp^-)}{\ff_{L,12}(\pp^+)\ff_{L,3}(\pp^+)}
    = \frac{\ff_{L,123}(\pp^-)\ff_{L,2}(\pp^+)}{\ff_{L,12}(\pp^+)\ff_{L,23}(\pp^+)}
    = \frac{\ff_{L,123}(\pp^-)}{\ff_{L,123}(\pp^+)} = \cK_L . 
\end{split} \eeq

We take the contours to be the circles given by 
\beqq  \begin{split}
    &\gamma_1= \{ z\in \C \, : \, |z+1|= 1 + \pp^- \},\\
    &\gamma_2= \{ z\in \C \, : \, |z+1|= 1+\pp^+ \},\\
    &\Gamma_i= \{ z\in \C \, : \, |z|= |\pp^+|- (4-i)L^{-1/2} \}, \qquad i=1,2,3.\\
\end{split} \eeqq
These contours satisfy the necessary nesting structure. 
We now evaluate the integrals. 
From the formula of the functions, 
\beqq
    \Pi^{(a)}(\bsxi, \bseta) = O(L) , \qquad
    \Pi^{(b)}(\bsxi, \bseta) = O(L^{1/2}),   \qquad 
    \Pi^{(c)}(\bsxi, \bseta) = O(L) 
\eeqq
uniformly for $(\bsxi, \bseta)$ on the contours. 
Fix $\epsilon\in (0,1/2)$ and denote the disks
\beqq
    D_-= \{z\in \C : |z-\pp^-|\le L^{-\frac12+\frac{\epsilon}{3} }\}, \qquad
    D_+= \{z\in \C : |z-\pp^+|\le L^{-\frac12+\frac{\epsilon}{3} }\}. 
\eeqq
Let $\gamma_1^\epsilon$ be the part of the circle $\gamma_1$ contained in the disk $D_-$. 
Let $\gamma_2^\epsilon$, $\Gamma_1^\epsilon$, $\Gamma_2^\epsilon$, $\Gamma_3^\epsilon$ denote the parts of the corresponding circles contained  in the disk $D_+$. 
Note that $z_2^-=\pp^+$, and $\gamma_2^\epsilon$ is a sub-arc of the circle $\gamma_2$ near this critical point, when $\w \in \rgs \cup \rgse$. 
Lemma \ref{lem:fl_asym_sameside} implies that Lemma \ref{lem:asymptotics_f2}  \ref{eq:propertya} applies to $f_{L, 123}$, and Lemma \ref{lem:asymptotics_f2}  \ref{eq:propertyb} applies to $\ff_{L,1}, \ff_{L, 3}, \ff_{L, 12}$, and $\ff_{L,23}$. 
Furthermore, Lemma \ref{lem:asymptotics_f2}  \ref{eq:propertyb} applies to $\ff_{L,2}$ when $\w \in \rgo \cup \cdots \cup \rgfi$, while Lemma \ref{lem:asymptotics_f2}  \ref{eq:propertya} applies to $\ff_{L,2}$ when $\w \in \rgs \cup \rgse$. 
Thus, we find that 
\beqq
    \frac{\FF_{L}^{(a)} (\bsxi, \bseta)}{\cK_L}  = O(e^{-cL^{2\epsilon/3}}), \qquad
    \frac{\FF_{L}^{(b)} (\bsxi, \bseta)}{\cK_L}  = O(e^{-cL^{2\epsilon/3}}), \qquad
    \frac{\FF_{L}^{(c)} (\bsxi, \bseta)}{\cK_L}  = O(e^{-cL^{2\epsilon/3}})
\eeqq
for $(\bsxi, \bseta)$ on the contours outside the parts $\gamma_1^\epsilon\times \Gamma_1^\epsilon\times \Gamma_2^\epsilon\times \Gamma_3^\epsilon$, or $\gamma_1^\epsilon\times \Gamma_1^\epsilon\times \Gamma_2^\epsilon$, or $\gamma_1^\epsilon\times \gamma_2^\epsilon\times \Gamma_1^\epsilon\times \Gamma_2^\epsilon$, respectively. 
On the other hand, for $(\bsxi, \bseta)$ on the parts $\gamma_1^\epsilon\times \Gamma_1^\epsilon\times \Gamma_2^\epsilon\times \Gamma_3^\epsilon$, or $\gamma_1^\epsilon\times \Gamma_1^\epsilon\times \Gamma_2^\epsilon$, or $\gamma_1^\epsilon\times \gamma_2^\epsilon\times \Gamma_1^\epsilon\times \Gamma_2^\epsilon$ (respectively), 
we change variables as follows: 
\beqq
    \eta^*= \pp^+  + \frac{v_*}{L^{1/2}} \quad \text{for $*=1,2, 3,12,23$}, 
    \qquad \xi^{123}= \pp^- + \frac{u}{L^{1/2}}, 
    \qquad \xi^2= \pp^+ +\frac{v_0}{L^{1/2}}. 
\eeqq
Noting $\pp^+- \pp^-= \sqrt{D}/\lv$, we find from \eqref{eq:Piina} that 
\beqq  \begin{split}
    \Pi^{(a)}(\bsxi, \bseta)= \frac{\lv L/\sqrt{D} (1 + o(1))}{(v_2-v_1)(v_3-v_2)} , \quad 
    \Pi^{(b)}(\bsxi, \bseta) = \frac{\ell L^{1/2}/\sqrt{D} (1 + o(1))}{v_3-v_{12}},\quad 
    \Pi^{(c)}(\bsxi, \bseta) = \frac{\ell L/\sqrt{D} (1 + o(1))}{(v_{23}-v_0)(v_0-v_{12})} 
\end{split} \eeqq
for variables $|u|, |v_0|, |v_1|, |v_2|, |v_3|, |v_{12}|, |v_{23}|\le L^{\epsilon/3}$ on appropriate contours. 
Using Lemma \ref{lem:asymptotics_f1} and recalling \eqref{eq:mcHf}, we also find that 
\beqq  \begin{split}
    &\frac{\FF_{L}^{(a)} (\bsxi, \bseta)}{\cK_L}
    = \frac{e^{\frac{1}{2}\Bt u^2 } }{e^{-\frac{1}{2}\At_{1}v^2_1+  \sqrt{2} \ab  \mr r_1v_1 - \frac{1}{2}\At_{2}v^2_2 + \sqrt{2} \ab (\mr r_2 - \mr r_1)v_2  - \frac{1}{2}\At_{3}v^2_3 - \sqrt{2} \ab \mr r_2v_3} }
    (1+o(1)), \\
    &\frac{F_{L}^{(b)} (\bsxi, \bseta)}{\cK_L} 
    = \frac{e^{\frac{1}{2}\Bt u^2 }}
    {e^{ - \frac{1}{2}\At_{12}v^2_{12} +\sqrt{2} \ab\mr r_2v_{12} - \frac{1}{2}\At_{3}v^2_3 - \sqrt{2} \ab\mr r_2v_3}} 
    (1+o(1)), \\
    &\frac{F_L^{(c)} (\bsxi, \bseta)}{\cK_L} 
    =  \frac{ e^{\frac12 \Bt u^2 + \frac12 \Bt_{2}v^2_0 + \sqrt{2} \ab(\mr r_2- \mr r_1)v_0 } }
    {e^{- \frac12 \At_{12}v^2_{12} + \sqrt{2} \ab \mr r_2v_{12}  - \frac12 \At_{23}v^2_{23} - \sqrt{2} \ab \mr r_1v_{23}} }
    (1+o(1))
\end{split} \eeqq
for the same variables, where we set 
\beqq
    \At_{*} = -\mcG''_{*}(\pp^+)>0 \quad \text{for} \quad *=1, 3,12,23, 
    \quad \Bt= \mcG''_{123}(\pp^-) =2\ab^2\cd^2_->0, 
    \quad \At_{2} = -\mcG''_{2}(\pp^+)= -\Bt_{2}. 
\eeqq
Note that $\At_2>0$ if $\w \in \rgo \cup \cdots \cup \rgfi$, since $\pp^+=z_2^+$ in this case, and $\Bt_2>0$ if $\w \in \rgs \cup \rgse$, since $\pp^+=z_2^-$ in this case. 
Hence, noting that $\dd\bsxi\dd\bseta$ is equal to $L^{-2} \dd u \dd v_1\dd v_2\dd v_3$, $L^{-3/2} \dd u \dd v_{12} \dd v_3$, or $L^{-2} \dd u \dd v_0\dd v_{12}\dd v_{23}$, respectively, 
we conclude that 
\beqq
    \lim_{L \to \infty} \frac{L\sqrt{D}}{\ell \cK_L} \ing^{(123)}_{123} = P_0P_a, \qquad
    \lim_{L \to \infty} \frac{L\sqrt{D}}{\ell \cK_L} \ing^{(123)}_{(12)3} = P_0P_b, \qquad
    \lim_{L \to \infty} \frac{L\sqrt{D} }{\ell \cK_L} \ing^{(123)2}_{(23)(12)} = P_0P_c, \qquad
\eeqq
where
\beqq
    P_0= \frac1{2\pi\ii} \int_{\ii \R} e^{\frac{1}{2} \Bt u^2 }  \dd u = \frac{1}{\sqrt{2\pi \Bt}}
    = \frac1{\ab\cd_-\sqrt{4\pi}},
\eeqq
and
\beqq  \begin{split}
    &P_a = \frac{1}{(2\pi \ii)^3}   \int_{3+\ii\R}  \d v_1  \int_{2+\ii\R} \d v_2  \int_{1+\ii\R} \d v_3
    \frac{ e^{ \frac{1}{2}\At_{1}v^2_1 -\sqrt{2} \ab \mr r_1v_1 + \frac{1}{2}\At_{2}v^2_2- \sqrt{2} \ab (\mr r_2 - \mr r_1)v_2 
    +  \frac{1}{2}\At_{3}v^2_3 + \sqrt{2} \ab \mr r_2v_3} }{(v_2-v_1)(v_3-v_2)}, \\
    &P_b=  \frac{1}{(2\pi \ii)^2} \int_{3+\ii\R }\dd v_{12} \int_{2+\ii \R} \dd v_3 \, 
    \frac{ e^{ \frac{1}{2}\At_{12}v^2_{12} - \sqrt{2} \ab\mr r_2v_{12} + \frac{1}{2}\At_{3}v^2_3 + \sqrt{2} \ab\mr r_2v_3} }{v_3-v_{12}}, \\
    &P_c= \frac{1}{(2\pi i)^3}  \int_{\ii\R} \d v_0  \int_{3+\ii\R} \d v_{23}   \int_{2+\ii\R} \d v_{12}
    \frac{ e^{ \frac{1}{2}\At_{12}v^2_{12} - \sqrt{2} \ab \mr r_2v_{12} +
    \frac{1}{2}\Bt_{2}v^2_0 + \sqrt{2} \ab (\mr r_2 - \mr r_1)v_0 +
    \frac{1}{2}\At_{23}v^2_{23} + \sqrt{2} \ab \mr r_1v_{23}}}{(v_{23}-v_0)(v_0-v_{12})}. 
\end{split} \eeqq
Here, $\ii\R$ is oriented upward, while $1+\ii\R$, $2+\ii\R$, and $3+\ii \R$ are oriented downwards. 

We now evaluate $P_a$, $P_b$, and $P_c$ using Lemma \ref{lem:bridge}. 
Noting that $\GG_1(z)+\GG_2(z)+\GG_3(z)= \GG_{12}(z)+\GG_{3}(z)= \GG_{12}(z)+\GG_{23}(z)-\GG_2(z)=\GG_{123}(z)$, set
\beqq
    \At := \At_1+\At_2+\At_3= \At_{12}+\At_{3}=\At_{12}+\At_{23} + \Bt_2= - \mcG''_{123}(z^+_{c}) = 2\ab^2\cd^2_+. 
\eeqq
Changing variables $v_*\mapsto -v_*$ in $P_a$ and $P_b$, we find from Lemma \ref{lem:bridge} (noting $\At_1+\At_2=\At_{12}$) that
\beqq  \begin{split}
    &P_a = - \frac1{\sqrt{2\pi\At} } \prob \left[ \frac{\sqrt{\At}}{\sqrt{2}\ab} \B\left(\frac{\At_1}{\At}\right) > \mr r_1, \ \frac{\sqrt{\At}}{\sqrt{2}\ab}  \B\left(\frac{\At_{12}}{\At}\right) > \mr r_2 \right], \\
    &P_b = - \frac1{\sqrt{2\pi\At} } \prob \left[ \frac{\sqrt{\At}}{\sqrt{2}\ab}  \B\left(\frac{\At_{12}}{\At}\right) > \mr r_2 \right]. 
\end{split} \eeqq
For $P_c$, we first move the contour for $v_{12}$ to the right of the contour for $v_{23}$. This can be done without changing the value of the integral since the integrand is analytic at $v_{12}=v_{23}$. 
We then change all variables $v_*\mapsto -v_*$. Reversing the orientation of the contour for $v_0$, we find that 
\beqq  \begin{split}    
    &P_c= \frac{1}{(2\pi i)^3}  \int_{\ii\R} \d v_0  \int_{-2+\ii\R} \d v_{23}   \int_{-3+\ii\R} \d v_{12}
    \frac{ e^{ \frac{1}{2}\At_{12}v^2_{12} + \sqrt{2} \ab \mr r_2v_{12} +
    \frac{1}{2}\Bt_{2}v^2_0 - \sqrt{2} \ab (\mr r_2 - \mr r_1)v_0 +
    \frac{1}{2}\At_{23}v^2_{23} - \sqrt{2} \ab \mr r_1v_{23}}}{(v_{23}-v_0)(v_0-v_{12})}
\end{split} \eeqq
where all contours are oriented upwards. 
Moving the $v_0$-contour across the $v_{23}$-contour to the left, and taking into account the simple pole $v_0=v_{23}$, we find that 
\beqq  \begin{split}
    P_c= &  \frac{1}{(2\pi i)^3}  \int_{-2+\ii\R} \d v_0  \int_{\ii\R} \d v_{23}   \int_{-3+\ii\R} \d v_{12}
    \frac{ e^{ \frac{1}{2}\At_{12}v^2_{12} + \sqrt{2} \ab \mr r_2v_{12} +
    \frac{1}{2}\Bt_{2}v^2_0 - \sqrt{2} \ab (\mr r_2 - \mr r_1)v_0 +
    \frac{1}{2}\At_{23}v^2_{23} - \sqrt{2} \ab \mr r_1v_{23}}}{(v_{23}-v_0)(v_0-v_{12})} \\
    &
    -\frac{1}{(2\pi i)^2}  \int_{-2+\ii\R} \d v_{23}   \int_{-3+\ii\R} \d v_{12}
    \frac{ e^{ \frac{1}{2}\At_{12}v^2_{12} + \sqrt{2} \ab \mr r_2v_{12} +
    \frac{1}{2}(\Bt_2+ \At_{23}) v^2_{23} - \sqrt{2} \ab \mr r_2 v_{23}}}{v_{23}-v_{12}}.
\end{split} \eeqq
Noting that $\At_{12}+\Bt_2=\At_1$ and $\Bt_2+\At_{23}=\At_3$, Lemma \ref{lem:bridge} implies that
\beqq  \begin{split}
    \sqrt{2\pi\At}  P_c 
    =  \prob \left[ \frac{\sqrt{\At}}{\sqrt{2}\ab} \B\left(\frac{\At_{12}}{\At}\right) >   \mr r_2, \ \frac{\sqrt{\At}}{\sqrt{2}\ab}  \B\left(\frac{\At_{1}}{\At}\right) >  \mr r_1 \right] 
    - \prob \left[ \frac{\sqrt{\At}}{\sqrt{2}\ab} \B\left(\frac{\At_{12}}{\At}\right) >   \mr r_2 \right].
\end{split} \eeqq
We have $\sqrt{\At}=\sqrt{2} \ab\cd_+$, and, from \eqref{eq:Aratio}, $\frac{\At_1}{\At} = \frac{\slope y_1- x_1}{\slope-1}$ and $\frac{\At_{12}}{\At} = \frac{\slope y_2- x_2}{\slope-1}$. 
The result now follows since $\ab^2\cd_+\cd_-= \frac{\lv \sqrt{D}}{2\sqrt{ab}}$. 
\end{proof}

\subsection{Estimation of the remainder and the proof of the theorem} \label{sec:equalt}

In this section, we state estimates for the remaining integrals and use them to complete the proof of Theorem \ref{thm:offdiagflucsameside}. 
The estimates are given in two propositions, each applying to different choices of $\bsigma$ and $\btau$. 
The first proposition implies estimates for $\QQ^{(1,1,1)}_{L}$ when $\w\in \rgo\cup\cdots\cup \rgfi$, and for both $\QQ^{(1,1,1)}_{L}$ and $\QQ^{(1,2,1)}_{L}$ when $\w\in \rgs\cup\rgse$. 
The second proposition gives estimates on the remaining cases of $\QQ^{(\bn)}_L$. 
The proof of Theorem \ref{thm:offdiagflucsameside} is given at the end of this subsection. 

Recall that the integrals $\ing^{\bsigma}_{\btau}$ depend on $L>0$. 
Recall also the constant $\cK_L$ from \eqref{eq:cKdef2}. The proof of the following proposition is given in Subsection \ref{sec:bounds}. 

\begin{prop} \label{cor:incor}
For every $\w \in \rgo\cup \cdots \cup \rgse$,  there exist constants $C, c, L_0>0$ such that for every $L\ge L_0$ and $\vecn\in \N^3$, 
and for every $\bsigma, \btau\in \listn$ of the forms 
\begin{enumerate} [(a)]
\item $\bsigma =2^{a_2} (23)^{a_{23}}3^{a_3}(123)^{a_{123}}(12)^{a_{12}}1^{a_1}$ and $\btau = 3^{b_3'}2^{b_2'}1^{b_{1}}2^{b_2''}3^{b_3''}$ if $\w \in \rgo\cup \rgt$; 

\item $\bsigma = 3^{a_3}(23)^{a_{23}}2^{a_2}(123)^{a_{123}}(12)^{a_{12}}1^{a_1}$  and $\btau = 3^{b_3'}2^{b_2'}1^{b_{1}}2^{b_2''}3^{b_3''}$ if $\w \in \rgth\cup \rgf$;

\item $\bsigma =3^{a_3}(23)^{a_{23}}(123)^{a_{123}}1^{a_1} (12)^{a_{12}} 2^{a_2}$ and $\btau = 3^{b_3'}2^{b_2'}1^{b_{1}}2^{b_2''}3^{b_3''}$ if $\w \in \rgfi$; 

\item $\bsigma = 3^{a_3}(23)^{a_{23}}(123)^{a_{123}}1^{a_1} (12)^{a_{12}} 2^{a_2}$ and $\btau = 2^{b_2}(23)^{b_{23}}3^{b_{3}'}(12)^{b_{12}}1^{b_1}3^{b_{3}''}$ if $\w \in \rgs\cup \rgse$;
\end{enumerate} 
satisfying 
\begin{itemize}
\item $(a_{123}, a_{12}, a_{23}, a_1, a_2, a_3) \neq (1,0,0,0,0,0)$ when $\w \in \rgo\cup \cdots\cup \rgfi$, 
\item $a_{123}+a_1\ge 1$  when $\w \in \rgfi\cup \rgs\cup \rgse$, 
\item $(a_{123}, a_{12}, a_{23}, a_1, b_2, a_3)\neq (1,0,0,0,0,0)$ when $\w \in \rgs\cup \rgse$, 
\end{itemize}
we have 
\beq \label{eq:incorr}
    |\ing^{\bsigma}_{\btau}| 
    \le C^{|\vecn|} \sqrt{n_1! (n_2-n_1+a_1+b_1)! (n_2-n_3+a_3+b_3)!n_3!} e^{-cL} \cK_L, 
\eeq
where $b_3=b_3'+b_3''$. 
\end{prop}


Together with Lemma \ref{lem:QQ111list}, Lemma \ref{lem:QQ121list}, Proposition \ref{prop:leading}, the above result implies the following.  

\begin{cor} \label{cor:Q1ll}
As $L\to \infty$, the following hold:
\beqq  \begin{split}
    &\QQ^{(1,1,1)}_{L} 
    =  \frac{\sqrt{ab}\cK_L}{2\pi  LD }  \prob \left[ \cd_{+} \B\left(\frac{\slope y_i- x_i}{\slope -1} \right) > \mr r_i,  i=1,2 \right] (1+o(1)) \qquad  \text{for $\w \in \rgo \cup \cdots \cup \rgfi$,}\\
    & \QQ^{(1,1,1)}_{L} + \frac14 \QQ^{(1,2,1)}_{L}
    = \frac{\sqrt{ab}\cK_L}{2\pi  LD} \prob \left[ \cd_{+} \B\left(\frac{\slope y_i- x_i}{\slope -1} \right) > \mr r_i,  i=1,2 \right] (1+o(1)) \quad \text{for $\w \in \rgs \cup \rgse$.}
\end{split} \eeqq
\end{cor}

\begin{proof}
For $\w \in \rgo \cup\rgt$, we use Lemma \ref{lem:QQ111list} (a) to see that 
$\QQ^{(1,1,1)}_{L} =-\ing^{231}_{123}-\ing^{3(12)}_{123}-\ing^{(123)}_{123}$. 
The integrals $\ing^{231}_{123}$ and $\ing^{3(12)}_{123}$ are of the forms in Proposition \ref{cor:incor} (a). Thus, comparing the estimate \eqref{eq:incorr} with the asymptotics of $\ing^{(123)}_{123}$ evaluated in Proposition \ref{prop:leading}, we obtain the result for $\w \in \rgo \cup\rgt$. 

For $\w \in \rgth \cup\rgf$, we use Lemma \ref{lem:QQ111list} (b), Proposition \ref{cor:incor} (b), and Proposition \ref{prop:leading}. 

For $\w \in \rgfi$, we use Lemma \ref{lem:QQ111list} (c), Proposition \ref{cor:incor} (c), and Proposition \ref{prop:leading}. 

For $\w \in \rgs\cup\rgse$, we use Lemma \ref{lem:QQ111list} (d), Lemma \ref{lem:QQ121list}, Proposition \ref{cor:incor} (d), and Proposition \ref{prop:leading}. Here we note that $(a_{123}, a_{12}, a_{23}, a_1, b_2, a_3)= (1,0,0,0,0,0)$ for 
the integral $\ing^{(123)}_{(12)3}$ from Lemma \ref{lem:QQ111list} (d), as well as for the integral $\ing^{(123)2}_{(23)(12)}$ from Lemma \ref{lem:QQ121list}. Thus, Proposition \ref{cor:incor} (d) does not apply to these integrals; they are instead evaluated in Proposition \ref{prop:leading}. 
\end{proof}


The next proposition is proved in Subsection \ref{sec:deform}. It will be used to estimate the remainder of the series \eqref{eq:Qld}. 

\begin{prop} \label{cor:decomposition}
For every $\w \in \rgo\cup \cdots \cup \rgse$, there exist constants $C, c, L_0>0$ such that for every $\vecn \in \N^3$ with $\vecn\neq (1,1,1)$, and for $\bsigma, \btau \in \listn$ of the forms 
\beq \label{eq:todecomp}
    \bsigma = 3^{n_{31}}2^{n_{21}}1^{n_{1}}2^{n_{22}}3^{n_{32}}, \qquad \btau = 3^{n'_{31}}2^{n'_{21}}1^{n_{1}}2^{n'_{22}}3^{n'_{32}}
\eeq 
satisfying 
\begin{itemize}
\item $\vecn\neq (1,2,1)$ when $\w\in \rgs\cup \rgse$, 
\item $n_1\ge n_{21}+1$ when $\w\in \rgfi\cup \rgs\cup \rgse$, 
\end{itemize}
we have, for every $L\ge L_0$, 
\beq \label{eq:decap}
    |\ing^{\bsigma}_{\btau} |  
    \le  C^{|\vecn|} 
    \frac{ (n_1!)^{3/2} (n_2!)^2 (n_3!)^{3/2}}{\sqrt{(n_1\vee n_2- n_1\wedge n_2)! (n_2\vee n_3- n_2\wedge n_3)!}}  e^{-cL} \cK_L .
\eeq
\end{prop}

To estimate the series \eqref{eq:Qld} using the above result, we also need the following lemma.
 
\begin{lem} \label{lem:serco}
For every $A>0$, the following series is convergent: 
\beq
    \sum_{n_1, n_2, n_3=1}^\infty \frac{A^{n_1+n_2+n_3}}{\sqrt{n_1!(n_1\vee n_2 - n_1\wedge n_2)!(n_2\vee n_3 - n_2\wedge n_3)! n_3! }}.   
\eeq
\end{lem}

\begin{proof}
Using the inequality $\frac{M!(N-M)!}{N!} \ge \frac1{2^N}$, we find that 
$a! (a\vee b - a\wedge b)! \ge a! (a\vee b - a)!\ge \frac{(a\vee b)!}{2^{a\vee b}}$ for all  positive integers $a$ and $b$. 
Thus, 
\beqq
    n_1! (n_1\vee n_2 - n_1\wedge n_2)!(n_2\vee n_3 - n_2\wedge n_3)! n_3!
    \ge \frac{(n_1\vee n_2)! (n_2\vee n_3)!}{2^{n_1\vee n_2+ n_2\vee n_3}}
    \ge \frac{(n_1\vee n_2\vee n_3)!}{2^{2(n_1\vee n_2\vee n_3)}} . 
\eeqq
Hence, the series is dominated by 
\beq
    \sum_{n_1, n_2, n_3=1}^\infty \frac{(2A^3)^{n_1\vee n_2\vee n_3}}{\sqrt{(n_1\vee n_2 \vee n_3)!}}   
    = \sum_{n=1}^\infty  \frac{(2A^3)^{n}}{\sqrt{n!}} (n^3-(n-1)^3). 
\eeq
The last series is convergent. 
\end{proof}

We now obtain an estimate for the remainder of the series \eqref{eq:Qld}. 

\begin{cor} \label{cor:errorsum}
For every $\w \in \rgo\cup \cdots \cup \rgse$, there exists a constant $c>0$ such that, as $L\to \infty$, 
\beqq
    \sum_{\vecn\in \N^3\setminus \{(1,1,1)\}} \frac{1}{(\vecn!)^2} | \QQ_L^{(\vecn)}|
    =  O\left( e^{- cL}\cK_L  \right) \qquad \text{if $\w\in \rgo\cup \cdots \cup \rgfi$,}
\eeqq
and
\beqq
    \sum_{\vecn\in \N^3\setminus\{(1,1,1), (1,2,1)\}} \frac{1}{(\vecn!)^2} | \QQ_L^{(\vecn)}|
    =  O\left( e^{- cL} \cK_L  \right) \qquad \text{if $\w\in \rgs\cup \rgse$.} 
\eeqq
\end{cor}

\begin{proof} 
We use the formula for $\QQ^{(\bn)}_L$ given in Lemma \ref{lem:Dbaseint2}. 
We take the $z_i$-contours in the sum to be circles of fixed radii larger than $1$; for concreteness, we choose them to be the circles of radii $2$ centered at the origin. Since
\beqq 
    \left| \oint_{|z|=2}  \frac{(z+1)^{n-n'-1}}{z^{n'-i+1}} \frac{\dd z}{2\pi \ii} \right|
    \le \frac{3^n}{2^{n'-i}} \le 3^{n+n'} 
\eeqq
for $0\le i\le 2n'$, we find that, for each $\vecn=(n_1, n_2, n_3)\in \N^3$, 
\beq \label{eq:QQbasicb}
    |\QQ^{(\bn)}_{L} |
    \le  3^{2|\bn|}\sum_{i= 0 \vee (2n_2-n_1+1)}^{2n_2} \sum_{j=0 \vee (2n_3-n_2+1)}^{2n_3} |\alpha_{ij} |, 
\eeq
where $\alpha_{ij}$ is a sum of $\binom{2n_2}{i}\binom{2n_3}{j}$ terms, each of the form  $\ing^{\bsigma}_{\btau}$,  with $\bsigma, \btau\in\listn$ of the forms indicated in Lemma \ref{lem:Dbaseint2}. 
Since $i\ge 2n_2-n_1+1$ in the sum, we find that $n_{22}+n_{22}'=i\ge 2n_2-n_1+1$, which implies  $n_1\ge n_{21}+n_{21}'+1$. 
Hence, $n_1\ge n_{21}+1$, which is one of the conditions of Proposition \ref{cor:decomposition}. 
Using $\binom{2n_2}{i}\binom{2n_3}{j}\le 2^{2n_2+2n_3}\le 2^{2|\vecn|}$ and applying Proposition \ref{cor:decomposition}, 
\eqref{eq:QQbasicb} implies that for every $\w \in \rgo\cup \cdots \cup \rgse$, there exist  constants $C, c, L_0>0$ such that
\beq 
    |\QQ^{(\bn)}_{L} | \le  4n_2 n_3 3^{2|\vecn|}2^{2|\vecn|} C^{|\vecn|} 
    \frac{ (n_1!)^{3/2} (n_2!)^2 (n_3!)^{3/2}}{\sqrt{(n_1\vee n_2- n_1\wedge n_2)! (n_2\vee n_3- n_2\wedge n_3)!}}  e^{-cL} \cK_L 
\eeq
for every $L\ge L_0$, 
for every $\bn\in \N^3\setminus \{(1,1,1)\}$ if $\w\in \rgo\cup \cdots\cup \rgfi$, and for every $\bn\in \N^3\setminus \{(1,1,1), (1,2,1)\}$ if $\w\in \rgs\cup \rgse$. 
Thus, setting $A= 48C$, we have 
\beqq
    \sum_{\vecn\in \N^3\setminus \{(1,1,1)\}} \frac{1}{(\vecn!)^2} | \QQ_L^{(\vecn)}|
    \le  e^{-cL} \cK_L\sum_{\vecn\in \N^3\setminus \{(1,1,1)\}} 
    \frac{A^{|\vecn|} }{\sqrt{n_1!(n_1\vee n_2- n_1\wedge n_2)! (n_2\vee n_3- n_2\wedge n_3)!n_3!} }  
\eeqq
for $\w\in \rgo\cup \cdots\cup \rgfi$, and 
\beqq
    \sum_{\substack{\vecn\in \N^3 \\ \vecn\neq (1,1,1), (1,2,1)}} \frac{1}{(\vecn!)^2} | \QQ_L^{(\vecn)}|
    \le  e^{-cL} \cK_L \sum_{\substack{\vecn\in \N^3 \\ \vecn\neq (1,1,1), (1,2,1)}} 
    \frac{A^{|\vecn|} }{\sqrt{n_1!(n_1\vee n_2- n_1\wedge n_2)! (n_2\vee n_3- n_2\wedge n_3)!n_3!} }
\eeqq
for $\w\in \rgs\cup \rgse$. The series on the right converges due to Lemma \ref{lem:serco}, and we obtain the result. 
\end{proof}

We now complete the proof of Theorem \ref{thm:offdiagflucsameside}. 

\begin{proof}[Proof of Theorem \ref{thm:offdiagflucsameside}]
Suppose that (see \eqref{eq:x1x2y1y2_sameside}) $\frac{1}{\slope} < \frac{y_1}{x_1}, \frac{y_2}{x_2} < 1$ 
and $\mv(x_1, y_1)<\mv(x_2, y_2)$, i.e. $\w\in \rg$ (see \eqref{eq:rg}). 
Then, Corollary \ref{cor:Q1ll} and Corollary \ref{cor:errorsum}, together with \eqref{eq:Qld}, imply that for every $\w \in \rgo \cup \cdots \cup \rgse$, 
\beq \label{eq:rpa}  \begin{split}
    \lim_{L\to \infty} \frac{2\pi  LD }{\sqrt{ab}\cK_L}   \QQ_3(\bM_L, \bN_L, \bT_L)
    =  \prob \left[ \cd_{+} \B\left(\frac{\slope y_i- x_i}{\slope -1} \right) > \mr r_i,  i=1,2 \right].
\end{split} \eeq
The analysis for $\QQ_1(\lceil aL \rceil,\lceil bL \rceil,\lv L)$ is similar (and easier), and we find 
$\lim_{L\to \infty} \frac{2\pi  LD }{\sqrt{ab}\cK_L} \QQ_1(\lceil aL \rceil,\lceil bL \rceil,\lv L)=1$. 
Thus, we obtain \eqref{eq:Qlimitoffdiagsameside}, proving Theorem \ref{thm:offdiagflucsameside} in this case.

Now consider $\w \in \rg \setminus (\rgo \cup \cdots \cup \rgse)$.
Then $\w$ lies on the boundary between two subregions $\rg_i$ and
$\rg_{i+1}$ for some $i=1,\ldots,6$. This boundary is contained in
the hypersurface
\begin{equation*}
\bigl\{(x_1,y_1,x_2,y_2)\in(0,1)^4:
\mr g_i(x_1,y_1)=\mr g_i(x_2,y_2)\bigr\},
\end{equation*}
where
\begin{equation*}
(\mr g_1(x,y),\ldots,\mr g_6(x,y))
=
\left(x,\frac{1-y}{1-x},y-x,\frac{y}{x},\slope y-x,y\right).
\end{equation*}

For convenience, we write
\begin{equation*}
\mathcal Y_L(u,v)
:=
\frac{\LPP(\lceil uaL\rceil,\lceil vbL\rceil)-h(u,v)L}
{\sqrt{2}\ab L^{1/2}}, \qquad \mr t(u,v) := \frac{\slope v-u}{\slope -1}.
\end{equation*}

First, consider the boundary between $\rg_1$ and
$\rg_2$. Fix $\w=(x,y_1,x,y_2)$ on this boundary.
Then the condition \eqref{eq:x2y2region} in the definition of $\rg$ implies that $y_1<y_2$.
Fix $0<\rho<1/\slope$ and set
\begin{equation*}
\w(s)=(x(s),y(s),x,y_2) \qquad \text{where} \quad
x(s)=x+s,\qquad y(s)=y_1+\rho s.
\end{equation*}
There exists a sufficiently small open interval $I$ containing $0$ such
that $\w(s)\in\rg$ for every $s\in I$, with
\begin{equation*}
\w(s)\in\rg_2\quad\text{if }s<0,
\qquad
\w(s)\in\rg_1\quad\text{if }s>0.
\end{equation*}
Consider the conditional law
\begin{equation*}
\mathbb P_L(\,\cdot\,)
:=
\mathbb P\bigl(\,\cdot\,\bigm|
\LPP(\lceil aL\rceil,\lceil bL\rceil)=\ell L\bigr).
\end{equation*}
We apply Lemma 4.2.
By Theorem \ref{thm:offdiagflucsameside}, already proved in $\rg_1$ and $\rg_2$, for every $\mr r_1,\mr r_2\in\R$, with $\mathcal{A}_L=\{\mathcal{Y}_L(x,y_2)\le \mr r_2\}$, we have
\begin{equation*}
    \lim_{L\to \infty} \mathbb P_L( \{\mathcal{Y}_L(x(s), y(s)) \le \mr r_1\} \cap \mathcal{A}_L ) =
     \prob \left[ \cd_{+} \B\left(\mr t(x(s), y(s)) \right) \le \mr r_1, \cd_{+} \B\left(\mr t(x, y_2) \right) \le \mr r_2\right], \quad s\in I\setminus\{0\},
\end{equation*}
and the limit is continuous in $s\in I$. Thus, Lemma 4.2(a) is satisfied.

On the other hand, for $s_1\neq s_2$,
\begin{equation*}
\frac{y(s_2)-y(s_1)}{x(s_2)-x(s_1)}
=\rho\in(0,1/\slope).
\end{equation*}
If $s_1<s_2$, then $h(x(s_1),y(s_1))<h(x(s_2),y(s_2))$, and hence
$(x(s_1),y(s_1),x(s_2),y(s_2))\in\rg_6$.
If $s_1>s_2$, then $h(x(s_1),y(s_1))>h(x(s_2),y(s_2))$, and hence
$(x(s_2),y(s_2),x(s_1),y(s_1))\in\rg_6$.
In either case, for every $\mr r\in\R$,
\begin{equation*}
\begin{split}
&\lim_{L\to \infty} \mathbb P_L( \mathcal{Y}_L(x(s_1), y(s_1)) \le \mr r, \, \mathcal{Y}_L(x(s_2), y(s_2)) > \mr r )\\
&= \prob \left[ \cd_{+} \B\left(\mr t(x(s_1), y(s_1)) \right) \le \mr r, \cd_{+} \B\left(\mr t(x(s_2), y(s_2)) \right) > \mr r\right], \qquad s_1, s_2\in I, \quad s_1\neq s_2,
\end{split}
\end{equation*}
by Theorem \ref{thm:offdiagflucsameside}, already proved in $\rg_6$.
Furthermore, the limit is continuous on $(s_1,s_2)\in I\times I$ and vanishes when $s_1=s_2$.
Thus, Lemma 4.2(b) is satisfied. Hence, Lemma 4.2 implies that Theorem \ref{thm:offdiagflucsameside} holds for $\w$ on the boundary between $\rg_1$ and $\rg_2$.

We next consider the boundary between $\rg_i$ and $\rg_{i+1}$
for $i=2,\ldots,6$. Fix $\w=(x_1,y_1,x_2,y_2)$ on this boundary,
keep $x_1,x_2,y_2$ fixed, and vary $y_1$. More precisely, set
$y(s)=y_1+s$ and $\w(s)=(x_1,y(s),x_2,y_2)$.
Then Lemma 4.2(a) follows because, for $s\neq 0$ sufficiently small,
$\w(s)$ lies in either $\rg_i$ or $\rg_{i+1}$.
Lemma 4.2(b) follows from the case of the boundary between $\rg_1$ and $\rg_2$ established above, since, for $s_1\neq s_2$, the corresponding pair of points can be ordered so that $(x_1,y(s_1),x_1,y(s_2))$ or $(x_1,y(s_2),x_1,y(s_1))$ lies on that boundary.
Therefore, Lemma 4.2 applies, and Theorem \ref{thm:offdiagflucsameside} holds on each of these boundaries as well.
Hence, we have now proved Theorem \ref{thm:offdiagflucsameside} when $\frac{1}{\slope} < \frac{y_1}{x_1}, \frac{y_2}{x_2} < 1$ and $\mv(x_1, y_1) < \mv(x_2, y_2)$.

If $\mv(x_1, y_1) > \mv(x_2, y_2)$, the result follows by relabeling the points, since the limit is invariant under interchanging $(x_1, y_1)$ and $(x_2, y_2)$. 
If $\mv(x_1, y_1) = \mv(x_2, y_2)$, the result again follows from Lemma \ref{lem:bootstrap}. 
Thus, Theorem \ref{thm:offdiagflucsameside} is proved when $\frac{1}{\slope} < \frac{y_1}{x_1}, \frac{y_2}{x_2} < 1$. 

Now suppose that $1 < \frac{y_1}{x_1}, \frac{y_2}{x_2} < \slope$. 
Since $\LPP(m, n)\eqind \LPP(n,m)$, 
\beqq \begin{split}
    \prob \left[\frac{\mc L(x_iaN, y_ibN) - \mv(x_i,y_i)N}{\sqrt{2} \ab   N^{1/2}} > \mr r_i, \ i=1, 2 \bigg| \LPP(aN, bN)= \lv N \right] 
\end{split} \eeqq
is equal to 
\beqq
    \prob \left[\frac{\mc L(y_ibN, x_iaN) - \mv(x_i,y_i)N}{\sqrt{2} \ab   N^{1/2}} > \mr r_i, \ i=1, 2 \bigg| \LPP(bN, aN)= \lv N \right]. 
\eeqq 
We observe that $D$ in \eqref{eq:Ddefn} and $\ab$ in \eqref{eq:abdf} are symmetric with respect to $a$ and $b$. 
The function $\mv(x,y)$ in \eqref{eq:mvxyshaded}, which involves $a$ and $b$, is invariant under simultaneous exchange of $a \leftrightarrow b$ and $x \leftrightarrow y$. 
Finally, $\cd_\pm$ in \eqref{eq:abdf} become $\cd_\mp$ when $a$ and $b$ are swapped. 
From these observations, the part of Theorem \ref{thm:offdiagflucsameside} for $1 < \frac{y_1}{x_1}, \frac{y_2}{x_2} < \slope$ follows from the case $\frac{1}{\slope} < \frac{y_1}{x_1}, \frac{y_2}{x_2} < 1$. 
This completes the proof. 
\end{proof}

\subsection{Bounds of the integrals and proof of Proposition \ref{cor:incor}}\label{sec:bounds}

We estimate the integrals appearing in Proposition \ref{cor:incor}. From \eqref{eq:ingdef}, the integrals are of the form 
\beq \label{eq:intaga}
    \ing^{\bsigma}_{\btau}
    =  \frac{1}{(2\pi \ii)^{|\bsigma|+|\btau|}} \int \dd \bsxi^{\bsigma} \int \dd \bseta^{\btau} \, 
    \Pi^{\bsigma}_{\btau}(\bsxi, \bseta) \FF_{L}^{\bsigma|\btau} (\bsxi, \bseta) . 
\eeq
We note that if $\bsigma\in \listn$ and $\type(\bsigma)=\veca=(a_{123}, a_{12}, a_{23}, a_1, a_2, a_3)$, then 
\beq \label{eq:atypen}
    n_1=a_1+a_{12}+a_{123}, \quad n_2=a_2+a_{12}+a_{23}+a_{123}, \quad n_3=a_3+a_{23}+a_{123}, 
\eeq
and thus,
\beq \label{eq:atyeq}
    |\veca|= a_1+a_2+a_3+a_{12}+a_{23}+a_{123} \le |\vecn| \le 3|\veca|
    \quad \text{and} \quad n_2+a_1+a_3= |\veca|
\eeq

The rational function $\Pi^{\bsigma}_{\btau}$ satisfies the following estimate. Note that the well-known bound $N! \ge N^N e^{-N}$ implies that $N^N\le e^N N!\le 4^N N!$ for every positive integer $N$.

\begin{lem} \label{lem:Pibound}
Let $\bn=(n_1,n_2,n_3)\in\N^3$ and $\bsigma,\btau\in\listn$.
Set $\veca=\type(\bsigma)$ and $\vecb=\type(\btau)$.
Suppose that, for each $*\in\mc A_3$, the sets $A_*$ and $B_*$ are compact subsets of the complex plane $\C$, all contained in the disk of radius $2$ centered at the origin, and that any two distinct sets among these twelve sets are separated by a distance of at least $d>0$.
Then
\beq \label{eq:aori}
|\Pi^{\bsigma}_{\btau}(\bsxi,\bseta)|
\le
\frac{2^{4|\vecn|}}{d^{|\veca|+|\vecb|}}\sqrt{
n_1!(n_2-n_1+a_1+b_1)!(n_2-n_3+a_3+b_3)!n_3!}
\eeq
for every
$\bsxi=(\bsxi^{123},\bsxi^{12},\bsxi^{23},\bsxi^1,\bsxi^2,\bsxi^3)$ and $\bseta=(\bseta^{123},\bseta^{12},\bseta^{23},\bseta^1,\bseta^2,\bseta^3)$ such that $\bsxi^*\in(A_*)^{a_*}$ and $\bseta^*\in(B_*)^{b_*}$ for each $*\in\mc A_3$.
\end{lem}

\begin{proof}
By the definition \eqref{eq:Pisigmaxi}, $\Pi^{\bsigma}_{\btau}(\bsxi, \bseta)$ is the product of four Cauchy determinants of sizes $n_1, n_2-n_1+a_1+b_1, n_2-n_3+a_3+b_3, n_3$, respectively, and the polynomial $ \ST_{n_3}(\bsxi_{123}, \bsxi_{23}, \bsxi_3 |  \bseta_{123} , \bseta_{23}, \bseta_3)$ is given by \eqref{eq:STdf}. 
Hadamard's inequality implies that 
\beqq
    |\K_n(\ms r|\ms s)| \le \prod_{j=1}^n \left( \sum_{i=1}^n \frac{1}{|r_i-s_j|^2} \right)^{1/2} \le \frac{n^{n/2}}{(\min_{i,j}|r_i-s_j|)^n} \le \frac{2^n \sqrt{n!}}{(\min_{i,j}|r_i-s_j|)^n}. 
\eeqq 
On the other hand, $|\ST_n(\bfr | \bfs)| \le n \max_{i=1}^n |r_i-s_i|$. Thus, we obtain 
\beqq
    |\Pi^{\bsigma}_{\btau}(\bsxi, \bseta) | \le 
    4 n_3 \frac{2^{2n_2+a_1+b_1+a_3+b_3}\sqrt{n_1! (n_2-n_1+a_1+b_1)!(n_2-n_3+a_3+b_3)! n_3!}}{d^{2n_2+a_1+b_1+a_3+b_3}} 
\eeqq
for $(\bsxi, \bseta)$ on the contour. 
From \eqref{eq:atyeq}, we have $2n_2+a_1+b_1+a_3+b_3=|\veca|+|\vecb|\le 2|\vecn|$. 
Furthermore, since $4 n_3\le 4^{n_3}\le 2^{2|\vecn|}$, we obtain the result.
\end{proof}

\begin{lem} \label{prop:instest}
For every $\w \in \rgo\cup \cdots \cup \rgse$, define the constants 
\begin{equation} \label{eq:DelGdef}
    \Delta \mcG_{*} = \mcG_{*}(z_*^+) - \mcG_{*}(z_*^-), \qquad * \in \{1, 2, 3, 12, 23, 123\}. 
\end{equation}
There exist constants $C, c, L_0>0$ such that for every $L \ge L_0$ and $\vecn\in \N^3$, 
\beq \label{eq:Iste1}
    |\ing^{\bsigma}_{\btau}| 
    \le C^{|\vecn|} \sqrt{n_1! (n_2-n_1+a_1+b_1)! (n_2-n_3+a_3+b_3)!n_3!} 
    \, L^{3|\bv|} e^{c|\bv| L^{1/2}}
    e^{- L \sum_{* \in \mc A_3} v_* \De \mcG_*}
\eeq
for every $\bsigma, \btau\in \listn$  of the following forms: 
\begin{enumerate} [(a)]
\item 
$\bsigma =2^{a_2} (23)^{a_{23}}3^{a_3}(123)^{a_{123}}(12)^{a_{12}}1^{a_1}$ and 
$\btau = 3^{b_3'}2^{b_2'}1^{b_{1}}2^{b_2''}3^{b_3''}$ 
if $\w\in \rgo\cup \rgt$; 
\item 
$\bsigma = 3^{a_3}(23)^{a_{23}}2^{a_2}(123)^{a_{123}}(12)^{a_{12}}1^{a_1}$ and $\btau = 3^{b_3'}2^{b_2'}1^{b_{1}}2^{b_2''}3^{b_3''}$  if $\w \in \rgth\cup \rgf$; 
\item 
$\bsigma = 3^{a_3}(23)^{a_{23}}(123)^{a_{123}}1^{a_1} (12)^{a_{12}} 2^{a_2}$ and $\btau = 3^{b_3'}2^{b_2'}1^{b_{1}}2^{b_2''}3^{b_3''}$   if $\w \in \rgfi$; 
\item 
$\bsigma = 3^{a_3}(23)^{a_{23}}(123)^{a_{123}}1^{a_1} (12)^{a_{12}} 2^{a_2}$ and $\btau = 2^{b_2}(23)^{b_{23}}3^{b_{3}'}(12)^{b_{12}}1^{b_1}3^{b_{3}''}$ if $\w \in \rgs\cup \rgse$,
\end{enumerate}
where $\bv=(v_{123}, v_{12}, v_{23}, v_1, v_2, v_3)$ is given by 
\beqq
    \bv= \begin{cases}
    (a_{123}, a_{12}, a_{23}, a_1, a_2, a_3) \qquad \text{for $\w\in \rgo\cup\cdots \cup \rgfi$,}\\
    (a_{123}, a_{12}, a_{23}, a_1, b_2, a_3) \qquad \text{for $\w\in \rgs\cup \rgse$.}
    \end{cases}
\eeqq
Furthermore, if $\w\in \rgo$ and $a_2>0$ in (a), then $\ing^{\bsigma}_{\btau}=0$. Similarly, if $\w\in \rgse$ and $b_2>0$ in (d), then $\ing^{\bsigma}_{\btau}=0$. 
\end{lem}

\begin{proof}
Suppose that $\w\in \rgo$ and $a_2>0$. Since $\w\in \rgo$, \eqref{eq:signofxy} implies that $x_2-x_1<0$ and $y_2-y_1>0$. 
Thus, recalling \eqref{eq:MNTLiorj},  $M_{L,2}-M_{L,1}= \lceil x_2 aL \rceil - \lceil x_1 aL \rceil$ is a non-positive integer
and $N_{L,2}-N_{L,1}= \lceil y_2 bL \rceil - \lceil y_1 bL \rceil$ is a non-negative integer. 
Therefore, 
\beqq
    f_{L,2}(z)= \frac{z^{N_{L,2}-N_{L,1}}e^{(T_{L,2}-T_{L,1})z}}{(1+z)^{M_{L,2}-M_{L,1}}} 
\eeqq
is analytic at $z=-1$. Hence, the function $\FF_{L}^{\bsigma|\btau} (\bsxi, \bseta)$ is analytic at $\xi^{2}_i=-1$. 
Since the $\xi^{2}_i$-contour is the innermost among all $\bsxi$-contours, it follows by Cauchy's theorem that 
$\ing^{\bsigma}_{\tau}=0$. 

Similarly, suppose that $\w\in \rgse$ and $b_2>0$. 
Since $\w\in \rgse$, \eqref{eq:signofxy} implies that $x_2-x_1>0$ and $y_2-y_1<0$. 
Thus, in this case, $\frac1{f_{L,2}(z)}$ is analytic at $z=0$, so  $\FF_{L}^{\bsigma|\btau} (\bsxi, \bseta)$ is analytic at $\eta^{2}_i=0$. 
Again, the $\eta^2_i$-contour is the innermost among all $\bseta$-contours, and thus $\ing^{\bsigma}_{\tau}=0$ by Cauchy's theorem.

In what follows, we assume that $a_2 = 0$ if $\w \in \rgo$ and $b_2 = 0$ if $\w \in \rgse$.

(i) Let $\w \in  \rgt$ and consider the integral $\ing^{\bsigma}_{\btau}$, where $\bsigma$ and $\btau$ are as in (a). 
From Lemma \ref{cor:criticalptofGlo}, the critical points satisfy 
\beqq
    -1 < z^-_2 < z^-_{23} < z^-_{3} < z_{123}^- < z^-_{12} < z^-_1 < z_1^+=z^+_2 = z^+_3 =\pp< 0.
\eeqq 
For each $*$, we take the $\xi_i^*$-contour to be the circle $\{ z\in \C \, : \, |z+1|= |1 + z^-_*| \}$. 
On the other hand, we take the contours for the $\bseta$-variables to be
\beqq
    ( \Sigma_{1,L})^{b_3'} \times (\Sigma_{2,L})^{b_2'} \times (\Sigma_{3,L})^{b_1} \times (\Sigma_{4,L})^{b_2''} \times (\Sigma_{5,L})^{b_3''}, 
\eeqq
where $\Sigma_{k, L} = \{ z\in \C \, : \, |z|= |\pp | - (6-k)L^{-1/2} \}$. 
We may choose these contours as above without changing the value of the integral, since $\bsigma$ and $\btau$ are of the forms specified in (a).

Note that all circles are contained in the disk of radius $2$ centered at the origin, and each pair is separated  by $L^{-1/2}$, for all sufficiently large $L$.
From Lemma \ref{lem:Pibound}, with $d=L^{-1/2}$, and using $|\veca|+|\vecb|\le |\veca|+|\vecn|\le 4|\veca|$ from \eqref{eq:atyeq}, 
we find that
\beq \label{eq:aori}
    |\Pi^{\bsigma}_{\btau}(\bsxi, \bseta) | \le 
    2^{4|\vecn|} L^{2|\veca|} 
    \sqrt{n_1! (n_2-n_1+a_1+b_1)! (n_2-n_3+a_3+b_3)! n_3!}
\eeq 
uniformly for every $(\bsxi, \bseta)$ on the contour, for all sufficiently large $L$.  

On the other hand, from Lemmas \ref{lem:fl_asym_sameside} and  \ref{lem:asymptotics_f2}, 
also using $|\veca|+|\vecb|\le 2|\vecn|$, there exists a constant $C>0$, independent of $\bn$ and $L$, such that
\beq \label{eq:FLFF}
    \left|\FF_{L}^{\bsigma|\btau} (\bsxi, \bseta)\right| 
    = \prod_{*\in \mc A_3}   \frac{\prod_{i=1}^{a_*} \left| f_{L, *} (\xi_i^*)\right|}{ \prod_{i=1}^{b_*} \left| f_{L, *} (\eta_i^*)\right|} 
    \leq C^{2|\vecn|} \prod_{*\in \mc A_3} \frac{\left| f_{L, *} (z^-_*)\right|^{a_*}}{\left| f_{L, *} (z^+_*)\right|^{b_*}}
\eeq 
uniformly for every $(\bsxi, \bseta)$ on the contour, for all sufficiently large $L$.  
From \eqref{eq:ftwoth} and the relation  \eqref{eq:atypen}, we find that  
\beq \label{eq:prdfo}
    \prod_{*\in \mc A_3} |f_{L, *}(z)|^{a_*}
    = \prod_{i=1}^3 |f_{L,i}(z)|^{n_i} = \prod_{*\in \mc A_3} |f_{L, *}(z)|^{b_*}
    \qquad \text{for every $z$.} 
\eeq
Thus, since $z_*^+=\pp$ for all $*$, we have 
\beqq
    \prod_{*\in \mc A_3} \frac{\left| f_{L, *} (z^-_*)\right|^{a_*}}{\left| f_{L, *} (z^+_*)\right|^{b_*}}
    = \prod_{*\in \mc A_3} \frac{\left| f_{L, *} (z^-_*)\right|^{a_*}}{\left| f_{L, *} (\pp)\right|^{a_*}}
    = \prod_{*\in \mc A_3}  \left| \frac{ f_{L, *} (z^-_*) }{f_{L, *} (z^+_*)} \right|^{a_*} . 
\eeqq
Hence, from the formula \eqref{eq:fLstard}, we find that there exists a constant $c>0$ such that 
\beq \label{eq:forif}
    \left|\FF_{L}^{\bsigma|\btau} (\bsxi, \bseta)\right| 
    \le  C^{2|\vecn|} e^{c |\veca| L^{1/2}} e^{- L \sum_{* \in \mc A_3} a_* \De \mcG_*} 
\eeq
uniformly for every $(\bsxi, \bseta)$ on the contour, for all sufficiently large $L$.  
Applying estimates \eqref{eq:aori} and \eqref{eq:forif} to \eqref{eq:intaga} yields \eqref{eq:Iste1}, possibly after adjusting the constants.

(ii) When $\w\in \rgo$, the analysis is the same as in (i), except that we do not have the $\bsxi^2$-integrals since we assumed that $a_2=0$. 

(iii) When $\w \in \rgth\cup \rgf$ or  $\w \in \rgfi$, the proof is nearly the same as in (i). We omit the details. 

(iv) Let $\w \in \rgs$. The analysis is again similar, but with some modifications, since in this case the critical points satisfy
\beqq
    -1 < z^-_3 < z^-_{23} < z_{123}^- < z^-_{1} < z^-_{12} < z^-_2 = z^+_{1} = z^+_3 = z^+_{12}  = z^+_{23}  =\pp < z^+_2 < 0.
\eeqq
Now $z_2^-=\pp<z_2^+$, unlike in cases (i)--(iii). 
For every $*$, we take the $\xi_i^*$-contour to be the circle $\{ z\in \C \, : \, |z+1|= |1 + z^-_*| \}$. 
On the other hand, we take the contours for the $\bseta$-variables to be
\beqq
    \Sigma^{b_2} \times ( \Sigma_{1,L})^{b_{23}} \times (\Sigma_{2,L})^{b_{3}'} \times (\Sigma_{3,L})^{b_{12}} \times (\Sigma_{4,L})^{b_{1}} \times (\Sigma_{5,L})^{b_3''}  
\eeqq
where $\Sigma= \{ z\in \C \, : \, |z|= |z^+_2| \}$ 
and $\Sigma_{k, L} = \{ z\in \C \, : \, |z|= |\pp | - (6-k)L^{-1/2} \}$. 

From \eqref{eq:atypen}, we see that
$a_2\le n_2= b_2+b_{12}+b_{23}+b_{123}$ and 
$b_1+b_3+2(b_{12}+b_{23}+b_{123})\le 2(n_1+n_3)=2(a_1+a_3+a_{12}+a_{23}+2a_{123})\le 4 (|\veca|-a_2)$. 
Thus,
\beqq
    |\veca|+|\vecb| \le 
    (|\veca|-a_2)+ 2b_2+b_1+b_3+2(b_{12}+b_{23}+b_{123})
    \le 2b_2+ 5 (|\veca|-a_2)\le 6(|\veca|-a_2+b_2) .
\eeqq
Hence, Lemma \ref{lem:Pibound} with $d=L^{-1/2}$ implies that 
\beqq
    |\Pi^{\bsigma}_{\btau}(\bsxi, \bseta) | \le  2^{4|\vecn|} L^{3(|\veca|-a_2+b_2)}
    \sqrt{n_1! (n_2-n_1+a_1+b_1)! (n_2-n_3+a_3+b_3)!n_3!} 
\eeqq
uniformly for every $(\bsxi, \bseta)$ on the contour, for all sufficiently large $L$.  

From Lemmas \ref{lem:fl_asym_sameside} and  \ref{lem:asymptotics_f2}, the estimate \eqref{eq:FLFF} still holds. 
Since $z_2^-  = z^+_{1} = z^+_{3}= z^+_{12} = z^+_{23} = \pp$, the identity \eqref{eq:prdfo} implies that  
\beqq
    \prod_{*\in \mc A_3} \frac{\left| f_{L, *} (z^-_*)\right|^{a_*}}{\left| f_{L, *} (z^+_*)\right|^{b_*}}
    = \frac{ |f_{L, 2}(\pp)|^{b_2}}{|f_{L, 2}(z_2^+)|^{b_2}} 
     \prod_{*\neq 2}   \frac{ \left| f_{L, *} (z^-_*) \right|^{a_*}}{ \left|f_{L, *} (\pp)\right|^{a_*}}  
    = \frac{ |f_{L, 2}(z_2^-)|^{b_2}}{|f_{L, 2}(z_2^+)|^{b_2}}  \prod_{*\neq 2}  \left| \frac{ f_{L, *} (z^-_*) }{f_{L, *} (z^+_*)} \right|^{a_*} . 
\eeqq
Hence, there exists a constant $c>0$ such that 
\beqq
    \left|\FF_{L}^{\bsigma|\btau} (\bsxi, \bseta)\right| 
    \le  C^{2|\vecn|} e^{c (|\veca|-a_2+b_2) L^{1/2}} e^{- L (b_2\De\mcG_2+ \sum_{* \neq 2} a_*\De \mcG_* ) } 
\eeqq
uniformly for $(\bsxi, \bseta)$ on the contour, for all sufficiently large $L$. 
Hence, \eqref{eq:Iste1} follows, after adjusting the constants if necessary.

(v) When $\w \in  \rgse$, the analysis is the same as in (iv), except that there are no $\eta^2_i$-integrals since we have assumed that $b_2=0$.
\end{proof}

\medskip

We now estimate the terms $\sum_{* \in \mc A_3} v_* \De \mcG_*$ appearing in \eqref{eq:Iste1}. 
This estimate is provided in Lemma \ref{lem:deGlb} below, which makes use of the following two lemmas. 

\begin{lem} \label{lem:Gprop}
\begin{enumerate} [(a)]
    \item If $\w\in \rgo\cup \cdots \cup \rgse$, then for every $*\in \{1,3,12, 23, 123\}$, the function 
    $\mcG_*$ is strictly decreasing on $(-1, z_*^-]$, strictly increasing on $[z_*^-, z_*^+]$, and strictly decreasing on $[z_*^+, 0)$. 
    \item If $\w\in \rgt\cup \cdots \cup \rgs$,  then $\mcG_2$ is strictly decreasing on $(-1, z_2^-]$, strictly increasing on $[z_2^-, z_2^+]$, and strictly decreasing on $[z_2^+, 0)$. 
    \item If $\w\in \rgo$, then $\mcG_2$ is strictly increasing on $(-1, z_2^+]$ and strictly decreasing on $[z_2^+, 0)$. 
    \item If $\w\in \rgse$, then $\mcG_2$ is strictly decreasing on $(-1, z_2^-]$ and strictly increasing on $[z_2^-, 0)$. 
\end{enumerate}
\end{lem}

\begin{proof}
The proof follows from Lemma \ref{lem:general_crit_behavior}, since $\mcG'_*(z) = \frac{q_*(z)}{z(z+1)}$ for a convex quadratic polynomial  $q_*$.
\end{proof}

\begin{lem} \label{eq:delGin} 
For every $\w\in \rgo\cup \cdots \cup \rgse$, there exists a constant $\epsilon_0>0$ such that the following statements hold: 
\begin{enumerate} [(a)]
    \item $\Delta \mcG_{1} + \Delta \mcG_{2} + \Delta \mcG_{3} \ge \Delta \mcG_{123} + \epsilon_0$
    if $\w  \in \rgt \cup \cdots \cup \rgs$. 
    \item $\Delta \mcG_{12} + \Delta \mcG_{3} \ge \Delta \mcG_{123} + \epsilon_0 $ if $\w\in \rgo\cup \cdots \cup \rgse$. 
    \item $\Delta \mcG_{1} + \Delta \mcG_{23} \ge \Delta \mcG_{123} +\epsilon_0 $ if $\w\in \rgo\cup \cdots \cup \rgse$. 
    \item $\Delta \mcG_{1} + \Delta \mcG_{3} \ge \Delta \mcG_{123} + \epsilon_0 $ if $\w  \in \rgs \cup \rgse$.
    \item $\Delta \mcG_{12} + \Delta \mcG_{23} \ge \Delta \mcG_{123} +  \epsilon_0$  if $\w \in \rgo \cup \cdots \cup \rgf$.  
\end{enumerate}
\end{lem}

\begin{proof}
\begin{enumerate} [(a)]
\item Suppose $\w  \in \rgt \cup \cdots \cup \rgs$. 
By definition, $\mcG_{1} + \mcG_{2} +  \mcG_{3}= \mcG_{123}$. 
If $\w  \in \rgt \cup \cdots \cup \rgfi$, then $z_*^+=\pp$ for all $*\in \{1,2,3,123\}$ by Lemma \ref{cor:criticalptofGlo}, and thus, $\mcG_{1}(z_1^+) + \mcG_{2}(z_2^+) +  \mcG_{3}(z_3^+)= \mcG_{123}(\pp)$. 
If $\w \in \rgs$, then by Lemma \ref{lem:Gprop} (b), $\mcG_{2}(\pp) = \mcG_2(z^-_2) < \mcG_{2}(z^+_{2})$, and thus, 
$\mcG_{1}(z_1^+) + \mcG_{2}(z_2^+) +  \mcG_{3}(z_3^+) > \mcG_{1}(\pp) + \mcG_{2}(\pp) +  \mcG_{3}(\pp) = \mcG_{123}(\pp)$.
In either case, $\mcG_{1}(z_1^+) + \mcG_{2}(z_2^+) +  \mcG_{3}(z_3^+) \ge \mcG_{123}(z_{123}^+)$.
On the other hand, since $z_{123}^-\in  (-1, z_*^+)\setminus \{z_*^-\}$ 
for each $*=1,2,3$ by Lemma 
\ref{cor:criticalptofGlo}, Lemma \ref{lem:Gprop} (a) and (b) imply that $ \mcG_{*}(z_{123}^-)>\mcG_*(z_*^-)$ for $*=1,2,3$. Thus, $\mcG_{1}(z_1^-) + \mcG_{2}(z_2^-) +  \mcG_{3}(z_3^-)< \mcG_{1}(z_{123}^-) + \mcG_{2}(z_{123}^-) +  \mcG_{3}(z_{123}^-)=  \mcG_{123}(z_{123}^-)$. Therefore, we find that $\Delta \mcG_{1} + \Delta \mcG_{2} + \Delta \mcG_{3} > \Delta \mcG_{123}$. This implies the result. 
     
\item Suppose $\w\in \rgo\cup \cdots \cup \rgse$. The result follows by noting that $\mcG_{12} + \mcG_{3}= \mcG_{123}$, that $z_*^+=\pp$ for all $*\in\{12, 3, 123\}$, and that $z_{123}^-\in  (-1, z_*^+) \setminus \{z_*^-\}$  for $*\in \{12,3\}$. 
    
\item Suppose $\w\in \rgo\cup \cdots \cup \rgse$. The result follows by noting that  $\mcG_{1} + \mcG_{23}= \mcG_{123}$, that $z_*^+=\pp$ for all $*\in\{1, 23, 123\}$, and that $z_{123}^-\in  (-1, z_*^+) \setminus \{z_*^-\}$  for $*\in \{1, 23\}$.
    
\item Suppose $\w  \in \rgs \cup \rgse$. Note that $\mcG_1+\mcG_3=\mcG_{123}-\mcG_2$. 
Since $z_1^+=z_3^+=z_{123}^+= z_2^-$, we have
$\mcG_1(z_1^+)+ \mcG_3(z_3^+) = \mcG_{123}(z_{123}^+) - \mcG_2(z_2^-)$. 
On the other hand, since $z_{123}^-\in  (-1, z_*^+) \setminus \{z_*^-\}$ for $*=1,3$, Lemma \ref{lem:Gprop} (a) implies that 
$\mcG_1(z_1^-) + \mcG_3(z_3^-)< \mcG_1(z_{123}^-)+  \mcG_3(z_{123}^-) = \mcG_{123}(z_{123}^-)- \mcG_2(z_{123}^-)$. 
Hence, $\Delta \mcG_1 + \Delta \mcG_3 >  \Delta \mcG_{123} + \mcG_2(z_{123}^-)- \mcG_2(z_2^-)$. 
Since $z_{123}^-\in (-1, z_2^-)$, Lemma \ref{lem:Gprop} (b) and (d) imply that $\mcG_2(z_{123}^-)>\mcG_2(z_2^-)$. 
Thus, we obtain the result. 

\item Suppose $\w  \in \rgo \cup \cdots \cup \rgf$. The proof is trickier in this case. 
Since $\mcG_{12}+\mcG_{23}=\mcG_{123}+\mcG_2$ and $z_*^+$ are equal for all $*$, $\mcG_{12}(z_{12}^+)+\mcG_{23}(z_{23}^+) = \mcG_{123}(z_{123}^+)+\mcG_{2}(z_{2}^+)$. 
On the other hand, since $z_{123}^-\in  (z_{23}^- , z_{23}^+)$, Lemma \ref{lem:Gprop} (a) implies that $\mcG_{23}(z_{23}^-) < \mcG_{23}(z_{123}^-)$. 
Furthermore, noting that $z^-_1 \in (z_{12}^-, z^+_{12})$, we find from Lemma \ref{lem:Gprop} (a) that 
$\mcG_{12}(z_{12}^-)< \mcG_{12}(z_1^-)= \mcG_{1}(z_1^-)+\mcG_{2}(z_1^-)$. 
Since $z_{123}^-\in (-1, z_1^-)$, applying Lemma \ref{lem:Gprop} (a) again, we have $\mcG_{1}(z_{1}^-)<\mcG_{1}(z_{123}^-)$.  
Hence, $\mcG_{12}(z_{12}^-)+\mcG_{23}(z_{23}^-) < \mcG_{1}(z_{123}^-)+\mcG_{2}(z_1^-)+\mcG_{23}(z_{123}^-)= \mcG_{123}(z_{123}^-)+\mcG_{2}(z_1^-)$. 
Thus, $\Delta \mcG_{12} + \Delta \mcG_{23} > \Delta \mcG_{123} + \mcG_{2}(z_2^+)-\mcG_{2}(z_1^-)$. 
Since $z_1^-\in (\max\{-1, z_2^-\}, z_2^+)$, Lemma \ref{lem:Gprop} (a) and (c) imply that $\mcG_2(z_2^+)> \mcG_2(z_{1}^-)$. We thus obtain the result.  
\end{enumerate}
\end{proof}

We now estimate $\sum_{* \in \mc A_3} v_* \De \mcG_*$. We note that, for positive integers $p$ and $q$, 
\beq \label{eq:Mktr}
    \text{if $p\ge q$, then $p-q+1\ge \frac{p}{q} $.} 
\eeq

\begin{lem} \label{lem:deGlb}
For every $\w \in \rgo \cup \cdots \cup \rgse$,  there exists a constant $c > 0$ such that 
\begin{equation} \label{eq:llll}
    \sum_{* \in \mc A_3} v_*\De \mcG_{*}  \geq \Delta \mcG_{123} + c|\bv| 
\end{equation}  
for every 
$\bv= (v_{123}, v_{12}, v_{23}, v_1, v_2, v_3)\in \N_0^6 \setminus \{(1,0,0,0,0,0)\}$ satisfying
\beq \label{eq:acases0}
    v_1+v_{12}+v_{123}\ge 1, \qquad v_3+v_{23}+v_{123}\ge 1 
\eeq 
with the following extra assumptions:
\begin{itemize}
\item $v_2+v_{12}+v_{23}+v_{123}\ge 1$ when $\w\in \rgo\cup \cdots \cup \rgfi$, 
\item $v_2=0$ when $\w\in \rgo\cup \rgse$, 
\item $v_{123}+v_1\ge 1$ when $\w\in \rgfi \cup \rgs\cup \rgse$. 
\end{itemize}
\end{lem}

\begin{proof} 
Fix $\w \in \rgo \cup \cdots \cup \rgse$. 
Let $\epsilon_0>0$ be the constant from Lemma \ref{eq:delGin}. Define the constants
\beq \label{eq:conc12} \begin{split}
    &c_1= \min\left\{ \De \mcG_{123} , \De \mcG_{12}, \De \mcG_{23},\De \mcG_{1}, \De \mcG_{2}, \De \mcG_{3}  \right\} ,
    \qquad 
    c_2= \min\{c_1, \epsilon_0 \}, \\    
    &c_1'= \min\left\{ \De \mcG_{123} , \De \mcG_{12}, \De \mcG_{23},\De \mcG_{1}, \De \mcG_{3}  \right\}, 
    \qquad
    c_2'= \min\{ c_1', \epsilon_0 \} . 
\end{split}  \eeq 
Lemma \ref{lem:Gprop} (a) and (b) imply that $c_1, c_2>0$ when $\w\in  \rgt \cup \cdots \cup \rgs$, 
and $c_1', c_2'>0$ when $\w \in  \rgo \cup \cdots \cup \rgse$. 
Set $\tn{LHS} := \sum_{* \in \mc A_3} v_*\De \mcG_{*}$. 

\begin{enumerate}[(a)]
\item Suppose $\w\in  \rgt\cup \rgth \cup \rgf$. 

\begin{itemize}
\item Suppose $v_{123}\ge 1$. Then  $\tn{LHS} \ge \Delta \mcG_{123} + (|\bv|-1)c_1$. 
If $v_{123}\ge 2$, then $|\bv| \ge 2$. 
On the other hand, if $v_{123}=1$, then by the assumption that $\bv \neq (1,0,0,0,0,0)$,  there is at least one $*\neq 123$ with $v_*\ge 1$. 
Hence, $|\bv| \ge 2$ in this case as well. 
Thus, from \eqref{eq:Mktr}, we find that $\tn{LHS}\ge\Delta \mcG_{123}+\frac12|\bv| c_1$. 

\item Suppose $v_{123}=0$. 
    \begin{enumerate}[(i)]
    \item Suppose $v_{12}\ge 1$ and $v_{23}\ge 1$. 
    Then,  $\tn{LHS} \ge \De \mcG_{12} + \De \mcG_{23} + (|\bv| - 2)c_1 \ge \Delta \mcG_{123}  + \epsilon_0 + (|\bv|-2) c_1$ by Lemma \ref{eq:delGin} (e). Since $\epsilon_0\ge c_2$ and $c_1\ge c_2$, we find
    $\tn{LHS} \ge \Delta \mcG_{123}  +  (|\bv|-1) c_2$. Since $|\bv |\ge v_{12}+v_{23}\ge 2$, we conclude from \eqref{eq:Mktr} that $\tn{LHS} \ge \Delta \mcG_{123}  +  \frac12 |\bv| c_2$. 
  
    \item Suppose $v_{12}\ge 1$ and $v_{23}=0$. Then the second inequality of \eqref{eq:acases0} implies $v_{3}\ge 1$. Using Lemma \ref{eq:delGin} (b) and the fact that  $|\bv|\ge 2$, we obtain 
    $\tn{LHS} \ge \De \mcG_{12} + \De \mcG_{3} + (|\bv| - 2)c_1\ge \Delta \mcG_{123} +  \epsilon_0 + (|\bv|-2)c_1
    \ge \Delta \mcG_{123} + (|\bv|-1)c_2 \ge\Delta \mcG_{123}+ \frac12 |\bv| c_2$ . 
    
    \item Suppose $v_{12}=0$ and $v_{23}\ge 1$. Then the first inequality of \eqref{eq:acases0} implies $v_{1}\ge 1$. Thus, applying Lemma \ref{eq:delGin} (c), we again obtain $\tn{LHS} \ge\Delta \mcG_{123}+ \frac12 |\bv| c_2$. 
    
    \item Suppose $v_{12}=v_{23}=0$.  Then, $v_1,v_3 \ge 1$ by \eqref{eq:acases0}. Additionally, from the condition $v_2+v_{12}+v_{23}+v_{123}\ge 1$, we find that $v_2\ge 1$. Thus, by Lemma \ref{eq:delGin} (a) and the fact that $|\bv|\ge 3$, we obtain $\tn{LHS}  \ge \De \mcG_1 + \De \mcG_{2} + \De \mcG_{3} + (|\bv| - 3)c_1 \ge \De \mcG_{123}+ \epsilon_0 + (|\bv| - 3)c_1  \ge \De \mcG_{123}+ (|\bv| - 2)c_2\ge \Delta \mcG_{123} +\frac13  |\bv| c_2$. 
    \end{enumerate}
\end{itemize}

\item Suppose $\w \in \rgfi$. The proof of case (a) holds except for part (i), since Lemma \ref{eq:delGin} (e) is not applicable when $\w \in \rgfi$. The part (i) is modified as follows. 
    \begin{itemize}
    \item Suppose $v_{123}=0$, $v_{12}\ge 1$, and $v_{23}\ge 1$. Since $v_{123}+v_1\ge 1$ when $\w\in \rgfi$, we find that $v_1 \ge 1$. Thus, by Lemma \ref{eq:delGin} (c), $\tn{LHS} \ge \De \mcG_{1} + \De \mcG_{23} +  (|\bv| - 2)c_1
    \ge \De \mcG_{123} + \epsilon_0 +   (|\bv| - 2)c_1 \ge \Delta \mcG_{123} + (|\bv|-1)c_2 \ge \Delta \mcG_{123} +\frac12 |\bv|c_2$, since $|\bv|\ge 3\ge 2$. 
    \end{itemize}

\item Suppose $\w \in \rgo$. 
The proof of case (a) again holds, with $c_1$ and $c_2$ replaced by $c_1'$ and $c_2'$, respectively, with $v_2 = 0$, except for  part (iv), since Lemma \ref{eq:delGin} (a) is not applicable when $\w\in \rgo$. 
However, part (iv) does not occur because $v_2 = 0$ by assumption. 

\item Suppose $\w \in \rgs$. If $v_{123}\ge 1$, the result follows from the same proof as in case (a). On the other hand, suppose  $v_{123}=0$. Then, by the assumption that $v_{123}+v_1\ge 1$, we have $v_1\ge 1$. 
\begin{itemize}
    \item If $v_{123}=0$ and $v_{23} \ge 1$, then, $\tn{LHS} \ge \De \mcG_1 + \De \mcG_{23} +  (|\bv| - 2)c_1\ge \Delta \mcG_{123} + \epsilon_0 + (|\bv| - 2)c _1$ by Lemma \ref{eq:delGin} (c). Hence, 
     $\tn{LHS} \ge \Delta \mcG_{123} +  (|\bv| - 1)c _2 \ge \Delta \mcG_{123}+ \frac12 |\bv| c_2$ since $ |\bv|\ge 2$. 

    \item If $v_{123}=0$ and $v_{23} = 0$, then the second inequality of \eqref{eq:acases0} implies $v_{3}\ge 1$. 
    Thus, $\tn{LHS} \ge \De \mcG_1 + \De \mcG_{3} +  (|\bv| - 2)c_1$. By Lemma \ref{eq:delGin} (d), $\tn{LHS} \ge \Delta \mcG_{123} + \epsilon_0 + (|\bv| - 2)c_1 \ge \Delta \mcG_{123} + (|\bv| - 1)c_2 \ge \Delta \mcG_{123}+ \frac12 |\bv| c_2$ since $|\bv|\ge 2$. 
\end{itemize}

\item Suppose $\w \in \rgse$. The proof is exactly the same as case (d), with $c_1$ and $c_2$ replaced by $c_1'$ and $c_2'$, respectively, and with $v_2=0$.
\end{enumerate}
\end{proof}

We now prove Proposition \ref{cor:incor}. 

\begin{proof}[Proof of Proposition \ref{cor:incor}] 
If $\w \in \rgo$ and $a_2> 0$, or if  $\w \in \rgse$ and $b_2> 0$, then $\ing^{\bsigma}_{\tau}=0$ by Lemma \ref{prop:instest}. 
Thus, we may assume that $a_2=0$ if $\w\in \rgo$, and $b_2=0$ if $\w\in \rgse$. 
By Lemma \ref{prop:instest}, it remains to estimate $L^{3|\bv|} e^{c|\bv|L^{1/2}} e^{- L \sum_{* \in \mc A_3} v_* \De \mcG_*}$. 
We verify that Lemma \ref{lem:deGlb} applies. 
We have $v_1=a_1, v_3=a_3, v_{12}=a_{12}, v_{23}=a_{23}, v_{123}=a_{123}$, 
while $v_2=a_2$ if $\w\in \rgo\cup \cdots \cup \rgfi$, and $v_2=b_2$ if $\w\in \rgs\cup \rgse$.  
From \eqref{eq:atypen}, we see that $v_1+v_{12}+v_{123}=n_1\ge 1$ and $v_3+v_{23}+v_{123}=n_3\ge 1$, and thus \eqref{eq:acases0} is satisfied. 
Furthermore, if $\w\in \rgo\cup \cdots \cup \rgfi$, then $v_2+v_{12}+v_{23}+v_{123}=n_2\ge 1$ again by \eqref{eq:atypen}. 
If $\w\in \rgo\cup \rgse$, then $v_2=0$ by our assumption that $a_2=0$ if $\w\in \rgo$, and $b_2=0$ if $\w\in \rgse$. 
If $\w\in \rgfi \cup \rgs\cup \rgse$, then $v_{123}+v_1= a_{123}+a_1\ge 1$ by assumption. 
Therefore, all conditions of Lemma \ref{lem:deGlb} are satisfied, and hence there exists $c'>0$ such that $\sum_{* \in \mc A_3} v_* \De \mcG_* \ge \De \mcG_{123} +c'|\bv|$. 
Thus, there exists $L_0 > 0$ such that  
\beqq
    L^{3|\bv|} e^{c|\bv|L^{1/2}} e^{- L \sum_{* \in \mc A_3} v_* \De \mcG_*}
    \le  e^{ -L\De \mcG_{123} -|\bv|(c'L- cL^{1/2}-3\ln L)}  \le e^{-L \De \mcG_{123} - \frac12 c' |\bv|L}
\eeqq
for  every $L \ge L_0$. 
Now, from the explicit formulas, $\De \mcG_{123} = \J(\lv)$ in the equation \eqref{eq:large_deviation_fn}. 
Comparing with the formula \eqref{eq:cKdef2} of $\cK_L$, we see that $e^{-L \De \mcG_{123} }\le \cK_L$. 
Thus, we obtain the result. 
\end{proof}

\subsection{Deformation of Integrals and proof of Proposition \ref{cor:decomposition}} \label{sec:deform}

Unlike those in Proposition \ref{cor:incor}, the integrals arising in Proposition \ref{cor:decomposition} cannot be evaluated directly using the method of steepest descent, since the ordering of the critical points does not match the nesting of the contours.
To address this, we deform the contours and, after accounting for residues, rewrite these integrals as sums of integrals to which the method of steepest descent applies directly.
Because a residue term may produce integrals with critical points still incompatible with the contour nesting, this procedure must sometimes be repeated multiple times until all resulting integrals have compatible critical points and contour structures.
This reduction is accomplished by Lemmas \ref{lem:integralordering} and \ref{lem:integralordering2}.
In this way, we express the integrals appearing in Proposition \ref{cor:decomposition} as sums of those appearing in Proposition \ref{cor:incor}. These lemmas also yield the proofs of Lemmas \ref{lem:QQ111list} and \ref{lem:QQ121list}.
The formal proof of Proposition \ref{cor:decomposition} is given at the end of this subsection.

\begin{lem} \label{lem:defbsc}
Let $\Omega$ be a region, and let $\Gamma_1$, $\Gamma_2$, and $\Gamma_3$ be Jordan curves in $\Omega$ that are nested from innermost $(\Gamma_1)$ to outermost $(\Gamma_3)$, and can be continuously deformed into one another within $\Omega$. Let $F(\bu, \bv)$ be a meromorphic function with
\begin{itemize}
    \item simple poles at $u_i = v_j$ for all $i$ and $j$,
    \item no poles of $u_i$ and $v_j$ within $\Omega$ other than at $u_i=v_j$,
    \item zeros of order at least one at $u_i = u_j$ and at $v_i = v_j$ for all $i \neq j$,
    \item symmetry in $u_1, \ldots, u_m$, and separately in $v_1, \ldots, v_n$.
\end{itemize}
Then, 
\beq \label{eq:defel} \begin{split}
    &\int_{\Gamma_1^m} \dd \bu \int_{\Gamma_2^n} \dd \bv  \, \, F(\bu, \bv)  
    =  \sum_{i=0}^{m\wedge n}  (-2\pi \ii)^i i! \binom{m}{i} \binom{n}{i}   \int_{\Gamma_3^{m-i}} \dd u_{i+1}\cdots \dd u_m \int_{\Gamma_2^n} \dd \bv  \, \, \res_{u_1=v_1, \cdots, u_i=v_i} F(\bu, \bv).
\end{split} \eeq
\end{lem}

\begin{proof}
The general case can be readily proved by induction on $m$. We omit the details and instead illustrate the case $m = 2$ with $n \geq 2$ to show how the assumptions on $F$ are used. 
By moving the $\bu$-contour outside the $\bv$-contour, the Cauchy residue theorem implies  
\beqq \begin{split}
    &\int_{\Gamma_1^2} \dd u_1 \dd u_2 \int_{\Gamma_2^n} \dd \bv  \, \, F(\bu, \bv)  
    =  \int_{\Gamma_3^2} \dd u_1 \dd u_2 \int_{\Gamma_2^n} \dd \bv  \, \, F(\bu, \bv) 
    -  2\pi \ii \sum_{k=1}^n  \int_{\Gamma_3} \dd u_2 \int_{\Gamma_2^n} \dd \bv \res_{u_1=v_k} F(\bu, \bv)  \\
    &\qquad\qquad -  2\pi \ii \sum_{k=1}^n  \int_{\Gamma_3} \dd u_1 \int_{\Gamma_2^n} \dd \bv \res_{u_2=v_k} F(\bu, \bv)  
    + (2\pi \ii)^2 \sum_{k_1, k_2=1}^n \int_{\Gamma_2^n} \dd \bv \res_{u_1=v_{k_1}, u_2=v_{k_2}} F(\bu, \bv) . 
\end{split} \eeqq
Observe that $\res_{u_1=v_{k}, u_2=v_{k}} F(\bu, \bv)=0$, since $F$ has a zero at $u_1=u_2$ of order at least one and simple poles at $u_1=v_k$,  $u_2=v_k$. 
Thus, the last double sum is only over $k_1\neq k_2$. 
By the symmetry of $F$ in both $\bu$ and $\bv$, we have $\res_{u_1=v_k} F(\bu, \bv)= \res_{u_2=v_k} F(\bu, \bv)= \res_{u_1=v_1} F(\bu, \bv)$, and $\res_{u_1=v_{k_1}, u_2=v_{k_2}} F(\bu, \bv) = \res_{u_1=v_{1}, u_2=v_{2}} F(\bu, \bv)$ whenever $k_1\neq k_2$. 
Therefore, the two middle sums are equal, and the double sum runs over all pairs of distinct indices. Thus, collecting terms and using the symmetry, the right-hand side is equal to  
\beqq \begin{split}
    \int_{\Gamma_3^2} \dd \bu  \int_{\Gamma_2^n} \dd \bv  \, \, F(\bu, \bv) 
    +2  n (-2\pi \ii)  \int_{\Gamma_3} \dd u_2 \int_{\Gamma_2^n} \dd \bv \res_{u_1=v_1} F(\bu, \bv)
    + 2 \binom{n}{2} (2\pi \ii)^2  \int_{\Gamma_2^n} \dd \bv \res_{u_1=v_1, u_2=v_2} F(\bu, \bv), 
\end{split} \eeqq
which coincides with the right side of \eqref{eq:defel} when $m=2$ and $n\ge 2$. 
\end{proof}

\begin{cor} \label{lem:deformbasic}
Let $g$ and $h$ be analytic functions in a region $\Omega$. 
Suppose $\Gamma_1$, $\Gamma_2$, and $\Gamma_3$ are nested Jordan curves in $\Omega$ that can be continuously deformed into one another within $\Omega$. 
Let $m,n\in \N$. 
Let $a,b,c$ be integers such that $a\ge m$, $b\ge m\vee n$, and $c\ge n$. 
For all vectors $\bfr_1\in (\C \setminus \Omega)^a$, $(\bfs_1, \bfs_2) \in \C^{a-m}$, $(\bfs_3, \bfs_4)\in (\C \setminus \Omega)^{b-m}$, $(\bfs_5, \bfs_6)\in (\C \setminus \Omega)^{b-n}$,  $(\bfs_7, \bfs_8)\in \C^{c-n}$, and $\bfr_2\in (\C \setminus \Omega)^c$ for which the integrand of the below integral is defined, we have
\beq \label{eq:aGaa} 
\begin{split}
    &\int_{\Gamma_1^m} \dd \bu \int_{\Gamma_3^n} \dd \bv  \, \, 
    \K_a(\bfr_1 | \bfs_1, \bu, \bfs_2)  \K_b(\bfs_3, \bu, \bfs_4|  \bfs_5, \bv, \bfs_6 )  \K_c(\bfs_7, \bv, \bfs_8 |  \bfr_2)  \prod_{p=1}^m g(u_p) \prod_{q=1}^n h(v_q) \\
    &=  \sum_{i=0}^{m\wedge n} (-1)^{\#} (2\pi \ii)^i i! \binom{m}{i} \binom{n}{i} R_i, 
\end{split} \eeq
where $\#$ denotes an integer whose precise value is not specified here\footnote{If $\Gamma_1, \Gamma_2, \Gamma_3$ are nested such that $\Gamma_1$ is the innermost curve and $\Gamma_3$ is the outermost, then $\#=i(1+\sd(\bfs_3) +\sd(\bfs_5))$, where $\sd(\bfs)=d$ for vectors $\bfs\in \C^d$. If $\Gamma_1$ is the outermost and $\Gamma_3$ is the innermost, then $\#=i(\sd(\bfs_3) +\sd(\bfs_5))$.} and 
\beq \label{eq:RR} \begin{split}
    R_i &= \int_{\Gamma_1^{n-i}} \dd \bv  \int_{\Gamma_2^i} \dd \bw \int_{\Gamma_3^{m-i}} \dd \bu \, \,  
    \K_a(\bfr_1 | \bfs_1,  \bw,  \bu, \bfs_2)  \K_{b-i}(\bfs_3, \bu, \bfs_4|  \bfs_5, \bv, \bfs_6 )  \K_c(\bfs_7, \bw, \bv, \bfs_8 |  \bfr_2) \\
    &\qquad \qquad  \qquad \times \prod_{p=1}^{m-i} g(u_p) \prod_{q=1}^{n-i} h(v_q) \prod_{r=1}^i g(w_r)h(w_r). 
\end{split} \eeq
\end{cor}

\begin{proof} 
Assume that $\Gamma_1$ is the innermost curve and $\Gamma_3$ is the outermost; the proof is similar if the nesting order is reversed.
From the Cauchy determinant formula, the integrand on the left side of \eqref{eq:aGaa} satisfies the conditions of Lemma \ref{lem:defbsc}. We now compute the residues. 
The Cauchy determinant formula also implies that 
\beqq
    \res_{u_1=v_1, \cdots, u_i=v_i}  \K_b(\bfs_3, \bu, \bfs_4|  \bfs_5, \bv, \bfs_6 )
    = \pm  \K_{b-i}(\bfs_3, \hat \bu, \bfs_4|  \bfs_5, \hat \bv, \bfs_6 ) 
\eeqq
where $\hat \bu$ denotes the vector $\bu$ with entries $u_1,\cdots, u_i$ removed and $\hat \bv$ denotes $\bv$ with entries $v_1,\cdots, v_i$ removed. The sign is not specified, as it is not relevant for our purposes. 
Thus, the result follows from Lemma \ref{lem:defbsc}, after setting the variables
\beqq
    (v_1, \cdots, v_i)=\bw, \quad (u_{i+1}, \cdots, u_m)=\bu, \quad (v_{i+1}, \cdots, v_n)=\bv. 
\eeqq 
\end{proof}

\begin{cor} \label{cor:rulea}
Define
\beq
    \dc^{m,n}_i:= i! \binom{m}{i}\binom{n}{i} .
\eeq
\begin{enumerate}[(a)]
\item 
If $\bsigma=\cdots \alpha^m\beta^n\cdots$ with $\{\alpha,\beta\}\notin \{ \{1, 2\}, \{2, 3\}, \{12, 3\}, \{1, 23\} \}$, then $\ing^{\bsigma}_{\btau} = \ing^{\bsigma'}_{\btau}$ where 
$\bsigma'= \cdots \beta^{m} \alpha^{n}\cdots$.

\item  
If $\bsigma=\cdots \alpha^m\beta^n\cdots$ with $\{\alpha,\beta\}\in \{ \{1, 2\}, \{2, 3\}, \{12, 3\}, \{1, 23\} \}$, then\footnote{For example, if $\bsigma=\cdots 3^m (12)^n\cdots$, then $\bsigma_i= \cdots (12)^{n-i} (123)^i 3^{m-i}\cdots$.} 
\beqq \begin{split}
    \ing^{\bsigma}_{\btau} 
    =  \sum_{i=0}^{m\wedge n} (-1)^{\#} \dc^{m,n}_i \ing^{\bsigma_i}_{\btau} , \qquad
    \bsigma_i= \cdots \beta^{n-i} (\alpha\beta)^i \alpha^{m-i}\cdots.
\end{split} \eeqq

\item 
If $\btau=\cdots \alpha^m\beta^n\cdots$ with $\{\alpha,\beta\}\notin \{ \{1, 2\}, \{2, 3\}, \{12, 3\}, \{1, 23\} \}$, then 
$\ing^{\bsigma}_{\btau} = \ing^{\bsigma}_{\btau'}$ where $\btau'= \cdots \beta^{m} \alpha^{n}\cdots$.

\item  
If $\btau=\cdots \alpha^m\beta^n\cdots$ with $\{\alpha,\beta\}\in \{ \{1, 2\}, \{2, 3\}, \{12, 3\}, \{1, 23\} \}$, then 
\beqq \begin{split}
    \ing^{\bsigma}_{\btau} 
    =  \sum_{i=0}^{m\wedge n} (-1)^{\#} \dc^{m,n}_i \ing^{\bsigma}_{\btau_i} , \qquad
    \btau_i= \cdots \beta^{n-i} (\alpha\beta)^i \alpha^{m-i}\cdots.
\end{split} \eeqq
\end{enumerate}
\end{cor}

\begin{proof}
(a) From the formula \eqref{eq:Pisigmaxi}, $\Pi^{\bsigma}_{\btau}(\bsxi, \bseta)$ as a function of $\bsxi$ can have a pole at $\xi^{\alpha}_i=\xi^{\beta}_j$ only when $\{\alpha,\beta\}= \{1, 2\}, \{2, 3\}, \{12, 3\}$, or $\{1, 23\}$.
Therefore, (a) holds because $\Pi^{\bsigma}_{\btau}$ is analytic at $\xi^{\alpha}_i = \xi^{\beta}_j$ for all $i, j$ when ${\alpha, \beta}$ does not belong to these sets.

(b) From \eqref{eq:ingdef}, 
\beqq
    \ing^{\bsigma}_{\btau}
    =  \frac{1}{(2\pi \ii)^{|\bsigma|+|\btau|}} \int \dd \bseta^{\btau}  \int \dd \bsxi^{\bsigma} \, 
    \Pi^{\bsigma}_{\btau}(\bsxi, \bseta) \FF_{L}^{\bsigma|\btau} (\bsxi, \bseta) .
\eeqq 
Since $\bsigma=\cdots \alpha^m\beta^n\cdots$, the variable $\bsxi$ is of the form 
$\bsxi=\cdots \xi^\alpha_{a+1}\cdots \xi^\alpha_{a+m}\xi^\beta_{b+1}\cdots \xi^\beta_{b+n}\cdots$ for some nonnegative integers $a$ and $b$. 
The contours for $\xi^\alpha_{a+1},\ldots,\xi^\alpha_{a+m}$ can all be taken to be the same circle $\gamma_\alpha$, while the contours for $\xi^\beta_{b+1},\ldots,\xi^\beta_{b+n}$ can all be taken to be the same circle $\gamma_\beta$, such that both circles are centered at $-1$, $\gamma_\beta$ encloses $\gamma_\alpha$, and there exists an annulus $\Omega$ containing both $\gamma_\alpha$ and $\gamma_\beta$ but none of the other contours. 
Thus, Corollary \ref{lem:deformbasic} applies with $\bu=(\xi^\alpha_{a+1},\ldots,\xi^\alpha_{a+m})$, $\bv=(\xi^\beta_{b+1},\ldots,\xi^\beta_{b+n})$, $g=\ff_{L,\alpha}$, and $h=\ff_{L,\beta}$, while all other variables are held fixed and $\bw=(\xi^{(\alpha\beta)}_{c+1},\ldots,\xi^{(\alpha\beta)}_{c+m\wedge n})$ for some nonnegative integer $c$. 
Since $\ff_{L,\alpha}(z)\ff_{L,\beta}(z)=\ff_{L,(\alpha\beta)}(z)$, we see that $\FF_{L}^{\bsigma|\btau}(\bsxi,\bseta)$ becomes $\FF_{L}^{\bsigma_i|\btau}(\bsxi,\bseta)$ upon computing $R_i$ in \eqref{eq:RR}. 
Thus, the result follows.

The arguments for (c) and (d) are similar.
\end{proof}

We now express the integrals in Proposition \ref{cor:decomposition} as sums of integrals to which Proposition \ref{cor:incor} applies. 
In doing so, we obtain the following two results, which are applicable to a broader class of integrals.

\begin{lem} \label{lem:integralordering}
Let $\bn =(n_1, n_2, n_3)\in \N^3$ and $\bsigma = 3^{n_{31}}2^{n_{21}}1^{n_{1}}2^{n_{22}}3^{n_{32}}$ with $n_{21}+n_{22}=n_2$ and $n_{31}+n_{32}=n_3$. 
Then the following hold for every $\btau \in \listn$. Here, $\#$ denotes an integer whose exact value is not specified.\footnote{It is possible to determine the value of $\#$. For example, $\#=i(1+b_1)+j+k(i-j) + (j+k)(b_3+n_2-n_3)$ for (a), if $\type(\btau)= (b_{123}, b_{12}, b_{23}, b_1, b_2, b_3)$.}
\begin{enumerate}[(a)]
\item We have 
\beqq
    \ing^{\bsigma}_{\btau} 
    =  \sum_{i=0}^{n_1 \wedge n_{22}} \sum_{j=0}^{i \wedge n_{32}}  \sum_{k=0}^{(n_2-i) \wedge n_{31}} 
    (-1)^{\#} \dc^{n_{1}, n_{22}}_i \dc^{i,n_{32}}_j \dc^{n_2-i,n_{31}}_{k}
    \, \ing^{\bsigma_{ijk}}_{\tau},
\eeqq
where  $\bsigma_{ijk} =2^{n_{2}-i-k}(23)^k3^{n_{3}-j-k}(123)^j(12)^{i-j}1^{n_{1}-i}$. 
   
\item We have  
\beqq
    \ing^{\bsigma}_{\btau} 
    =  \sum_{i=0}^{n_1 \wedge n_{22}}    \sum_{j=0}^{i \wedge n_{32}} \sum_{k=0}^{(n_2-i) \wedge (n_{32}-j)}
    (-1)^{\#} \dc^{n_1, n_{22}}_i \dc^{i,n_{32}}_j \dc^{n_2-i, n_{32}-j}_k 
    \, \ing^{\bsigma_{ijk}}_{\tau},
\eeqq
where $\bsigma_{ijk} =3^{n_{3}-j-k}(23)^k2^{n_{2}-i-k}(123)^j(12)^{i-j}1^{n_{1}-i}$. 

\item We have 
\beqq
    \ing^{\bsigma}_{\btau} 
    = \sum_{i=0}^{n_{22} \wedge n_{32}}  \sum_{j=0}^{n_{21} \wedge (n_{32}-i)}  \sum_{k=0}^{n_{1} \wedge i} \sum_{l=0}^{(n_{21}-j) \wedge (n_1-k)}
    (-1)^{\#} 
    \dc^{n_{22}, n_{32}}_i \dc^{n_{21}, n_{32}-i}_j \dc^{n_1, i}_k \dc^{n_{21}-j, n_1-k}_l
    \, \ing^{\bsigma_{ijkl}}_{\tau}, 
\eeqq
where  $\bsigma_{ijkl} =  3^{n_{3}-i-j}(23)^{i+j-k} (123)^{k} 1^{n_{1}-k-l} (12)^l 2^{n_{2}-i-j-l}$. 
\end{enumerate}
\end{lem}

\begin{proof}
We repeatedly use Corollary \ref{cor:rulea} to yield the following. 

(a) First,
\beq \label{eq:23_sigmai}
    \ing^{\bsigma}_{\btau} 
    = \sum_{i=0}^{n_1 \wedge n_{22}} (-1)^{\#_i} \dc^{n_{1}, n_{22}}_i  \ing^{\bsigma_{i}}_{\btau}, 
    \qquad
    \bsigma_{i} = 3^{n_{31}}2^{n_{21}}2^{n_{22}-i}(12)^{i}1^{n_{1}-i}3^{n_{32}} 
    =  3^{n_{31}}2^{n_{2}-i}(12)^{i}1^{n_{1}-i}3^{n_{32}} .
\eeq
Moreover, $\ing^{\bsigma_i}_{\btau}= \ing^{\hat{\bsigma}_i}_{\btau}$ where $\hat{\bsigma}_i= 3^{n_{31}}2^{n_2-i}(12)^{i} 3^{n_{32}}1^{n_{1}-i}$. We further obtain  
\beq \label{eq:23_sigmaij}
    \ing^{\bsigma_i}_{\btau} = \ing^{\hat{\bsigma}_i}_{\btau}
    =  \sum_{j=0}^{i \wedge n_{32}} (-1)^{\#_j} \dc_j^{i, n_{32}} \ing^{\bsigma_{ij}}_{\tau}, 
    \qquad
    \bsigma_{ij} =3^{n_{31}}2^{n_{2}-i}3^{n_{32}-j}(123)^j(12)^{i-j}1^{n_{1}-i}, 
\eeq
where 
\beq \label{eq:23_sigmaijk}
    \ing^{\bsigma_{ij}}_{\btau} = \sum_{k=0}^{(n_2-i) \wedge n_{31}} (-1)^{\#_k} \dc_k^{n_2-i, n_{31}} \ing^{\bsigma_{ijk}}_{\btau}, 
    \quad
    \bsigma_{ijk} =2^{n_{2}-i-k}(23)^k3^{n_{3}-j-k}(123)^j(12)^{i-j}1^{n_{1}-i} .
\eeq
Thus, we obtain (a).

(b) Instead of \eqref{eq:23_sigmaijk}, we write 
\beqq
    \ing^{\bsigma_{ij}}_{\btau} = \sum_{k=0}^{(n_2-i) \wedge (n_{32}-j)} (-1)^{\#} 
    \dc_k^{n_2-i, n_{32}-j} \ing^{\bsigma_{ijk}}_{\tau}	,
    \quad
    \bsigma_{ijk} =3^{n_{31}} 3^{n_{32}-j-k}(23)^k2^{n_{2}-i-k}(123)^j(12)^{i-j}1^{n_{1}-i}. 
\eeqq

(c) We find that
\beqq \label{eq:rr1}
    \ing^{\bsigma}_{\btau} = \sum_{i=0}^{n_{22} \wedge n_{32}} (-1)^{\#} \dc_i^{n_{22}, n_{32}} \ing^{\bsigma_{i}}_{\tau}, 
    \qquad \bsigma_{i} = 3^{n_{31}}2^{n_{21}}1^{n_{1}}3^{n_{32}-i}(23)^{i}2^{n_{22}-i}
\eeqq
and $\ing^{\bsigma_{i}}_{\tau}= \ing^{\hat\bsigma_{i}}_{\tau}$ where $\hat\bsigma_{i}=3^{n_{31}}2^{n_{21}}3^{n_{32}-i}1^{n_{1}}(23)^{i}2^{n_{22}-i}$. We further find that 
\beqq \label{eq:rr2}
    \ing^{\bsigma_{i}}_{\tau}= \ing^{\hat\bsigma_{i}}_{\tau}= \sum_{j=0}^{n_{21} \wedge (n_{32}-i)} 
    (-1)^{\#} \dc_j^{n_{21}, n_{32}-i}  \ing^{\bsigma_{ij}}_{\tau}, 
    \qquad
    \bsigma_{ij} =  3^{n_{3}-i-j}(23)^{j}2^{n_{21}-j}1^{n_{1}}(23)^{i}2^{n_{22}-i}. 
\eeqq
Additionally, 
\beqq \label{eq:rr3}
    \ing^{\bsigma_{ij}}_{\btau} = \sum_{k=0}^{n_{1} \wedge i} (-1)^{\#} \dc_{k}^{n_1, i} \ing^{\bsigma_{ijk}}_{\tau}, 
    \qquad
    \bsigma_{ijk} =  3^{n_{3}-i-j}(23)^{j}2^{n_{21}-j}(23)^{i-k}(123)^{k} 1^{n_{1}-k} 2^{n_{22}-i}
\eeqq
where $\ing^{\bsigma_{ijk}}_{\tau}= \ing^{\hat\bsigma_{ijk}}_{\tau}$, and  
$\hat\bsigma_{ijk} = 3^{n_{3}-i-j}(23)^{i+j-k} (123)^{k} 2^{n_{21}-j} 1^{n_{1}-k} 2^{n_{22}-i}$. 
Finally, 
\beqq \label{eq:rr4}
    \ing^{\bsigma_{ijk}}_{\btau} = \sum_{l=0}^{(n_{21}-j) \wedge (n_1-k)} (-1)^{\#} \dc_l^{n_{21}-j, n_1-k} 
    \ing^{\bsigma_{ijkl}}_{\tau}, \quad
    \bsigma_{ijkl} =  3^{n_{3}-i-j}(23)^{i+j-k} (123)^{k} 1^{n_{1}-k-l} (12)^l 2^{n_{2}-i-j-l}.
\eeqq
\end{proof}

\begin{lem} \label{lem:integralordering2}
Let $\bn \in \N^3$ and $\btau = 3^{n_{31}}2^{n_{21}}1^{n_{1}}2^{n_{22}}3^{n_{32}}$, where $n_{21}+n_{22}=n_2$ and $n_{31}+n_{32}=n_3$. 
Then, for every $\bsigma \in \listn$, 
\beqq
    \ing^{\bsigma}_{\btau} 
    =  \sum_{i=0}^{n_1 \wedge n_{22}} 
    \sum_{j=0}^{ n_{31} \wedge (n_2-i)} (-1)^{\#}
    \dc^{n_{1}, n_{22}}_i \dc^{n_2-i,n_{31}}_{j}
    \, \ing^{\bsigma}_{\btau_{ij}}
\eeqq
for some $\#\in \Z$, where $\btau_{ij} = 2^{n_{2}-i-j}(23)^{j}3^{n_{31}-j}(12)^{i}1^{n_{1}-i}3^{n_{32}}$. 
\end{lem}

\begin{proof}
Corollary \ref{cor:rulea} implies 
\beqq \label{eq:167_taui}
    \ing^{\bsigma}_{\btau} = \sum_{i=0}^{n_1 \wedge n_{22}} (-1)^{\#} \dc_i^{n_1, n_{22}} \ing^{\bsigma}_{\btau_i}, 
    \qquad
    \btau_{i} = 3^{n_{31}}2^{n_{2}-i}(12)^{i}1^{n_{1}-i}3^{n_{32}}, 
\eeqq
and that 
\beqq \label{eq:167_tauij}
    \ing^{\bsigma}_{\btau_i} = \sum_{j=0}^{n_{31} \wedge (n_2-i) } (-1)^{\#} \dc_j^{n_{31}, n_2-i}  \ing^{\bsigma}_{\btau_{ij}} ,
    \qquad
    \btau_{ij} =2^{n_{2}-i-j}(23)^{j}3^{n_{31}-j}(12)^{i}1^{n_{1}-i}3^{n_{32}}.
\eeqq
\end{proof}

\medskip

The above lemmas yield the proofs of Lemma \ref{lem:QQ111list} and \ref{lem:QQ121list}.

\begin{proof}[Proofs of Lemmas \ref{lem:QQ111list} and \ref{lem:QQ121list}] 
These results follow from Lemmas \ref{lem:integralordering} and \ref{lem:integralordering2} upon careful tracking of the signs $(-1)^{\#}$.
It is straightforward to verify the signs explicitly, and the results follow.
\end{proof}

\medskip

We now focus on the proof of Proposition \ref{cor:decomposition}.
The proof uses the following estimates.

\begin{lem} \label{lem:sumsss}
The following estimates hold for every $n, n'\in \N$: 
\begin{enumerate}[(a)]
\item $\displaystyle \sum_{0\le a\le n \wedge n'} a! \sqrt{(n+n'-a)!} \le \frac{2^{2(n+n')} n!n'!}{\sqrt{(n\vee n'-n\wedge n')!}}$.
\item $\displaystyle\sum_{0\le a+b \le n \wedge n'} a!b! \sqrt{(n+n'-a-b)!}\le  \frac{2^{3(n+n')} n!n'!}{\sqrt{(n\vee n'-n\wedge n')!}}$.
\item $\displaystyle\sum_{0\le  a+b, c \le n \wedge n'} a!b!c! \sqrt{(n+n'-a-b-c)!} \le  \frac{2^{4(n+n')} n!n'!}{\sqrt{(n\vee n'-n\wedge n')!}}$. 
\end{enumerate} 
\end{lem}

\begin{proof}
Since the multinomial coefficients $\frac{k!}{a!(k-a)!}, \frac{k!}{a!b!(k-a-b)!}, \frac{k!}{a!b!c!(k-a-b-c)!}$ are each at least $1$, we find that 
\beqq \begin{split}
    &(a) = \sum_{0\le a\le n \wedge n'} \frac{a! (n+n'-a)!}{ \sqrt{(n+n'-a)!}}
    \le \frac{(n+n')!}{\sqrt{(n+n'-n\wedge n')!}} (n\wedge n'+1) \\
    &(b) = \sum_{0\le a+b \le n \wedge n'} 
    \frac{a!b! (n+n'-a-b)!}{\sqrt{(n+n'-a-b)!}} 
    \le \frac{(n+n')!}{\sqrt{(n+n'-n\wedge n')!}} (n\wedge n'+1)^2 \\
    &(c) = \sum_{0\le  a+b, c \le n \wedge n'} \frac{a!b!c! (n+n'-a-b-c)!}{ \sqrt{(n+n'-a-b-c)!}}
    \le \frac{(n+n')!}{\sqrt{(n+n'-2(n\wedge n'))!}} (n\wedge n'+1)^3. 
\end{split} \eeqq 
Since $\binom{n+n'}{n}\le 2^{n+n'}$, it follows that $(n+n')!\le 2^{n+n'}n! n'!$. 
Additionally, $n\wedge n'+1 \le n+n' \le 2^{n+n'}$ and 
$n+n'-n\wedge n'\ge  n\vee n'-n\wedge n'$. Putting these together we obtain the result. 
\end{proof}

We are now ready to prove Proposition~\ref{cor:decomposition}.

\begin{proof}[Proof of Proposition \ref{cor:decomposition}] 
We use Lemma \ref{lem:integralordering}, Lemma \ref{lem:integralordering2}, Proposition \ref{cor:incor}, 
and Lemma \ref{lem:sumsss}.
Let $L_0, C$, and $c$ be the positive constants appearing in Proposition \ref{cor:incor}. 
Note that $\dc^{m,n}_i\le 2^{m+n}i!$. 

Suppose $\w\in \rgo\cup \rgt$. 
From Lemma \ref{lem:integralordering} (a), we have 
\beqq
    |\ing^{\bsigma}_{\btau} |\le 
    2^{2|\vecn|}
    \sum_{i=0}^{n_1 \wedge n_{22}} \sum_{j=0}^{i \wedge n_{32}}  \sum_{k=0}^{(n_2-i) \wedge n_{31}} 
    i!j!k! \, |\ing^{\bsigma_{ijk}}_{\tau}|, 
    \quad \bsigma_{ijk} =2^{n_{2}-i-k}(23)^k3^{n_{3}-j-k}(123)^j(12)^{i-j}1^{n_{1}-i}. 
\eeqq
If $\type (\bsigma_{ijk})=(j, i-j, k, n_1-i, n_2-i-k, n_3-j-k)= (1,0,0,0,0,0)$, then it necessarily follows that $n_1=n_2=n_3=1$ (and $i=j=1$ and $k=0$), which contradicts the assumption that $\vecn\neq (1,1,1)$. 
Hence, Proposition \ref{cor:incor} (a) applies to each $\ing^{\bsigma_{ijk}}_{\tau}$, and we find that
\beqq
    |\ing^{\bsigma}_{\btau} |\le 
    (2^2C)^{|\vecn|} e^{- cL} \cK_L \sqrt{n_1!n_3!} 
    \sum_{i=0}^{n_1 \wedge n_{22}} \sum_{j=0}^{i \wedge n_{32}}  \sum_{k=0}^{(n_2-i) \wedge n_{31}} 
    i!j!k! \sqrt{(n_1+ n_2-i)! (n_2+n_3-j-k)!}. 
\eeqq
Note that $j+k\le i+ (n_2-i)= n_2$ and $j+k\le n_{32}+n_{31}=n_3$. Hence, the triple sum is bounded by 
\beqq
    \sum_{i=0}^{n_1 \wedge n_{2}} \sum_{0\le  j+k\le n_2 \wedge n_3}
    i!j!k! \sqrt{(n_1+ n_2-i)! (n_2+n_3-j-k)!}. 
\eeqq
Applying Lemma \ref{lem:sumsss} (a) and (b), and possibly adjusting the value of the constant $C$, we obtain the desired result \eqref{eq:decap}.

Suppose $\w\in  \rgth\cup \rgf$. From Lemma \ref{lem:integralordering} (b), we have
\beqq
    |\ing^{\bsigma}_{\btau} | 
    \le 2^{2|\vecn|} \sum_{i=0}^{n_1 \wedge n_{22}}    \sum_{j=0}^{i \wedge n_{32}} \sum_{k=0}^{(n_2-i) \wedge (n_{32}-j)}
    i!j!k!
    \, | \ing^{\bsigma_{ijk}}_{\tau} |, 
    \quad \bsigma_{ijk} =3^{n_{3}-j-k}(23)^k2^{n_{2}-i-k}(123)^j(12)^{i-j}1^{n_{1}-i}.
\eeqq
Again, $\type (\bsigma_{ijk})=(j, i-j, k, n_1-i, n_2-i-k, n_3-j-k)$ is not equal to $(1,0,0,0,0,0)$, and thus Proposition \ref{cor:incor} (b) applies to each $\ing^{\bsigma_{ijk}}_{\tau}$. 
Noting that $j+k\le i+ (n_2-i)= n_2$ and  $j+ k\le j+ (n_{32}-j)\le n_3$, we find that
\beqq
    |\ing^{\bsigma}_{\btau} |\le 
    (2^2C)^{|\vecn|} e^{- cL} \cK_L \sqrt{n_1!n_3!} 
    \sum_{i=0}^{n_1 \wedge n_{2}} \sum_{0\le j+k\le n_2 \wedge n_3}
    i!j!k! \sqrt{(n_1+ n_2-i)! (n_2+n_3-j-k)!}. 
\eeqq
The result\eqref{eq:decap} then follows from Lemma\ref{lem:sumsss} (a) and (b).

Suppose $\w\in \rgfi$. From Lemma \ref{lem:integralordering} (c), we have
\beqq
    |\ing^{\bsigma}_{\btau} | 
    \le 2^{2|\vecn|} \sum_{i=0}^{n_{22} \wedge n_{32}}  \sum_{j=0}^{n_{21} \wedge (n_{32}-i)}  \sum_{k=0}^{n_{1} \wedge i} \sum_{l=0}^{(n_{21}-j) \wedge (n_1-k)} 
    i!j!k!l! 	
    \, | \ing^{\bsigma_{ijkl}}_{\tau} |
\eeqq
where 
\beqq
    \bsigma_{ijkl} =  3^{n_{3}-i-j}(23)^{i+j-k} (123)^{k} 1^{n_{1}-k-l} (12)^l 2^{n_{2}-i-j-l}.
\eeqq 
We again see that $\veca:= \type (\bsigma_{ijkl})=(k, l, i+j-k, n_1-k-l, n_2-i-j-l, n_3-i-j)$ is not $(1,0,0,0,0,0)$ since $\bn\neq (1,1,1)$. 
Moreover, $a_{123}+a_1=k+ (n_1-k-l)= n_1-l\ge n_1- (n_{21}-j)\ge n_1-n_{21}\ge 1$ where the final inequality uses the assumption $n_1 \geq n_{21} + 1$.  
Hence, Proposition \ref{cor:incor} (c) applies to $\ing^{\bsigma_{ijkl}}_{\tau}$. 
Noting that $i+j\le n_2\wedge n_3$ and $k+l\le n_1\wedge n_2$, 
\beqq
    |\ing^{\bsigma}_{\btau} |\le 
    (2^2C)^{|\vecn|} e^{- cL} \cK_L \sqrt{n_1!n_3!} 
    \sum_{\substack{i,j=0\\ i+j\le n_2 \wedge n_3}}^{n_2 \wedge n_{3}}  
    \sum_{\substack{k,l=0\\ k+l\le n_1 \wedge n_2}}^{n_1 \wedge n_{2}} 
    i!j!k! l! \sqrt{(n_1+ n_2-k-l)! (n_2+n_3-i-j)!}.
\eeqq
We obtain the result \eqref{eq:decap} from Lemma \ref{lem:sumsss} (b). 

Suppose $\w\in  \rgs\cup \rgse$. 
From Lemma \ref{lem:integralordering} (c) and Lemma \ref{lem:integralordering2}, we have 
\beqq \begin{split} 
    |\ing^{\bsigma}_{\btau} |\le 
    &2 ^{4|\vecn|} \sum_{i=0}^{n_{22} \wedge n_{32}}  \sum_{j=0}^{n_{21} \wedge (n_{32}-i)}  \sum_{k=0}^{n_{1} \wedge i} \sum_{l=0}^{(n_{21}-j) \wedge (n_1-k)} \sum_{p=0}^{n_1 \wedge n'_{22}} 
    \sum_{q=0}^{(n_2-p) \wedge n'_{31}}  i!j!k!l!p!q!
    \, | \ing^{\bsigma_{ijkl}}_{\tau_{pq}} |
\end{split} \eeqq
where     
\beqq \begin{split}
    &\bsigma_{ijkl} =  3^{n_{3}-i-j}(23)^{i+j-k} (123)^{k} 1^{n_{1}-k-l} (12)^l 2^{n_{2}-i-j-l} ,\qquad \btau_{pq} = 2^{n_{2}-p-q}(23)^{q}3^{n'_{31}-q}(12)^{p}1^{n_{1}-p}3^{n'_{32}}.
\end{split} \eeqq
Let $\veca=\type(\bsigma_{ijkl})$ and $\vecb=\type(\btau_{pq})$. 
If $(a_{123}, a_{12}, a_{23}, a_1, b_2, a_3)= (k, l, i+j-k, n_1-k-l, n_2-p-q, n_3-i-j)= (1,0,0,0,0,0)$, 
then $n_1=n_3=1$, $n_2=p+q$. 
Since $b_1= n_1-p\ge 0$ and $n'_{31}-q\ge 0$, we find that 
$n_2=p+q\le n_1+n'_{31}\le n_1+n_3=2$. 
Given our assumption that $\vecn\neq (1,1,1), (1,2,1)$, we see that $(a_{123}, a_{12}, a_{23}, a_1, b_2, a_3)\neq (1,0,0,0,0,0)$. 
Furthermore, $a_{123}+a_1=  n_1-l\ge n_1-n_{21}+j\ge n_1-n_{21}\ge 1$ by assumption.  
Hence, Proposition \ref{cor:incor} (d) applies to each $\ing^{\bsigma_{ijkl}}_{\tau_{pq}}$. 
Noting that $i+j\le n_2\wedge n_3$,  $k+l\le n_1\wedge n_2$, $p\le n_1\wedge n_2$, and $q\le n_2\wedge n_3$, we find 
\beqq
    |\ing^{\bsigma}_{\btau} |\le 
    (2^4C)^{|\vecn|} e^{- cL} \cK_L  \sqrt{n_1!n_3!} 
    \sum_{\substack{i,j, q=0\\ i+j\le n_2 \wedge n_3}}^{n_2 \wedge n_{3}}  
    \sum_{\substack{k,l, p=0\\ k+l\le n_1 \wedge n_2}}^{n_1 \wedge n_{2}} 
    i!j!k! l! p! q!\sqrt{(n_1+ n_2-k-l-p)! (n_2+n_3-i-j-q)!}
\eeqq
Applying Lemma \ref{lem:sumsss} (c), we obtain the result \eqref{eq:decap}. 
\end{proof}


\subsection*{Conflicts of Interest}

There are no conflicts of interest relevant to this article. Data sharing is not applicable as no data was involved.


\end{document}